\documentclass[a4paper]{article}

\usepackage{a4wide}
\usepackage{amsmath,amssymb,amsthm,amsfonts,mathrsfs}
\usepackage[UKenglish]{babel}
\usepackage[pdftex]{graphicx}
	\graphicspath{{./}{figures/}}
\usepackage{subfig}
\usepackage{tikz}
\usetikzlibrary{patterns}

\newcommand{\patP}{crosshatch}
\newcommand{\patPmezzi}{Hlines}
\newcommand{\patPquarti}{Dlines}
\newcommand{\patUnoMenoP}{north east lines}
\newcommand{\patUnoMenoPmezzi}{crosshatch dots}
\newcommand{\patUnoMenoPquarti}{Vlines}
\newcommand{\patUno}{}

\newcommand{\skeleton}{
\draw[help lines,dashed] (0,0) -- (3.1,3.1);
\draw[help lines,<->] (2.02,0.1) -- (3,0.1)
      node[pos=0.5, above] {$r$};
\draw[very thick] (0.1,-0.03) -- (-0.03,-0.03) -- 
             (-0.03,3.13) -- (0.1,3.13);
\draw[very thick] (3.1-0.1,-0.03) -- (3.1+0.03,-0.03) -- 
             (3.1+0.03,3.13) -- (3.1-0.1,3.13);
\foreach \k in {1,2,3}
{
  \draw[help lines,xshift=\k cm] (0,0) -- (0,3.1);
  \draw[help lines,yshift=-\k cm] (0,0) -- (3.1,0);
}
}

\newcommand{\skeletonB}{
\draw[help lines,dashed] (0,0) -- (2.1,2.1);
\draw[help lines,<->] (0.02,0.1) -- (1,0.1)
      node[pos=0.5, above] {$r$};
\draw[very thick] (0.1,-0.03) -- (-0.03,-0.03) -- 
             (-0.03,3.13) -- (0.1,3.13);
\draw[very thick] (2.1-0.1,-0.03) -- (2.1+0.03,-0.03) -- 
             (2.1+0.03,3.13) -- (2.1-0.1,3.13);
\foreach \k in {1,2}
{
  \draw[help lines,xshift=\k cm] (0,0) -- (0,3.1);
}
\foreach \k in {1,2,3}
{
  \draw[help lines,yshift=-\k cm] (0,0) -- (2.1,0);
}
}

\newcommand{\skeletonC}{
\draw[help lines,dashed] (0,0) -- (2.1,2.1);
\draw[help lines,<->] (1.02,0.1) -- (2,0.1)
      node[pos=0.5, above] {$r$};
\draw[very thick] (0.1,-0.03) -- (-0.03,-0.03) -- 
             (-0.03,2.13) -- (0.1,2.13);
\draw[very thick] (3.1-0.1,-0.03) -- (3.1+0.03,-0.03) -- 
             (3.1+0.03,2.13) -- (3.1-0.1,2.13);
\foreach \k in {1,2,3}
{
  \draw[help lines,xshift=\k cm] (0,0) -- (0,2.1);
}
\foreach \k in {1,2}
{
  \draw[help lines,yshift=-\k cm] (0,0) -- (3.1,0);
}
}

\newcommand{\p}{\mathsf{p}}
\newcommand{\q}{\mathsf{q}}
\newcommand{\vm}{V_{\max}}
\newcommand{\vmp}{V_{\max}^{\p}}
\newcommand{\vmq}{V_{\max}^{\q}}
\newcommand{\vb}{v_*}
\newcommand{\va}{v^*}
\newcommand{\fp}{f^{\p}}
\newcommand{\fq}{f^{\q}}

\newcommand{\fbp}{f_*^{\p}}
\newcommand{\fap}{f^{\p*}}
\newcommand{\fbq}{f_*^{\q}}
\newcommand{\faq}{f^{\q*}}
\newcommand{\fu}{f^{1}}
\newcommand{\fd}{f^{2}}
\newcommand{\fub}{\f^{1}}
\newcommand{\fdb}{\f^{2}}
\newcommand{\ru}{\rho^{1}}
\newcommand{\rd}{\rho^{2}}
\newcommand{\nd}{n^{2}}
\newcommand{\Vp}{\mathcal{V}^{\p}}
\newcommand{\Vq}{\mathcal{V}^{\q}}
\newcommand{\Qp}{Q^{\p}}
\newcommand{\Qpp}{Q^{\p\p}}
\newcommand{\Qpq}{Q^{\p\q}}
\newcommand{\Ap}{A^{\p}(\vb\to v|\va;s)}
\newcommand{\Apq}{A^{\p\q}(\vb\to v|\va;s)}
\newcommand{\Dv}{\Delta v}
\newcommand{\dv}{\delta v}
\newcommand{\np}{n^\p}
\newcommand{\nq}{n^\q}
\newcommand{\A}{\mathbf{A}}
\newcommand{\B}{\mathbf{B}}
\newcommand{\uu}{\mathbf{u}}
\newcommand{\vv}{\mathbf{v}}

\newcommand{\BB}{\mathcal{B}}
\newcommand{\ds}{\displaystyle}

\newcommand{\scr}{s_{\mathsf{cr}}}
\newcommand{\abs}[1]{\left\vert#1\right\vert}
\newcommand{\unit}[1]{\textup{#1}}
\newcommand{\f}{\mathbf{f}}

\newcommand{\norm}[1]{\left\lVert#1\right\rVert}
\newcommand{\Cf}{C_{\sf f}}
\newcommand{\Cs}{C_{\sf s}}

\newcounter{the}
\theoremstyle{remark}
\newtheorem{theorem}[the]{Theorem}
\newtheorem{remark}[the]{Remark}

\title{Analysis of a multi-population kinetic model for traffic flow}
\author{Gabriella Puppo \\
		{\small\it Department of Science and High Technology} \\
		{\small\it University of Insubria} \\
		{\small\it Via Valleggio 11, 22100 Como, Italy} \\[5mm]
	Matteo Semplice \\
		{\small\it Department of Mathematics ``G. Peano''} \\
		{\small\it University of Turin} \\
		{\small\it Via Carlo Alberto 10, 10123 Turin, Italy} \\[5mm]
	Andrea Tosin \\
		{\small\it Istituto per le Applicazioni del Calcolo ``M. Picone''} \\
		{\small\it Consiglio Nazionale delle Ricerche} \\
		{\small\it Via dei Taurini 19, 00185 Rome, Italy} \\[5mm]
	Giuseppe Visconti \\
		{\small\it Department of Science and High Technology} \\
		{\small\it University of Insubria} \\
		{\small\it Via Valleggio 11, 22100 Como, Italy}
	   }
\date{}

\begin{document}

\pgfdeclarepatternformonly{Hlines}
  {\pgfpointorigin}
  {\pgfpoint{.03cm}{.03cm}}
  {\pgfpoint{.03cm}{.03cm}}
  {
  \pgfline{\pgfpointorigin}{\pgfpoint{.03cm}{0cm}}
  \pgfusepath{stroke}
 }

\pgfdeclarepatternformonly{Vlines}
  {\pgfpointorigin}
  {\pgfpoint{.03cm}{.03cm}}
  {\pgfpoint{.03cm}{.03cm}}
  {
  \pgfline{\pgfpointorigin}{\pgfpoint{0cm}{.03cm}}
  \pgfusepath{stroke}
 }

\pgfdeclarepatternformonly{Dlines}
  {\pgfpointorigin}
  {\pgfpoint{.05cm}{.05cm}}
  {\pgfpoint{.05cm}{.05cm}}
  {
  \pgfline{\pgfpointorigin}{\pgfpoint{.05cm}{.05cm}}
  \pgfusepath{stroke}
 }
 
\maketitle

\begin{abstract}
In this work we extend a recent kinetic traffic model~\cite{PgSmTaVg2} to the case of more than one class of vehicles, each of which is characterized by few different microscopic features. We consider a Boltzmann-like framework with only binary interactions, which take place among vehicles belonging to the various classes. Our approach differs from the multi-population kinetic model proposed in~\cite{PgSmTaVg} because here we assume continuous velocity spaces and we introduce a parameter describing the physical velocity jump performed by a vehicle that increases its speed after an interaction. The model is discretized in order to investigate numerically the structure of the resulting fundamental diagrams and the system of equations is analyzed by studying well posedness. Moreover, we compute the equilibria of the discretized model and we show that the exact asymptotic kinetic distributions can be obtained with a small number of velocities in the grid. Finally, we introduce a new probability law in order to attenuate the sharp capacity drop occurring in the diagrams of traffic.
\end{abstract}

\paragraph{Keywords} Multispecies models, Boltzmann-like kinetic models, discrete velocity models, multivalued diagrams, two-phases diagrams

\paragraph{MSC} 35Q20, 65Z05, 90B20

\section{Introduction}
\label{sec:introduction}

In~\cite{PgSmTaVg} we introduced a Boltzmann-like kinetic model for traffic flow, which draws inspiration from the ideas presented in~\cite{benzoni2003EJAM} for macroscopic models, in order to take into account the heterogeneous composition of the flow of vehicles on the road. This aspect, which is rather neglected in the literature, is important to obtain a richer description of the macroscopic behavior of traffic flow. In fact, we showed that the model is able to recover the whole structure of the diagrams relating the macroscopic flux and speed to the vehicle density in homogeneous space conditions and which represents a basic tool to study traffic problems.

In particular, the multi-population model~\cite{PgSmTaVg} provides multivalued diagrams reproducing two-phases of traffic which are computed from moments of equilibrium solutions of the kinetic equations. As a matter of fact, any kinetic model does not require an a priori link between the local density of traffic and the macroscopic speed, as it is done in standard macroscopic traffic models, see e.g. the classical works~\cite{lighthill1955PRSL,richards1956OR}, or the reviews~\cite{piccoli2009ENCYCLOPEDIA,Rosini}.
The diagrams provided by the model are characterized by a sharp phase transition between the free and the congested phases of traffic, with a capacity drop across the phase transition. Moreover, at low densities, there is a small dispersion of the flux values, which increase nearly linearly with respect to the density. Instead, at high densities, the model reproduces naturally the scattered data usually observed in the experimental diagrams of traffic by taking into account the macroscopic variability of the flux and mean speed at equilibrium due to the heterogeneous composition of the traffic ``mixture''. This result is reached without invoking further elements of microscopic randomness of the system: for example in~\cite{FermoTosin14} the explanation for the phase transition appeals to the stochasticity of the drivers' behavior and to the consequent variability of the microscopic speeds at equilibrium.  In the literature, a variety of multiphase models, also at the macroscopic and at the microscopic scales, have been introduced in order to reflect the features of traffic, for a review see~\cite{klarReview} and references therein.

This work can be seen as a natural sequel of~\cite{PgSmTaVg} because we revisit and refine the model by extending the construction introduced in~\cite{PgSmTaVg2} to the case of traffic mixtures. We will consider more than one class of vehicles, characterized by a few parameters accounting for the microscopic differences which allow one to distinguish two (or more) types of vehicles. Here, these parameters will be the typical length and the maximum speed, and we will introduce a kinetic distribution function for each class of vehicles. We stress that the heterogeneity of traffic is not only described by introducing two or more classes  with different physical features, but also by considering two or more types of drivers with different behavioral attributes according to the maximum velocity they intend to keep, as in~\cite{benzoni2003EJAM,LebacqueGSOM,MendezVelasco13} for macroscopic models.

However, this approach differs from the model proposed in~\cite{PgSmTaVg} because the latter is based on a lattice of admissible microscopic speeds and the output of an interaction depends on the number of velocities chosen in the lattice. Here, instead, using the framework discussed in~\cite{PgSmTaVg2}, we will consider continuous and bounded velocity spaces and we will introduce a parameter proportional to the actual acceleration of a vehicle in order to describe the physical {\em velocity jump} performed by vehicles when they increase their speed as a result of an interaction. Clearly, this parameter may depend on the mechanical characteristics of a vehicle, but we will suppose that it is fixed. This choice preserves the quantized structure of the asymptotic functions, already observed in~\cite{PgSmTaVg2} for the single population model. In fact, Theorem~\ref{th:equilibria} in Section~\ref{sec:analysis} shows that the asymptotic kinetic distribution approaches a combination of delta functions, centered in the velocities which are proportional to the fixed parameter.

The purpose of the kinetic approach is to obtain an aggregate representation of the distribution of vehicles on the road, thus it is not based on single particles as the microscopic ``follow the leader'' models~\cite{ZhangMultiphase}, and it allows one to recover the macroscopic behavior of car flow by means of a detailed modeling of the microscopic interactions. As in any Boltzmann-like kinetic model, here the relaxation of the distribution functions in time is due to the collision operator, which describes the acceleration and the slowing down interactions among vehicles. Notice that in the first kinetic traffic models~\cite{Prigogine61,PrigogineHerman} the collision term accounted only for slowing down interactions treating the acceleration by means of a relaxation term, e.g. depending on the desired speed of the drivers as in~\cite{paveri1975TR}. Instead, later, models with a kinetic description also for the acceleration were developed. In the present framework we will consider only binary interactions and, since there is more than one population of vehicles, we will suppose that the interactions take place also among vehicles belonging to different classes. Therefore, as in standard kinetic models for gas mixtures~\cite{andries2001REPORT,brull2012EJMB,groppi2004PF,hamel1965PF}, the collision term will be split in the sum of many collision terms. Each one will be characterized by a probability density which allows one to assign a post-interaction speed in a non-deterministic way. The probability of gaining or loosing a speed will depend on the local mean traffic conditions. We suppose also that the microscopic interaction rules, prescribing a post-interaction speed, do not depend on the type of vehicle which is considered.

For references on other kinetic approaches studied in the literature, see e.g. \cite{HertyIllner08, hertyillner09} where both the acceleration and the slowing down interactions are modeled with a Vlasov-type relaxation towards a desired speed. Kinetic theory is also used to model multilane traffic flow~\cite{HelbingGreiner,klar1999SIAP-1,klar1999SIAP-2,LoSchiavo2002,BonzaniC08}, flows on networks~\cite{fermotosin2015}, inhomogeneous space problems with non local interactions~\cite{klar1997Enskog}, control problems~\cite{herty2007M2AS} and safety aspects in vehicular traffic~\cite{FregugliaTosin15}. For a review on kinetic traffic models, see~\cite{klar2004BOOKCH}. %While for a review of the derivation of macroscopic traffic models from the microscopic ``follow the leader'' approach and from the kinetic approach, see~ \cite{klarReview} .

The paper is organized as follows. In Section~\ref{sec:multimodel} we introduce the general framework of the continuous-velocity multi-population model, then we discuss the modeling of the probability density and we prove an indifferentiability principle. In Section~\ref{sec:delta-multimodel} we discretize the model in order to find the asymptotic behavior of the distribution functions. Then we analyze the resulting system of ordinary differential equations by studying well posedness and the asymptotic kinetic distribution for the case of two populations. In Section~\ref{sec:funddiag} we show the macroscopic diagrams of traffic provided by the model with three classes of vehicles. We also discuss the impact that different probabilities of achieving the maximum speed in an interaction have on the sharp capacity drop observed at the transition between free and congested traffic flow. In Section~\ref{sec:conclusions} we propose some final comments and perspectives. Finally, we end in Appendix~\ref{app:discrete-terms} with the explicit computation of the terms resulting from the discretization of the collision operators and in Appendix~\ref{app:equilibria} with the analytical expression of the equilibria for the case of two populations.

\section{A multi-population kinetic model}
\label{sec:multimodel}
In this section we present the general form of a kinetic model for vehicular traffic with a new structure accounting for the heterogeneous composition of the traffic flow on the road. Next, we derive a simplified model based on particular choices made on the microscopic interaction rules. This model is a generalization of~\cite{PgSmTaVg2} to the case of a multi-population framework. Therefore, unlike~\cite{PgSmTaVg}, in this work we suppose that each class of vehicles ({\em population}) admits a continuous space of admissible velocities and we introduce a parameter describing the physical acceleration of each vehicle. Our approach differs from standard kinetic models in that we consider more than one kinetic distribution function. Each one refers to a class of vehicles characterized by precise physical features, in this case the maximum speed and the typical length of a vehicle.

We will focus only on the space homogeneous case, because we want to investigate the structure of the collision term, and of the resulting equilibrium distributions which allow one to obtain the fundamental diagrams of traffic. From now on, we adopt a compact notation, which makes use of an index $\p$, to label various quantities referred to the different classes of vehicles. Thus, let
\[
	\fp=\fp(t,v):\mathbb{R}^+ \times \Vp  \to \mathbb{R}^+
\]
be the {\em kinetic distribution function} of the $\p$-th class of vehicles, then $\fp(t,v)dv$ gives the number of vehicles belonging to the $\p$-class with velocity in $[v,v+dv]$ at time $t$. The space $\Vp=[0,\vmp]$ is the domain of the microscopic speeds related to the $\p$-class, where $\vmp$ is the maximum speed which can be reached by the $\p$-vehicles. It may depend on the mechanical characteristics of the vehicles, on imposed speed limits or on the type of drivers, according to the maximum velocity they intend to keep in free road conditions. Thus, the different maximum speeds allow one to model a first microscopic feature which identifies a class of vehicles. Another difference is introduced by considering the typical length $l^\p$ of vehicles which will be used later to define the concept of the total space occupied on the road.

As usual, macroscopic quantities  are obtained as moments of the distribution functions $\fp$ with respect to the velocity $v$: 
\begin{equation}
	\rho^\p(t)=\int_{\Vp} \fp(t,v)dv,\quad q^\p(t)=\int_{\Vp} v\fp(t,v)dv, \quad u^\p(t)=\frac{q^\p(t)}{\rho^\p(t)}
	\label{eq:macro-var}
\end{equation}
where $\rho^\p$ is the {\em density}, i.e. the number of vehicles of the $\p$-class per unit length (typically, kilometers), $q^\p$ and $u^\p$ are the {\em macroscopic flux} of vehicles and the {\em mean speed} of the $\p$-th class, respectively.

Here we consider a Boltzmann-type kinetic model for vehicular traffic, in which the relaxation to equilibrium is due to binary interactions. In the homogeneous case, we can model the evolution of the $\fp$'s in time by means of the following system of equations: 
\begin{equation}
	\partial_t \fp(t,v)=\Qp\left[ \fp , \left( \fp,\fq \right) \right] (t,v), \quad \forall\;\p
	\label{eq:model1}
\end{equation}
where $\Qp\left[ \fp , \left( \fp,\fq \right) \right] (t,v)$ is the {\em collision operator} which accounts for the change of $\fp$ in time due to the microscopic interactions among vehicles. Clearly, a multi-population model has to consider also the interactions taking place between $\p$- and $\q$-vehicles, where $\q$ represents all classes of vehicles which are not $\p$. For this reason, and following an approach frequently used for mixtures of gases in kinetic theory, see e.g.~\cite{brull2012EJMB,groppi2004PF,hamel1965PF}, $\Qp$ can be naturally thought of as a sum of two or more collision operators, one describing the interactions among vehicles belonging to the same class ({\em self-interactions}) and the other ones describing the interactions among vehicles belonging to different classes ({\em cross-interactions}), so that
\[
	\Qp\left[ \fp , \left( \fp,\fq \right) \right] (t,v) = \underbrace{\Qpp\left[ \fp , \fp \right] (t,v)}_{\text{self-interactions}} + \sum_{\q\in\neg\p} \underbrace{\Qpq\left[ \fp , \fq \right] (t,v)}_{\text{cross-interactions}}
\]
For mass conservation to hold the right-hand side of the above expression has to vanish when it is integrated over the space of admissible speeds of the $\p$-th class, and this is verified e.g. if we assume that the collision terms satisfy
\[
	\int_{\Vp} \Qpp[\fp,\fp](t,v)dv=\int_{\Vp} \Qpq[\fp,\fq](t,v)dv=0,
\]
for all $\q\in\neg\p$. In fact, this ensures that, in the space homogeneous case, the density remains constant:
\[
	\frac{d}{dt}\rho^{\p}(t)=\partial_t\int_{\Vp} \fp(t,v)dv=\int_{\Vp}\Qp\left[ \fp , \left( \fp,\fq \right) \right] (t,v) dv=0.
\]

Following the same logic underlying the construction of a classical Boltzmann-like kinetic model, each collision operator is written as a balance of a gain ($G^{\p\p}$ or $G^{\p\q}$) and a loss term that model statistically the interactions which lead to get or to loose the speed $v\in\Vp$:
\begin{subequations}
	\label{eq:collision}
	\begin{align}
		\Qpp[\fp,\fp](t,v)=& \underbrace{\int_{\Vp} \int_{\Vp} \eta^{\p}(\vb,\va) \Ap \fp(t,\vb) \fp(t,\va) d\va d\vb}_{G^{\p\p}[\fp,\fp](t,v)} \label{eq:self-collision}\\
		 & - \fp(t,v)\int_{\Vp} \eta^{\p}(\vb,\va) \fp(t,\va) d\va,  \nonumber \\
		 \Qpq[\fp,\fq](t,v)=&\underbrace{\int_{\Vp} \int_{\Vq} \eta^{\p\q}(\vb,\va) \Apq \fp(t,\vb) \fq(t,\va) d\va d\vb}_{G^{\p\q}[\fp,\fq](t,v)} \label{eq:cross-collision}\\
		 & - \fp(t,v)\int_{\Vq} \eta^{\p\q}(\vb,\va) \fq(t,\va) d\va. \nonumber
	\end{align}
\end{subequations}

Throughout the paper, we will denote by $\vb\in\Vp$ the pre-interaction velocity of the $\p$-vehicle which is likely to change speed after an interaction. Conversely, $\va\in\Vp$ or $\va\in\Vq$ identifies the velocity of other $\p$- or $\q$-vehicles which induce the gain or loss of the speed $v\in\Vp$. In order to shorten formulas, we will use the following traditional shorthand $\fp(t,\vb)=\fbp$, $\fp(t,\va)=\fap$, $\forall \; \p$, and similarly for $\fbq$, $\faq$, $\forall \; \q\in\neg\p$.

In~\eqref{eq:collision}, $\eta^\p(\vb,\va)$ and $\eta^{\p\q}(\vb,\va)$ are the {\em interaction rates} which model the frequency of self- and cross-interactions respectively. Although they may depend on the relative speed of the interacting vehicles, as in~\cite{coscia2007IJNM,KlarWegener96} for a single population case, in~\cite{PgSmTaVg} we found that a constant interaction rate is already sufficient to account for many aspects of the complexity of traffic. In the space homogeneous case, the interaction rates affect only the relaxation time towards equilibrium. Thus, in this paper we will set $\eta^\p(\vb,\va)=\eta^\p$ and $\eta^{\p\q}(\vb,\va)=\eta^{\p\q}$.

Finally, $\Ap$ and $\Apq$ are the probability densities of gaining the speed $v\in\Vp$ in the case of self- and cross-interactions, respectively. More precisely, $\Ap$ ($\Apq$, resp.) gives the probability that a $\p$-vehicle modifies its pre-interaction speed $\vb\in\Vp$ in the speed $v\in\Vp$ when it interacts with a $\p$-vehicle ($\q$-vehicle, resp.) traveling at the speed $\va\in\Vp$ ($\va\in\Vq$, resp.).

We will suppose that these probabilities depend also on the macroscopic traffic conditions (local road congestion) through the {\em fraction of occupied space} on the road:
\begin{equation}
	\label{eq:s}
	0 \leq s=\sum_{\p} l^\p \rho^\p \leq 1.
\end{equation}
Notice that $s$ was already introduced in~\cite{benzoni2003EJAM} for a multi-population macroscopic model and it was also used in~\cite{PgSmTaVg} for a two-population kinetic model based on a discrete-velocity framework. The quantity $\rho^\p$ appearing in~\eqref{eq:s} is defined in~\eqref{eq:macro-var}. We will assume that $\rho^\p\in[0,\rho_{\max}^\p]$ where $\rho_{\max}^\p$ is the maximum density of $\p$-vehicles chosen as $\frac{1}{l^\p}$, i.e. as the maximum number of vehicles per unit length in bumper-to-bumper conditions when $\rho^\q=0$, $\forall\;\q\in\neg\p$. Therefore, $s$ can be rewritten as
\[
	0 \leq s=\sum_{\p} \frac{\rho^\p}{\rho_{\max}^\p} \leq 1.
\]
From the last expression it is clear that the parameter $s$ generalizes the term $\frac{\rho}{\rho_{\max}}$ appearing in the case of single population models, see~\cite{hertyillner09,KlarWegener96,PrigogineHerman}.

Since $\Ap$ and $\Apq$ are probability densities, they fulfill the following properties:
\begin{equation}\label{eq:Aprop}
\begin{aligned}
	\Ap \geq 0,	\quad \int_{\Vp} \Ap dv=1, \quad\text{for } \vb,\va,v\in \Vp,\;s\in[0,1]\\
	\Apq \geq 0,	\quad \int_{\Vp} \Apq dv=1, \quad\text{for } \vb,v\in \Vp,\;\va\in\Vq,\;s\in[0,1].
\end{aligned}
\end{equation}

\begin{remark} \label{rem:consmass}
All transition probability densities $\Ap$ and $\Apq$ satisfying properties~\eqref{eq:Aprop} guarantee mass conservation. In fact, by integrating over the velocity space $\Vp$ we obtain
%\begin{align*}
%	\frac{d}{dt}\rho^\p(t)= & \int_{\Vp}\int_{\Vp}  \fbp \fap d\vb d\va + \sum_{q=\neg \p} \int_{\Vp}\int_{\Vq}  \fbp \faq d\vb d\va \\
%	& - \int_{\Vp} f^\p dv \left( \int_{\Vp} \fap d\va + \sum_{\q\in\neg\p} \int_{\Vq} \faq d\va \right) = \left(\rho^{\p}\right)^2+\rho^\p\sum_{\q\in\neg\p}\rho^\q - \left(\rho^{\p}\right)^2-\rho^\p\sum_{\q\in\neg\p}\rho^\q = 0.
%\end{align*}
\[
	\int_{\Vp} \Qpp[\fp,\fp](t,v) dv = \int_{\Vp}\int_{\Vp}  \fbp \fap \left(\int_{\Vp} \Ap dv \right) d\vb d\va - \int_{\Vp} f^\p dv \int_{\Vp} \fap d\va = 0.
\]
Analogously for the cross-interaction operators we have $\int_{\Vp} \Qpq[\fp,\fq](t,v) dv = 0$, for all $\q\in\neg\p$.
\end{remark}

\begin{remark} \label{rem:bounded-s}
Since the mass of each population is conserved, also $s$ given in~\eqref{eq:s} is conserved. In particular, it satisfies the prescribed bounds if the $\fp$'s are properly chosen at the initial time.
\end{remark}

\subsection{Choice of the probability densities}
\label{sec:modeling}
As in any Boltzmann-like kinetic traffic model, the introduction of a probability density allows one to assign a post-interaction speed in a non-deterministic way, consistently with the intrinsic stochasticity of the drivers' behavior. Therefore, the construction of $\Ap$ and $\Apq$, $\forall\;\p,\q\in\neg\p$, is at the core of the model we propose. They are obtained with a very small set of rules which are similar to those given in~\cite{PgSmTaVg2} for a single population model.

We will suppose that all types of vehicles react in the same way to the parameter $s$, accounting for the state of congestion of the road, and to all field classes of vehicles. Clearly, it would also be possible to consider different reactive behaviors for different classes of vehicles. However, this choice is coherent with the experience and, as we showed in~\cite{PgSmTaVg} for the heterogeneous discrete-velocity model, and as we will see later in Section~\ref{sec:funddiag}, this simpler choice results in a realistic macroscopic behavior.

\begin{itemize}
\item If $\vb\leq\va$, i.e. the candidate vehicle is slower than the leading vehicle, the post-interaction rules are:
\begin{description}
\item[Do nothing:] the candidate vehicle keeps its pre-interaction speed, thus $v=\vb$ with probability $1-P_1$;
\item[Accelerate:] the candidate vehicle accelerates to a velocity $v>\vb$ with probability $P_1$.
\end{description}
\item If $\vb>\va$, i.e. the candidate vehicle is faster than the leading vehicle, the post-interaction rules are:
\begin{description}
\item[Do nothing:] the candidate vehicle keeps its pre-interaction velocity, i.e. $v=\vb$, with probability $P_2$, thereby overtaking the leading vehicle;
\item[Brake:] the candidate vehicle decelerates to the velocity $v=\va$ with probability $1-P_2$, thereby queuing up and following the leading vehicle.
\end{description}
\end{itemize}

From the previous rules, we observe that the probability densities will contain terms that will be proportional to a Dirac delta function at $v=\vb$, due to interactions in which the pre-interaction microscopic speed is preserved (the two ``Do nothing'' alternatives). Note that these are ``false gains'' for the distributions $\fp$, because the number of vehicles of the $\p$-class with speed $v$ is not altered by these interactions. 

The speed after braking is assigned as 
proposed  in~\cite{Prigogine61} and used also  in~\cite{FermoTosin14,PgSmTaVg} in the context of a discrete velocity model. Namely, we suppose that if a vehicle brakes, interacting with a slower vehicle, it slows down to the speed $\va$ of the leading vehicle, thus, after the interaction, $v=\va$, and the field vehicle will remain behind the leading one.

For the post-interaction speed due to acceleration we assume that the output velocity $v$ is obtained by accelerating instantaneously from $\vb$ to $\vb+\Dv^\p$, unless the resulting speed is larger than $\vmp$, namely the new speed is $\min\left\{\vb+\Dv^\p,\vmp\right\}$. This choice corresponds to the case of the quantized acceleration (or $\delta$ model) introduced in~\cite{PgSmTaVg2} for a single population model. Clearly, other choices are possible, in particular the case of a uniformly distributed acceleration. However, we proved that although a model with such an acceleration is more refined, at equilibrium the essential information is caught by the simpler $\delta$ model. For this reason, here we will not investigate other models.

Considering all possible outcomes, the resulting probability distribution accounting for self-interactions is
\begin{equation}
	\Ap=
	\begin{cases}
		(1-P_1(s)) \, \delta_{\vb}(v) + P_1(s) \, \delta_{\min\left\{\vb+\Dv^\p,\vmp\right\}}(v), &\text{if\; $\vb\leq\va$}\\
		(1-P_2(s)) \, \delta_{\va}(v) + P_2(s) \, \delta_{\vb}(v), &\text{if\; $\vb>\va$}.
	\end{cases}
	\label{eq:Adelta}
\end{equation}

\begin{figure}
\centering
\begin{tikzpicture}
\draw [<->] (0,3.5) -- (0,0) -- (3.5,0);
\draw (2.5,0) -- (2.5,2.5) -- (0,2.5);
\draw[dashed] (0,0) -- (2.5,2.5);
\node at (0.8,1.8) {$\vb \leq \va$};
\node at (1.8,0.7) {$\vb > \va$};
\node [below] at (1.25,0) {$\Vp$};
\node [left] at (0,1.25) {$\Vp$};
\node[left] at (0,2.5) {$\vmp$};
\node[below] at (2.5,0) {$\vmp$};
\end{tikzpicture}
\begin{tikzpicture}
\draw [<->] (0,3.5) -- (0,0) -- (3.5,0);
\draw (2.5,0) -- (2.5,2) -- (0,2);
\draw[dashed] (0,0) -- (2,2);
\draw[dotted] (2,0) -- (2,2);
\node at (0.8,1.6) {$\vb \leq \va$};
\node at (1.8,0.7) {$\vb > \va$};
\node [below] at (1.25,0) {$\Vp$};
\node [left] at (0,1) {$\Vq$};
\node[left] at (0,2) {$\vmq$};
%\node[below] at (2,0) {$\vmq$};
\node[below] at (2.5,0) {$\vmp$};
\end{tikzpicture}
\begin{tikzpicture}
\draw [<->] (0,3.5) -- (0,0) -- (3.5,0);
\draw (2.5,0) -- (2.5,3) -- (0,3);
\draw[dashed] (0,0) -- (2.5,2.5);
\draw[dotted] (0,2.5) -- (2.5,2.5);
\node at (0.8,1.8) {$\vb \leq \va$};
\node at (1.8,0.7) {$\vb > \va$};
\node [below] at (1.25,0) {$\Vp$};
\node [left] at (0,1.5) {$\Vq$};
\node[left] at (0,3) {$\vmq$};
\node[below] at (2.5,0) {$\vmp$};
\end{tikzpicture}
\caption{The domain of the probability density $\Ap$ (left), the domains of the probability densities $\Apq$ in the case $\Vp\supset\Vq$ (center) and $\Vp\subset\Vq$ (right).\label{fig:domains}}
\end{figure}
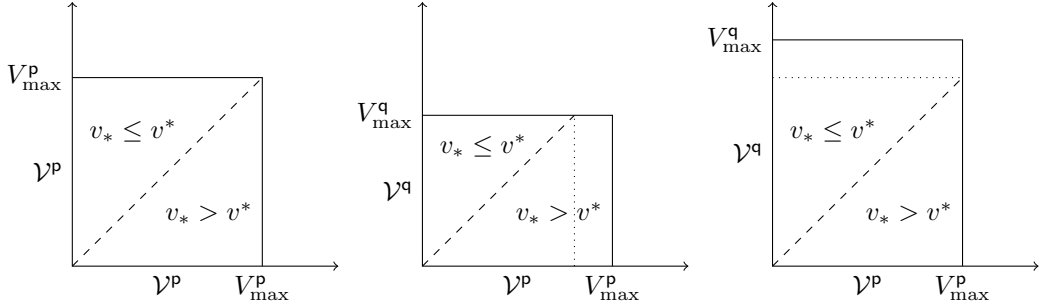

Since we have assumed that all classes of vehicles react in the same way to the parameter $s$ and to different interacting populations, the probability densities describing the cross-interactions differ from $\Ap$ only in their domain. In fact, $\Ap$ is defined for $(\vb,\va)\in\Vp\times\Vp$ and $v\in\Vp$, while $\Apq$ for $(\vb,\va)\in\Vp\times\Vq$ and $v\in\Vp$, see Figure~\ref{fig:domains}.

Note that the modeling~\eqref{eq:Adelta} of $\Ap$ and $\Apq$, $\forall\;\p,\q\in\neg\p$, may seem as a continuous extension of the multi-population model~\cite{PgSmTaVg} which was based on a discrete velocity space. However, in~\cite{PgSmTaVg} the {\em velocity jump} $\Dv^\p$ is chosen as the distance between two adjacent discrete velocities, thus $\Dv^\p$ depends on the number of elements in the speed lattice. In this work, $\Dv^\p$ is a {\em physical} parameter that represents the ability of a class of vehicles to change its pre-interaction speed $\vb$. With this choice, $\Dv^\p$ does not depend on the discretization of the velocity space $\Vp$ and the maximum acceleration is bounded, as in~\cite{Lebacque03} and~\cite{PgSmTaVg2}.  In contrast, deceleration can be larger than $\Dv^\p$, and this fact reflects the idea that drivers 
tend to brake immediately if the traffic becomes more congested, while they react more slowly when they accelerate (see the concept of {\em traffic hysteresis} in~\cite{ZhangMultiphase} and references therein).

In the following, the probabilities $P_1$ and $P_2$ are taken as $P=P_1=P_2$ and $P$ is a function of the fraction of occupied space $s$ only. In general $P$ should be a decreasing function of $s$, see also \cite{HertyIllner08} or \cite{Rosini}.
For instance in~\cite{PgSmTaVg} we have taken 
\begin{equation}\label{eq:gamma-law}
P=1- s^{\gamma}, \quad \gamma\in (0,1).
\end{equation}
In more sophisticated models, one may also choose $P$ as a function of the relative velocity of the interacting vehicles, but we will not explore this possibility in the present work.
 
\begin{remark} \label{rem:acceleration}
In~\cite{KlarWegener96} Klar and Wegener assume that the velocity after an acceleration is uniformly distributed over a range of speeds between $\vb$ and $\vb+\alpha(\vm-\vb)$, where $\alpha$ is supposed to depend on the local density. Thus, they suppose that the output speed resulting from the acceleration rule depends on the free space on the road, in other words they consider $\Dv^\p$ as function of $s$. Instead, here $\Dv^\p$ is fixed while $P$ is function of $s$, so that when the road becomes congested the probability of accelerating decreases.
\end{remark}

\section{Analysis of the model}
\label{sec:delta-multimodel}

In this section, first we rewrite explicitly the model using the expression~\eqref{eq:Adelta} for $\Ap$ and $\Apq$. Next, we discretize it in order to analyze later numerically the asymptotic traffic behavior, see Section~\ref{sec:funddiag}. Furthermore, we study the well-posedness (existence, uniqueness, and continuous dependence of the solution on initial data) of the discretized model and we characterize explicitly the asymptotic distributions $(\fp)^\infty$ of the discrete-velocity model.

The gain term of the collision operator~\eqref{eq:self-collision} describing self-interactions can be easily rewritten in the following way by distinguishing the cases $\vb\leq\va$ and $\vb>\va$:

\begin{align*}
	G^{\p\p}[\fp,\fp](t,v)=&\eta^\p \int_0^{\vmp}\int_{\vb}^{\vmp} \left[(1-P)\delta_{\vb}(v)+P\delta_{\min\left\{\vb+\Dv^\p,\vmp\right\}}(v)\right]\fbp\fap d\va d\vb \\
    &+\eta^\p \int_0^{\vmp}\int_0^{\vb}\left[(1-P)\delta_{\va}(v)+P\delta_{\vb}(v)\right]\fbp\fap d\va d\vb
\end{align*}
Observe that the Dirac delta function at $v=\min\left\{\vb+\Dv^\p,\vmp\right\}$ can be split as
\[
	\delta_{\min\left\{\vb+\Dv^\p,\vmp\right\}}(v)=
	\begin{cases}
		\delta_{\vb+\Dv^\p}(v), &\text{if\; $\vb\in[0,\vmp-\Dv^\p]$}\\
		\delta_{\vmp}(v), &\text{if\; $\vb\in(\vm-\Dv^\p,\vmp]$}
	\end{cases},
\]
%because the velocity jump of size $\Dv^\p$, leading to the output velocity $v=\vb+\Dv^\p$, can be performed only if $\vb \leq \vmp-\Dv^\p$. If instead  $\vb\in\left(\vmp-\Dv^\p,\vmp\right]$, the post-interaction velocity will be $v=\vmp$.
and since $\Qp$ is defined on $\Vp\times\Vp$, then $G^{\p\p}[\fp,\fp](t,v)$ corresponds to the gain term of the single population model. Thus, recalling~\cite{PgSmTaVg2}, $G^{\p\p}[\fp,\fp](t,v)$ can be written as
\begin{equation}\label{eq:self-gain}
\begin{aligned}
G^{\p\p}[\fp,\fp](t,v) = & \eta^\p (1-P) \fp \left[ \int_v^{\vmp}\fap d\va + \int_v^{\vmp}\fbp d\vb \right]  + \eta^\p P \fp \int_0^v\fap d\va \\
  & +\eta^\p P \fp(t,v-\Dv^\p) H_{\Dv^\p}(v) \int_{v-\Dv^\p}^{\vmp}\fap d\va \\
  & + \eta^\p P \delta_{\vmp}(v)\int_{\vmp-\Dv^\p}^{\vmp}\fbp d\vb\int_{\vb}^{\vmp}\fap d\va
\end{aligned}
\end{equation}
where $H_{\alpha}(v)$ is the Heaviside step function defined as $H_{\alpha}(v):=\frac{d}{dv}\max\{0,v-\alpha\}$, $\alpha\in\mathbb{R}$. Notice that the first two integrals on the right-hand side actually coincide, in the space homogeneous case. However, we kept them separate to stress the fact that they come from different contributions.

Conversely, the gain term resulting from the collision operators~\eqref{eq:cross-collision} describing the cross-interactions has to be treated differently because $\Qpq$ is defined on $\Vp\times\Vq$ and we have to distinguish $\vmp>\vmq$ (i.e. $\Vp\supset\Vq$) from $\vmp<\vmq$ (i.e. $\Vp\subset\Vq$), see Figure~\ref{fig:domains}. Since we also want to separate the cases $\vb\leq\va$ and $\vb>\va$, we take into account both these alternatives by rewriting the generic gain term $G^{\p\q}[\fp,\fq](t,v)$ as
\begin{align}\label{eq:cross-gain}
	G^{\p\q}[\fp,\fq](t,v)=&\eta^{\p\q} \int_0^{\min\{\vmp,\vmq\}} \int_{\vb}^{\vmq} \left[ (1-P) \delta_{\vb}(v) + P \delta_{\min\{\vb+\Dv^\p,\vmp\}}(v) \right] \fbp \faq d\va d\vb \\
	& + \eta^{\p\q} \int _0^{\vmp} \int_0^{\min\{\vb,\vmq\}} \left[ (1-P)\delta_{\va}(v) + \delta_{\vb}(v) \right] \fbp \faq d\va d\vb. \nonumber
\end{align}
Computing explicitly the terms appearing in $G^{\p\q}[\fp,\fq](t,v)$, we obtain two different expressions. If $\vmp>\vmq$:
\begin{equation}\label{eq:cross-gain1}
\begin{aligned}
	G^{\p\q}[\fp,\fq](t,v)=&\eta^{\p\q} (1-P) \fp \chi_{\left[0,\vmq\right]}(v) \int_{v}^{\vmq} \faq d\va + \eta^{\p\q} (1-P) \fq \int_{v}^{\vmp} \fbp d\vb \\
	& + \eta^{\p\q} P \fp \int_{0}^{\min\{v,\vmq\}} \faq d\va \\
	& + \eta^{\p\q} P \fp(t,v-\Dv^\p) \chi_{\left[\Dv^\p,\min\{\vmq+\Dv^\p,\vmp\}\right]}(v) \int_{v-\Dv^\p}^{\vmq} \faq d\va \\
	& + \eta^{\p\q} P \delta_{\vmp}(v) \int_{\min\{\vmq,\vmp-\Dv^\p\}}^{\vmq} \fbp d\vb \int_{\vb}^{\vmq} \faq d\va,
\end{aligned}
\end{equation}
while if $\vmp<\vmq$:
\begin{align}\label{eq:cross-gain2}
	G^{\p\q}[\fp,\fq](t,v)=&\eta^{\p\q} (1-P) \fp \int_{v}^{\vmq} \faq d\va + \eta^{\p\q} (1-P) \fq \int_{v}^{\vmp} \fbp d\vb + \eta^{\p\q} P \fp \int_{0}^{v} \faq d\va \nonumber \\
	& + \eta^{\p\q} P \fp(v-\Dv^\p) H_{\Dv^\p}(v) \int_{v-\Dv^\p}^{\vmq} \faq d\va \\ \nonumber
	& + \eta^{\p\q} P \delta_{\vmp}(v) \int_{\vmp-\Dv^\p}^{\vmp} \fbp d\vb \int_{\vb}^{\vmq} \faq d\va.
\end{align}

Using the explicit expressions~\eqref{eq:self-gain}, \eqref{eq:cross-gain1} and~\eqref{eq:cross-gain2} of the gain term, the model is then globally defined and it can be discretized as we do in the following subsection.

Finally, the following theorem states that when all the classes of vehicles composing the mixture of traffic have the same microscopic characteristics then the multi-population model is consistent with the equation defining the single population model introduced in~\cite{PgSmTaVg2}. This property is known in the kinetic theory of gas mixtures as {\em indifferentiability principle}, see e.g.~\cite{andries2001REPORT}.

\begin{theorem}[Indifferentiability principle]\label{th:indiff-principle}
Assume that the types of vehicles %are identical, i.e. they
have the same physical and kinematic characteristics, i.e.
\[
	l^\p:=l, \quad \vmp:=\vm, \quad \Dv^\p:=\Dv, \; \forall\;\p.
\]
Let $\eta^\p=\eta^{\p\q}=:\eta$ be the interaction rate. Then the total distribution function $f:\mathbb{R}^+\times \mathcal{V} \to \mathbb{R}^+$, defined as
\begin{equation}\label{eq:total-distr}
	f(t,v)=\sum_{\p} \fp(t,v)
\end{equation}
obeys the evolution equation of the single population model introduced in~\cite{PgSmTaVg2}.
\end{theorem}
\begin{proof}
If the populations have the same microscopic features then % $l^\p:=l$, $\vmp:=\vm$ and $\Dv^\p:=\Dv$, $\forall\;\p$. Thus,
for any fixed $\p$ in the set of all classes of vehicles the velocity spaces are such that $\Vp=\Vq:=\mathcal{V}$, $\forall\;\q\in\neg \p$, and the gain terms~\eqref{eq:self-gain}-\eqref{eq:cross-gain} are the same because now the probability densities $\Ap$ and $\Apq$ are defined on the same velocity space $\mathcal{V}\times\mathcal{V}$. Then, we have $\Ap=\Apq:=A(\vb\to v|\va;s)$ and moreover $s=\sum_{\p} \frac{\rho^\p}{\rho^\p_{\max}}=\frac{\sum_{\p}  \rho^\p}{\rho_{\max}}=\frac{\rho}{\rho_{\max}}$. In addition to that, the interaction rates are the same because the populations are identical, hence finally the multi-population model writes as
\begin{align*}
	\partial_t \fp & = \Qpp[\fp,\fp](t,v) + \sum_{\q\in\neg \p} \Qpq[\fp,\fq](t,v) \\
	& = \eta \int_{\mathcal{V}}\int_{\mathcal{V}} A(\vb\to v|\va;s) \fbp \sum_{j=\p,\q\in\neg\p} f^{j *} d\va d\vb - \eta \fp \int_{\mathcal{V}} \sum_{j=\p,\q\in\neg\p} f^{j *} d\va.
\end{align*}
Summing this equation over $\p$ and using the definition~\eqref{eq:total-distr}, we obtain
\[
	\partial_t f(t,v) = \eta \int_{\mathcal{V}}\int_{\mathcal{V}} A(\vb\to v|\va;s) f(t,\vb) f(t,\va) d\va d\vb - \eta f(t,v) \int_{\mathcal{V}} f(t,\va) d\va
\]
which represents the equation for the single population model given in~\cite{PgSmTaVg2}.
\end{proof}

\subsection{Discretization of the model}
\label{sec:discrete-model}

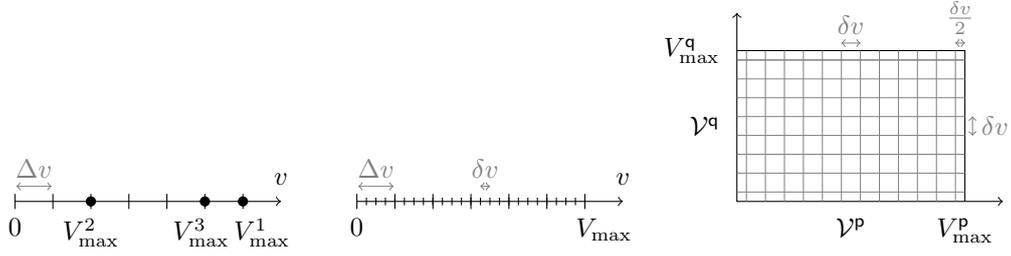
\begin{figure}
\centering
\begin{tikzpicture}
\draw [->] (0,0) -- (3.5,0);
\draw (0,-0.1) -- (0,0.1);
\node [below] at (0,-0.1) {$0$};
\node [above] at (3.5,0.1) {$v$};
 \draw \foreach \k in {0.5,1,1.5,2,2.5,3}
    {(\k,-0.1) -- (\k,0.1)};
\draw [fill] \foreach \k in {1,2.5,3}
    {(\k,0) circle [radius=0.06] };
\draw[help lines,<->] (0.02,0.2) -- (0.48,0.2)
    node[pos=0.5, above] {$\Delta v$};
\node[below] at (2.45,-0.1) {$V_{\max}^{3}$};
\node[below] at (3.25,-0.1) {$V_{\max}^{1}$};
\node[below] at (1,-0.1) {$V_{\max}^{2}$};
\begin{scope}[xshift=4.5 cm]
\draw [->] (0,0) -- (3.5,0);
\draw (0,-0.1) -- (0,0.1);
\node [below] at (0,-0.1) {$0$};
\node [above] at (3.5,0.1) {$v$};
 \draw \foreach \k in {0.5,1,1.5,2,2.5,3}
    {(\k,-0.1) -- (\k,0.1)};
\draw \foreach \k in
 {0.125,0.25,0.375,0.5,0.625,0.75,0.875,1,1.125,1.25,1.375,1.5,1.625,1.75,1.875,2,2.125,2.25,2.375,2.5,2.625,2.75,2.875,3}
    {(\k,-0.05) -- (\k,0.05)};
\draw[help lines,<->] (0.02,0.2) -- (0.48,0.2)
    node[pos=0.5, above] {$\Delta v$};
\draw[help lines,<->] (1.625,0.2) -- (1.75,0.2)
    node[pos=0.5, above] {$\delta v$};
\node[below] at (3.25,-0.1) {$V_{\max}$};
\end{scope}
\begin{scope}[xshift=9.5 cm]
\draw [<->] (0,2.5) -- (0,0) -- (3.5,0);
\draw (0,2) -- (3,2) -- (3,0);
\draw[help lines,-] (0,2-0.125) -- (3,2-0.125);
\draw[help lines,-] (0+0.125,2) -- (0+0.125,0);
\draw[help lines,-] (0,0+0.125) -- (3,0+0.125);
\draw[help lines,-] (3-0.125,0) -- (3-0.125,2);
\draw[help lines,-] \foreach \k in {0.25,0.5,0.75,1,1.25,1.5,1.75,2,2.25,2.5}
	{(\k+0.125,0) -- (\k+0.125,2)};
\draw[help lines,-] \foreach \k in {0.25,0.5,0.75,1,1.25,1.5,1.75}
	{(0,\k+0.125) -- (3,\k+0.125)};
\draw[help lines,<->] (3.1,0.875) -- (3.1,1.125)
    node[pos=0.5, right] {$\delta v$};
\draw[help lines,<->] (1.375,2.1) -- (1.625,2.1)
    node[pos=0.5, above] {$\delta v$};
\draw[help lines,<->] (3-0.125,2.1) -- (3,2.1)
    node[pos=0.5, above] {$\tfrac{\delta v}{2}$};
\node [below] at (1.5,-0.1) {$\mathcal{V}^{\mathsf{p}}$};
\node [left] at (-0.1,1) {$\mathcal{V}^{\mathsf{q}}$};
\node[left] at (-0.1,2) {$V_{\max}^{\mathsf{q}}$};
\node[below] at (3,-0.1) {$V_{\max}^{\mathsf{p}}$};
\end{scope}
\end{tikzpicture}
\caption{Discretization of velocity spaces.\label{fig:discretization}}
\end{figure}

In order to address the qualitative properties of the model, we consider a discretization of the velocity spaces. The study becomes simpler with the following assumptions:
\begin{description}
	\item[Assumption 1] The velocity jump $\Dv^\p$ is a fixed parameter and $\Dv^\p=\Dv$, $\forall\;\p$;
	\item[Assumption 2] $\Dv=\vmp/T^\p$, with $T^\p\in\mathbb{N}$, $\forall\;\p$. This means that
	\[
		\abs{\vmp-\vmq}=m^{\p\q}\Dv, \quad \forall\;\p,\, \q\in\neg\p,\, m^{\p\q}\in\mathbb{N},
	\]
	i.e. the distance between the maximum velocities of the populations is a multiple of $\Dv$, see the left panel of Figure~\ref{fig:discretization};
	\item[Assumption 3] Let $\dv$ be the numerical parameter of the velocity space discretization. We take $\dv=\frac{\Dv}{r}$, i.e. $\Dv$ corresponds to an integer number of intervals in the velocity discretization, see the center panel of Figure~\ref{fig:discretization}. Thus $r=\frac{\np-1}{T^\p}\in\mathbb{N}$ depends on the number of grid points $\np$.
\end{description}

Notice that Assumption 1 is not as restrictive as it appears. Since $\Dv^\p$ represents the instantaneous velocity jump of a vehicle as a result of the acceleration interaction, it might be thought of as another microscopic feature characterizing the classes of vehicles. Then we may suppose that the populations do not differ in the jump of velocity but we can characterize them by assuming that they have different microscopic accelerations $a^\p$. As proved in~\cite{PgSmTaVg2} in the case of a single population model, the acceleration of the vehicles depends on the product of $\Dv$ and the interaction rate, thus in order to account for different accelerations in each class we could act on the interaction rates $\eta^\p$ and $\eta^{\p\q}$, without modifying $\Dv$. However, this is relevant only in the spatially inhomogeneous model because in the framework analyzed in this work the interaction rates influence only the rate of convergence towards equilibrium solutions. However, the equilibrium solutions themselves do not depend on the interaction rates.

Finally, observe that Assumption 2 leads to:
\[
	\eta^{\p\q}P\delta_{\vmp}(v) \int_{\min\{\vmq,\vmp-\Dv\}}^{\vmq} \fbp d\vb \int_{\vb}^{\vmq} \faq d\va \equiv 0
\]
in the cross-collision term~\eqref{eq:cross-gain1}, i.e. when $\Vp\supset\Vq$. In fact, the previous integral is zero since $\min\{\vmq,\vmp-\Dv\}=\vmq$.

We define the velocity cells $I_j^\p=[(j-\tfrac32)\dv, (j-\tfrac12)\dv]\cap [0,\vmp]$, $\forall\;\p$ and for $j=1,\dots,\np=\frac{\vmp}{\dv}+1$. Note that all cells have amplitude $\dv$ except $I_1^\p$ and $I_{\np}^\p$ which have amplitude $\dv/2$. The velocity grid nodes, located at the center of each cell, are $v_1^\p=\dv/4, v_{\np}^\p=\vmp -\dv/4$, and $v_j^\p=(j-1)\dv$ for $j=2,\dots,\np-1$. This choice was already used in~\cite{PgSmTaVg2} and it is convenient for computations because all grids with $r>1$  contain all the points of the coarser mesh with $r=1$ (except the first and the last point). See the right panel of Figure~\ref{fig:discretization} for the discretization of $\Vp\times\Vq$.

In order to discretize the model, we approximate each velocity distribution with the piece-wise constant function
\begin{equation} \label{eq:f:discrete}
\fp(t,v) = \sum_{j=1}^{\np} f_j^\p(t) \frac{\chi_{I_j^\p}(v)}{\abs{I_j^\p}}
\end{equation}
where $f_j^\p:\mathbb{R}^+\to[0,\rho^\p]$ represents the number of $\p$-vehicles traveling with velocity $v \in I_j^\p$. 

By integrating the kinetic equation \eqref{eq:model1} over the cells $I_j^\p$ we obtain the following system of ordinary differential equations
\begin{equation}
		\frac{df_j^\p(t)}{dt}=\Qp_j[\fp,(\fp,\fq)](t) =\underbrace{\int_{I_j^\p} \Qpp[\fp,\fp](t,v)dv}_{\Qpp_j[\fp,\fp](t)}+\underbrace{\int_{I_j^\p} \Qpq[\fp,\fq](t,v)dv}_{\Qpq_j[\fp,\fq](t)}, \quad j=1,\dots,\np,\,\forall\;\p
    \label{eq:discretesys}
\end{equation}
whose initial condition $f_1^\p(0),\dots,f_{\np}^\p(0)$ are such that:
\[
	\sum_{j=1}^{\np} f_j^\p(0)=\int_{\Vp} \fp(t=0,v)dv=\rho^\p, \quad \forall\;\p,\; \rho^\p\in[0,\rho^\p_{\max}]
\]
and $\rho^\p$ is the initial density, which remains constant in the  spatially homogeneous case, see Remark~\ref{rem:consmass}.

Once the right-hand side of~\eqref{eq:discretesys} is computed, then the ODE system can be easily rewritten by means of the matrices $\A^{\p,j}$ and $\B^{\p\q,j}$, for $j=1,\dots,\np$, which are the so-called {\em self}- and {\em cross-interaction} matrices, respectively. The terms $\Qpp_j[\fp,\fp](t)$ and $\Qpq_j[\fp,\fq](t)$ can be explicitly written out and they appear in the Appendix~\ref{app:discrete-terms}.

The multi-population model writes finally as
\begin{equation}\label{eq:explicit-discretesys}
\begin{aligned}
	\frac{df_j^\p(t)}{dt}=& \eta^\p \sum_{h,k=1}^{\np} A_{hk}^{\p,j} f_h^\p f^{\p k} + \sum_{\q\in\neg\p} \eta^{\p\q} \sum_{h=1}^{\np} \sum_{k=1}^{\nq} B_{hk}^{\p\q,j} f_h^\p f^{\q k} \\
	& - f_j^\p\left( \eta^\p \sum_{k=1}^{\np} f^{\p k} + \sum_{\q\in\neg\p}  \eta^{\p\q} \sum_{k=1}^{\nq} f^{\q k} \right), \quad j=1,\dots,\np, \forall\;\p
\end{aligned}
\end{equation}
where the kinetic distribution functions of candidate and field vehicles are distinguished by the position of the index of the components: bottom-right for candidate vehicles (such as e.g. $f_h^\p$), top-right for field vehicles (such as e.g. $f^{\p k}$ or $f^{\q k}$). The matrices $\A^{\p,j}$ and $\B^{\p\q,j}$ are defined as
\begin{equation}
A^{\p,j}_{hk}=\mathcal{P}\left(\vb\in I_h^\p \to v\in I_j^\p | \va\in I_k^\p\right) \quad\text{and}\quad B^{\p\q,j}_{hk}=\mathcal{P}\left(\vb\in I_h^\p \to v\in I_j^\p | \va\in I_k^\q\right),
\end{equation}
namely they contain in the position $(h,k)$ the probability that the candidate vehicle with velocity in $I_h^\p$ will acquire a velocity in $I_j^\p$ when it interacts with a field vehicle traveling at a velocity in $I_k^\p$, if the $\p$-vehicles play also the role of field class, or in $I_k^\q$ otherwise.

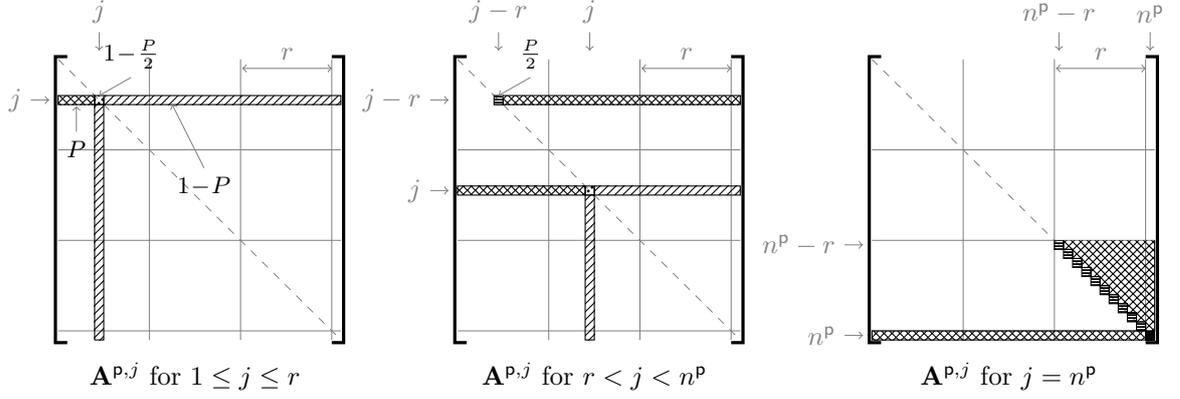
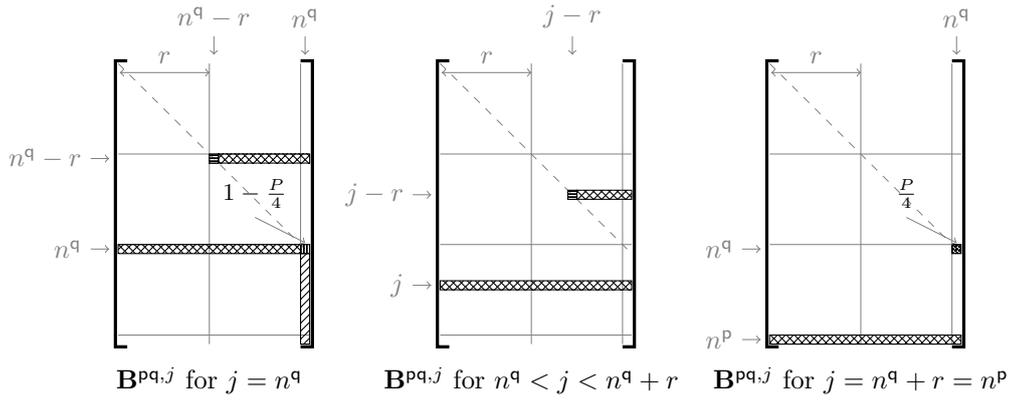
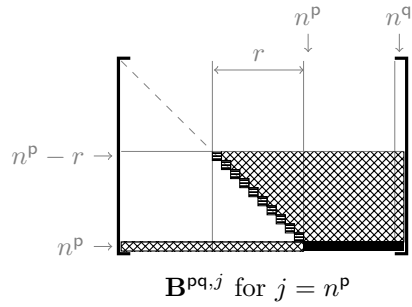
\begin{figure}
\centering
\subfloat[][Self-interaction matrices $\A^{\p,j}$, $j=1,\dots,\np$.\label{fig:self-matrices}]{
\begin{tikzpicture}[y={(0,-1cm)},scale=1.2]
%Matrice A per 0<=j<r
\begin{scope}
\skeleton
\node[anchor=center] at (1.5,3.5) {$\A^{\p,j}$ for $1\leq j\leq r$};
\filldraw[pattern=\patP] (0,0.4) rectangle (0.4,0.5);
\filldraw[pattern=\patUnoMenoPmezzi] (0.4,0.4) rectangle (0.5,0.5);
\filldraw[pattern=\patUnoMenoP] (0.5,0.4) rectangle (3.1,0.5);
\filldraw[pattern=\patUnoMenoP] (0.4,0.5) rectangle (0.5,3.1);
\draw[help lines,->,yshift=-0.45cm]
     (-0.3,0) node[left]{$j$} -- (-0.1,0);
\draw[help lines,->,xshift=0.45cm]
     (0,-0.3) node[above]{$j$} -- (0,-0.1);
\draw[help lines,->] (0.2,0.8) node[below,black]{\small$P$} -- (0.2,0.52);
\draw[help lines,->] (0.8,0.2) node[above,black]{\small$1\!-\!\tfrac{P}{2}$} -- (0.45,0.4);
\draw[help lines,->] (1.6,1.2) node[below,black]{\small$1\!-\!P$} -- (1.25,0.5);
\end{scope}
\end{tikzpicture}
\begin{tikzpicture}[y={(0,-1cm)},scale=1.2]
%Matrice A per r<j<np
\begin{scope}[xshift=4.5cm]
\skeleton
\node[anchor=center] at (1.5,3.5) {$\A^{\p,j}$ for $r<j<\np$};
\filldraw[pattern=\patPmezzi] (0.4,0.4) rectangle (0.5,0.5);
\filldraw[pattern=\patP] (0.5,0.4) rectangle (3.1,0.5);
\filldraw[pattern=\patP] (0,1.4) rectangle (1.4,1.5);
\filldraw[pattern=\patUnoMenoPmezzi] (1.4,1.4) rectangle (1.5,1.5);
\filldraw[pattern=\patUnoMenoP] (1.5,1.4) rectangle (3.1,1.5);
\filldraw[pattern=\patUnoMenoP] (1.4,1.5) rectangle (1.5,3.1);
\draw[help lines,->,yshift=-1.45cm]
     (-0.3,0) node[left]{$j$} -- (-0.1,0);
\draw[help lines,->,xshift=1.45cm]
     (0,-0.3) node[above]{$j$} -- (0,-.1);
\draw[help lines,->,yshift=-0.45cm]
     (-0.3,0) node[left]{$j-r$} -- (-0.1,0);
\draw[help lines,->,xshift=0.45cm]
     (0,-0.3) node[above]{$j-r$} -- (0,-0.1);

\draw[help lines,->] (0.8,0.2) node[above,black]{\small$\tfrac{P}{2}$} -- (0.45,0.4);
\end{scope}
\end{tikzpicture}
\begin{tikzpicture}[y={(0,-1cm)},scale=1.2]
%Ultima matrice A
\begin{scope}[xshift=9cm]
\skeleton
\node[anchor=center] at (1.5,3.5) {$\A^{\p,j}$ for $j=\np$};

\filldraw[pattern=\patP] (2.1,2)
  \foreach \k in {2.1,2.2,2.3,2.4,2.5,2.6,2.7,2.8,2.8,2.9,3.0}
    {-- (\k,\k) -- (\k+0.1,\k)}
  -- (3.1,2.0);

\filldraw[pattern=\patPmezzi] (2.1,2.0)
  \foreach \k in {2.1,2.2,2.3,2.4,2.5,2.6,2.7,2.8,2.8,2.9}
    {-- (\k,\k) -- (\k+0.1,\k)}
  -- (3.0,3.0) -- (2.9,3.0)
  \foreach \k in {2.9,2.8,2.7,2.6,2.5,2.4,2.3,2.2,2.1}
    {-- (\k,\k) -- (\k-0.1,\k)}
  -- (2.0,2.0) ;

\fill[pattern=\patUno] (3.0,3.0) rectangle (3.1,3.1);

\filldraw[pattern=\patP] (0,3.0) rectangle (3.0,3.1);

\draw[help lines,->,yshift=-2.05cm]
     (-0.3,0) node[left]{$\np-r$} -- (-0.1,0);
\draw[help lines,->,xshift=2.05cm]
     (0,-0.3) node[above]{$\np-r$} -- (0,-0.1);
\draw[help lines,->,yshift=-3.05cm]
     (-0.3,0) node[left]{$\np$} -- (-0.1,0);
\draw[help lines,->,xshift=3.05cm]
     (0,-0.3) node[above]{$\np$} -- (0,-0.1);
\end{scope}
\end{tikzpicture}
}\\
\subfloat[][Cross-interaction matrices $\B^{\p\q,j}$, $j\geq\nq$, for the case $\Vp\supset\Vq$.\label{fig:cross-matrices1}]{
\begin{tikzpicture}[y={(0,-1cm)},scale=1.2]
%Matrice B per j=nq
\begin{scope}[xshift=4.5cm]
\skeletonB
\node[anchor=center] at (1,3.5) {$\B^{\p\q,j}$ for $j=\nq$};
\filldraw[pattern=\patPmezzi] (1,1) rectangle (1.1,1.1);
\filldraw[pattern=\patP] (1.1,1) rectangle (2.1,1.1);
\filldraw[pattern=\patP] (0,2) rectangle (2,2.1);
\filldraw[pattern=\patUnoMenoPquarti] (2,2) rectangle (2.1,2.1);
\filldraw[pattern=\patUnoMenoP] (2,2.1) rectangle (2.1,3.1);
\draw[help lines,->,yshift=-2.05cm]
     (-0.3,0) node[left]{$\nq$} -- (-0.1,0);
\draw[help lines,->,xshift=2.05cm]
     (0,-0.3) node[above]{$\nq$} -- (0,-.1);
\draw[help lines,->,yshift=-1.05cm]
     (-0.3,0) node[left]{$\nq-r$} -- (-0.1,0);
\draw[help lines,->,xshift=1.05cm]
     (0,-0.3) node[above]{$\nq-r$} -- (0,-0.1);

\draw[help lines,->] (1.5,1.7) node[above,black]{\small$1-\tfrac{P}{4}$} -- (2.05,1.98);
\end{scope}
\end{tikzpicture}
\begin{tikzpicture}[y={(0,-1cm)},scale=1.2]
%Matrice N per nq<j<nq+r
\begin{scope}[xshift=4.5cm]
\skeletonB
\node[anchor=center] at (1,3.5) {$\B^{\p\q,j}$ for $\nq<j<\nq+r$};
\filldraw[pattern=\patPmezzi] (1.4,1.4) rectangle (1.5,1.5);
\filldraw[pattern=\patP] (1.5,1.4) rectangle (2.1,1.5);
\filldraw[pattern=\patP] (0,2.4) rectangle (2.1,2.5);
\draw[help lines,->,yshift=-2.45cm]
     (-0.3,0) node[left]{$j$} -- (-0.1,0);
\draw[help lines,->,yshift=-1.45cm]
     (-0.3,0) node[left]{$j-r$} -- (-0.1,0);
\draw[help lines,->,xshift=1.45cm]
     (0,-0.3) node[above]{$j-r$} -- (0,-0.1);

%\draw[help lines,->] (1.5,1.7) node[above,black]{\small$1-\tfrac{P}{4}$} -- (2.05,1.98);
\end{scope}
\end{tikzpicture}
\begin{tikzpicture}[y={(0,-1cm)},scale=1.2]
%Matrice B per j=np
\begin{scope}[xshift=4.5cm]
\skeletonB
\node[anchor=center] at (1,3.5) {$\B^{\p\q,j}$ for $j=\nq+r=\np$};
\filldraw[pattern=\patPquarti] (2,2) rectangle (2.1,2.1);
\filldraw[pattern=\patP] (0,3) rectangle (2.1,3.1);
\draw[help lines,->,yshift=-3.05cm]
     (-0.3,0) node[left]{$\np$} -- (-0.1,0);
\draw[help lines,->,yshift=-2.05cm]
     (-0.3,0) node[left]{$\nq$} -- (-0.1,0);
\draw[help lines,->,xshift=2.05cm]
     (0,-0.3) node[above]{$\nq$} -- (0,-0.1);

\draw[help lines,->] (1.5,1.7) node[above,black]{\small$\tfrac{P}{4}$} -- (2.05,1.98);
\end{scope}
\end{tikzpicture}
}\\
\subfloat[][\mbox{Cross-interaction matrices $\B^{\p\q,j}$, $j=\np$, for the case $\Vp\subset\Vq$.}\label{fig:cross-matrices2}]{
\begin{tikzpicture}[y={(0,-1cm)},scale=1.2]
%Matrice B per j=np
\begin{scope}[xshift=9cm]
\skeletonC
\node[anchor=center] at (1.5,2.5) {$\B^{\p\q,j}$ for $j=\np$};

\filldraw[pattern=\patP] (1.1,1)
  \foreach \k in {1.1,1.2,1.3,1.4,1.5,1.6,1.7,1.8,1.9,2.0}
    {-- (\k,\k) -- (\k+0.1,\k)}
  -- (3.1,2) -- (3.1,1);

\filldraw[pattern=\patPmezzi] (1.1,1)
  \foreach \k in {1.1,1.2,1.3,1.4,1.5,1.6,1.7,1.8,1.9,2.0}
    {-- (\k,\k) -- (\k+0.1,\k)}
  -- (2.0,2.0) -- (1.9,2.0)
  \foreach \k in {1.9,1.8,1.7,1.6,1.5,1.4,1.3,1.2,1.1}
    {-- (\k,\k) -- (\k-0.1,\k)}
  -- (1,1) ;

\fill[pattern=\patUno] (2.0,2.0) rectangle (3.1,2.1);

\filldraw[pattern=\patP] (0,2.0) rectangle (2.0,2.1);

\draw[help lines,->,yshift=-1.05cm]
     (-0.3,0) node[left]{$\np-r$} -- (-0.1,0);
\draw[help lines,->,xshift=2.05cm]
     (0,-0.3) node[above]{$\np$} -- (0,-0.1);
\draw[help lines,->,yshift=-2.05cm]
     (-0.3,0) node[left]{$\np$} -- (-0.1,0);
\draw[help lines,->,xshift=3.05cm]
     (0,-0.3) node[above]{$\nq$} -- (0,-0.1);
\end{scope}
\end{tikzpicture}
}
\caption{General structure of the interaction matrices.\label{fig:matrices}}
\end{figure}

In Figure~\ref{fig:self-matrices} we show the sparse structure of the self-interaction matrices $\A^{\p,j}\in\mathbb{R}^{\np\times \np}$, for $j=1,\dots,\np$. The cross-interaction matrices $\B^{\p\q,j}$ for the case $\Vp\supset\Vq$ have the same structure as the $\A^{\p,j}$'s, apart from being rectangular of dimensions $\np\times \nq$, with $\np>\nq$. Differences however arise for $j\geq \nq$, see Figure~\ref{fig:cross-matrices1} in which for simplicity we have assumed that $\nq+r=\np$ (i.e. $\vmp-\vmq=\Dv$). Finally, the cross-interaction matrices $\B^{\p\q,j}$ for the case $\Vp\subset\Vq$ are $\np\times \nq$, with $\np<\nq$. They can in turn be easily derived from the $\A^{\p,j}$'s, the only different case being the one for $j=\np$, see Figure~\ref{fig:cross-matrices2} in which for simplicity we have assumed that $\nq=\np+r$ (i.e. $\vmq-\vmp=\Dv$).

In all these figures, the nonzero elements are shaded with different hatchings, corresponding to the different values of the elements, as indicated in the panels in which they appear for the first time.

The fact that these matrices are sparse means that a velocity in $I_j^\p$ can be acquired only for special values of the pre-interaction velocity of candidate and field vehicles. In particular, the $j$-th row of the matrices contains the probability that the candidate vehicle {\em does not} change its speed. The nonzero elements of the $j$-th column are the probabilities that a candidate vehicle acquires a speed in $I_j^\p$ by braking down to the speed of the leading vehicle. The nonzero row, located at $h=j-r$, contains the probabilities that the candidate vehicle accelerates by $\Delta v$, acquiring therefore a velocity in $I_j^\p=I_h^\p+\Delta v$. The band between the rows  $h=j-r$ and $h=j$ is filled with zeroes because the acceleration is quantized, i.e. the post-interaction velocity is acquired by a velocity jump.

As it can be checked, both the self- and the cross-interaction matrices are stochastic with respect to the index $j$, i.e. all their elements are positive and bounded above by $1$ and
%\begin{gather*}
\[
%	0\leq A^{\p,j}_{hk},\, B^{\p\q,j}_{hk}\leq 1, \quad \forall\;h,\,k,\,j,\,\p,\,\q\in\neg\p \\
	\sum_{j=1}^{\np} \A^{\p,j}=\mathfrak{O}_{\np\times\np}, \quad \sum_{j=1}^{\np} \B^{\p\q,j}=\mathfrak{O}_{\np\times\nq}, \quad \forall\; \p,\,\q\in\neg\p.
%\end{gather*}
\]
where $\mathfrak{O}_{N\times M}$ is the matrix of size $N\times M$ with the value $1$ in each entry. These properties come from~\eqref{eq:Aprop}, and they guarantee mass conservation.

\subsection{Qualitative analysis}
\label{sec:analysis}
Now we study the well-posedness of the discrete-velocity model~\eqref{eq:explicit-discretesys} and for the sake of simplicity we work with two populations only, but the analysis can be generalized to multiple populations. We show that the model has a unique solution, which remains positive and bounded in time, and finally we compute the analytical equilibria.

To this end, we consider the Cauchy problem associated to the ODE system~\eqref{eq:explicit-discretesys} with initial conditions $f^\p_j(0)$, $j=1,\dots,\np$, $\forall\;\p$. Since in spatially homogeneous conditions the density is constant in time, the generic equation of the Cauchy problem can be written as
\begin{equation}\label{eq:CP}
\frac{df_j^\p(t)}{dt}= \eta^\p \sum_{h,k=1}^{\np} A_{hk}^{\p,j} f_h^\p f_k^\p + \sum_{\q\in\neg\p} \eta^{\p\q} \sum_{h=1}^{\np} \sum_{k=1}^{\nq} B_{hk}^{\p\q,j} f_h^\p f_k^\q - R^\p f_j^\p, \quad j=1,\dots,\np, \forall\;\p
\end{equation}
where $R^\p:=\eta^\p \rho^\p + \sum_{\q\in\neg\p}  \eta^{\p\q} \rho^\q$. Notice that in \eqref{eq:CP} we do not use the upper index to distinguish the distributions, but the index $h$ is referred to the candidate vehicle distribution, while the index $k$ to the field vehicle distributions. We assume that each density $\rho^\p$ is normalized with respect to $\rho^\p_{\max}$ so that
\begin{equation}\label{eq:init.cond}
\begin{aligned}
0 & \leq f^\p_j(0) \leq \rho^\p \leq 1, \quad j=1,\dots,n^\p, \forall\;\p\\
0 & \leq \sum_{j=1}^{n^\p} f^\p_j(0)=\rho^\p \leq 1, \quad \forall\;\p.
\end{aligned}
\end{equation}

Assume that the two populations are labeled with $\p$ and with $\q$, respectively, and let
\[
X=C\left(\mathbb{R}^+;\,\mathbb{R}^{n^\p}\right)\times C\left(\mathbb{R}^+;\,\mathbb{R}^{n^\q}\right)
\]
be the space of vector-valued continuous functions on $\mathbb{R}^+$. From now on, we will endow $\mathbb{R}^{n^\p}$ and $\mathbb{R}^{n^\q}$ with the $1$-norm. Let $(\uu,\vv)$ be a generic element of $X$ and let $\f^\p=\left(f^\p_1,\dots,f^\p_{n^\p}\right)$ and $\f^\q=\left(f^\q_1,\dots,f^\q_{n^\q}\right)$ be the solution of the Cauchy problem. We define
\[
\BB=\left\{(\uu,\vv)\in X : 0 \leq u_j(t) \leq \rho^\p, 0 \leq v_j(t) \leq \rho^\q,\; \sum_{j=1}^{n^\p} u_j(t) = \rho^\p, \sum_{j=1}^{n^\q} v_j(t) = \rho^\q \right\}
\]
the subset of $X$ such that $\rho^\p$ and $\rho^\q$ are fixed constants, the same as in~\eqref{eq:CP}, then the elements of $\BB$ have the properties we require for the solution $(\f^\p,\f^\q)$ of the Cauchy problem.

\begin{theorem}[Well posedness]
	\label{th:well.posedness}
	For any given initial condition $(\f^\p(0),\f^\q(0))$ satisfying the assumptions~\eqref{eq:init.cond}, there exists a unique global solution $(\f^\p,\f^\q)\in \BB$ of the Cauchy problem \eqref{eq:CP}.
\end{theorem}
\begin{proof}
	The proof of the theorem is organized by steps and is achieved by applying the basic general theorem on global existence and uniqueness of solutions to first order ODE systems with sublinear growth.
	\begin{description}
		\item[Step 1] First of all we prove that the solution $(\f^\p(t),\f^\q(t))$ (if any) of the Cauchy problem \eqref{eq:CP} remains positive in time if the initial conditions satisfy \eqref{eq:init.cond}. For this, consider the population $\p$ and assume that there exists an index $j=1,\dots,\np$ such that $f_j^\p(0)=0$. Then, since the elements of the transition matrices are non-negative, $\frac{df_j^\p(t)}{dt}\geq 0$ (see equation \eqref{eq:CP}). Therefore the solution $(\f^\p(t),\f^\q(t))$ of the first order ODE system \eqref{eq:CP} cannot cross the hyperplane $f_j^\p=0$ of the phase space during the time evolution and this guarantees that the solution is non-negative.
		
		\item[Step 2] Now we show that the hypotheses of the general theorem on global existence and uniqueness of solutions to first order ODE systems with sublinear growth are satisfied. In particular, observe that the right hand side of equation \eqref{eq:CP} is continuous and locally Lipschitz continuous, since it is at most a quadratic function in the unknowns $\f=(\f^\p,\f^\q)$. Moreover, focusing on the population $\p$, since $A_{hk}^{\p,j}\in[0,1]$ and $B_{hk}^{\p\q,j}\in[0,1]$, the non-negativeness of the solution (Step 1) and the conservation of mass yield the following bounds:
		\begin{gather*}
		0 \leq \sum_{h=1}^{\np} A_{hk}^{\p,j} f_h^\p \leq \sum_{h=1}^{\np} f_h^\p = \rho^\p, \quad 0 \leq \sum_{h=1}^{\np} B_{hk}^{\p\q,j} f_h^\p \leq \sum_{h=1}^{\np} f_h^\p = \rho^\p
		\end{gather*}
		which render each differential equation sublinear. In fact, we obtain
		\begin{align*}
		\norm{\frac{d\f^\p}{dt}}_1 & = \sum_{j=1}^{\np} \abs{\eta^\p \sum_{k=1}^{\np} \left(\sum_{h=1}^{\np} A_{hk}^{\p,j} f_h^\p\right) f_k^\p + \eta^{\p\q}  \sum_{k=1}^{\nq} \left(\sum_{h=1}^{\np} B_{hk}^{\p\q,j} f_h^\p\right) f_k^\q - R^\p f_j^\p} \\
		& \leq \sum_{j=1}^{\np} \abs{\eta^\p \left(\rho^\p\right)^2 + \eta^{\p\q} \rho^\p \rho^\q - R^\p f_j^\p} =  \sum_{j=1}^{\np} R^\p \abs{\rho^p - f_j^\p} \leq R^\p (n^\p-1) \norm{\f^\p}_1.
		\end{align*}
		Notice that a similar inequality holds for the population $\q$:
		\[
		\norm{\frac{d\f^\q}{dt}}_1 \leq R^\q (n^\q-1) \norm{\f^\q}_1.
		\]
		This implies the existence and uniqueness of a global and infinitely differentiable in time solution of the Cauchy problem associated to the ODE system~\eqref{eq:CP} with initial conditions \eqref{eq:init.cond}.
		
		\item[Step 3] Finally, we show that the solution $\f=(\f^\p,\f^\q)$ is an element of the subset $\BB$. We have already proved that $\f$ remains non-negative in time (Step 1). Again, let us focus the attention on the population $\p$. By Step 2 the following inequality holds:
		\[
		\frac{df_j^\p}{dt} + R^\p f_j^\p \leq R^\p \rho^\p.
		\]
		Multiplying it by $e^{t R^\p}$ and integrating in time we obtain
		\[
		f_j^\p e^{t R^\p} \leq f_j^\p(0) + \int_0^t R^\p \rho^\p e^{s R^\p} ds.
		\]
		Evaluating explicitly the integral and multiplying both sides by $e^{-t R^\p}$ we finally get
		\[
		f_j^\p \leq f_j^\p(0) e^{-t R^\p} + \rho^\p (1-e^{-t R^\p})
		\]
		which proves that $f_j^\p:\mathbb{R}^+\to[0,\rho^\p]\subseteq[0,1]$. Moreover, using the integral formulation of equation \eqref{eq:CP} we prove that the sum of the $f_j^\p$'s is constant for all $t\geq 0$ and equal to the density $\rho^\p$. In fact:
		\[
		\sum_{j=1}^{n^\p} f_j^\p = e^{-tR^\p} \sum_{j=1}^{n^\p} f^\p_j(0) + \int_0^t e^{(\tau-t)R^\p} \rho^\p R^\p d\tau=e^{-R^\p t} \rho^\p \left[ 1 + R^\p \int_0^t e^{\tau R^\p} d\tau \right] = \rho^\p.
		\]
		Similarly, we have that $f_j^\q:\mathbb{R}^+\to[0,\rho^\q]\subseteq[0,1]$ and $\sum_{j=1}^{n^\q} f_j^\q=\rho^\q$, for all $t\geq 0$. This proves that the solution $\f=(\f^\p,\f^\q)\in\BB$.
		%\item[Step 4] Since $\left(\f^\p,\f^\q\right)\in\BB$ and the matrices $\A^{\p,j}$ and $\B^{\p\q,j}$ do not depend on the time variable $t$ we conclude that the right-hand side of the differential system \eqref{eq:CP} is continuous in $t$. Thus, $\frac{d}{dt}f^\p_j$ and $\frac{d}{dt}f^\q_j$ are continuous in time $\forall\;j$. Therefore, $\left(\f^\p,\f^\q\right)\in C^1\left(\mathbb{R}^+;\,\mathbb{R}^{n^\p}\right)\times C^1\left(\mathbb{R}^+;\,\mathbb{R}^{n^\q}\right)$ and a classical solution of the problem exists. Iterating this consideration, $\left(\f^\p,\f^\q\right)$ are infinitely differentiable in time.
		\qedhere
	\end{description}
\end{proof}

Next we investigate the structure of the equilibria resulting from the ODE system~\eqref{eq:explicit-discretesys}. In particular, we prove prove that equilibria of each $\p$-class are uniquely determined by the initial densities $\rho^\p$, $\rho^\q$, $\forall\;\q\in\neg\p$, and that they do not depend on the number of grid points $\np$. Observe that a fixed value of $s\in[0,1]$ can be obtained with different values of the initial densities $\rho^\p$, $\rho^\q$, $\forall\;\q\in\neg\p$, thus the equilibria will not be uniquely determined once $s$ is chosen.

Moreover, the theorem shows that for any $s\in[0,1]$ the number of nonzero asymptotic distribution functions is determined by $T^\p=\frac{\vmp}{\Dv}$. More precisely, each $\p$-class of vehicles has exactly $T^\p+1$ non-vanishing equilibria, which are related to the cells $I_1^\p$, $I_{r+1}^\p$, $I_{2r+1}^\p,\dots,I_{\np}^\p$.

\begin{theorem}[Stable equilibria]
\label{th:equilibria}
	For any fixed $\Dv$, let $(\f_r)^\p$ denote the equilibrium of the ODE system~\eqref{eq:explicit-discretesys} related to the $\p$-class and obtained on the grid with spacing $\dv$ given by $\Dv=r\dv$ with $r=\frac{\np-1}{T^\p}\in\mathbb{N}$, $\forall\;\p$. Then
	%\[
	%	\left(\f_r\right)^\p_j=\begin{cases} \left(\f_1\right)^\p_{(j-1)/T^\p} & \text{mod$(j-1,T^\p)=0$}\\
	%	0 & \text{otherwise} \end{cases}
	%\]
	\[
		\left(\f_r\right)^\p_j=\begin{cases} \left(\f_1\right)^\p_{\left\lceil \frac{j}{r} \right\rceil } & \text{if mod$(j-1,r)=0$}\\
		0 & \text{otherwise} \end{cases}
	\]
	is the unique stable equilibrium for all classes of vehicles $\p$ and the values of $\left(\f_1\right)^\p$ depend uniquely on the initial densities.
\end{theorem}
\begin{proof}
To prove the statement, we compute explicitly the equilibrium solutions of the discretized model, using the explicit expression of the collision kernel given in~\eqref{eq:Jpp:allj}, \eqref{eq:Jpq:allj2} and~\eqref{eq:Jpq:allj1}. For the sake of simplicity we again consider only two populations and we indicate the related discrete distributions as $\fub$ and $\fdb$, with $\mathcal{V}^{1}\supset\mathcal{V}^{2}$.

The equilibrium equations resulting from $\frac{d}{dt}\fp_j=0$, $\forall\;j=1,\dots,\np$, $\forall\;\p\in\{1,2\}$, are quadratic functions of $\fp_j$ and we prove that for any $j$ they depend only on the previous $j-1$ equilibrium values. In order to find the stable equilibrium, we recall that it is the larger root of the quadratic equation if its leading coefficient is negative, while it is the smaller root otherwise. This consideration will be applied several times during the proof.

For $j=1$, the equation $\frac{d}{dt}\fp_1=0$ is computed by means of the expressions~\eqref{eq:Jppsmallj}-\eqref{eq:Jpqsmallj1} if $\p=1$  and~\eqref{eq:Jppsmallj}-\eqref{eq:Jpqsmallj2} if $\p=2$. Using the fact that $\sum_k \fp_k=\rho^\p$, $\forall\;\p\in\{1,2\}$, we obtain
\begin{equation}\label{eq:f1}
	\frac{d}{dt} \fp_1=0 \, \Leftrightarrow \, \left(\frac{3P-2}{2}\right) \left(\fp_1\right)^2 + \left[\left(\frac{3P-2}{2}\right)\fq_1 + (1-2P)\rho^\p-P\rho^\q\right]\fp_1+(1-P)\fq_1\rho^\p=0,\;\forall\;\p
\end{equation}
which is a quadratic equation for $\fp_1$, $\forall\;\p$. In order to define the asymptotic expression of $\fp_1$, we consider the sum of the equation~\eqref{eq:f1} for $\p=1$ and $\p=2$:
\begin{equation}\label{eq:sumf1}
	\left(\frac{3P-2}{2}\right) \left(\fu_1+\fd_1\right)^2+(1-2P)\left(\ru+\rd\right)\left(\fu_1+\fd_1\right)=0
\end{equation}
which has two real roots, $\fu_1+\fd_1=0$ and $\fu_1+\fd_1=2\frac{2P-1}{3P-2}\left(\ru+\rd\right)$. It is easy to see that one solution is stable and the other one unstable, depending on the value of $P$. More precisely, we find that the stable one is
\begin{equation}\label{eq:stable-sumf1}
	\fu_1+\fd_1=\begin{cases} 0 & P\geq\frac12\\
					   2\frac{2P-1}{3P-2}\left(\ru+\rd\right) & P<\frac12\end{cases}.
\end{equation}

Since solutions of~\eqref{eq:explicit-discretesys} are non-negative (see Theorem~\ref{th:well.posedness}) the equilibrium distributions of each population are
\[
	\fu_1=\fd_1=0, \quad \text{if } P\geq \frac12.
\]
On the other hand, if $P<\frac12$, since $\rho^{1}$ and $\rho^{2}$ are arbitrary, the equilibrium of the sum is equal to the sum of the steady states of each vehicle class, so that
\[
	\fu_1=2\frac{2P-1}{3P-2}\ru, \quad \fd_1=2\frac{2P-1}{3P-2}\rd
\]
and by substituting these quantities into the equation~\eqref{eq:f1} we find indeed that the equation is satisfied. Moreover, since they are positive provided $P<\frac12$ and the roots of~\eqref{eq:f1} have opposite sign if $P<\frac12$, the stable equilibrium solutions are
\begin{equation}\label{eq:stable-f1}
	\left(\f_r\right)^\p_1=\begin{cases} 0 & P\geq\frac12\\
					   2\frac{2P-1}{3P-2}\rho^\p & P<\frac12\end{cases}, \quad \forall\;\p\in\{1,2\}.
\end{equation}
Thus, no vehicle is in the lowest speed class $I^\p_1$ if $P\geq\frac12$.
%, which, for the case $P=1-s^{\gamma}$ means that all vehicles belonging to the class $\p$ are moving if $s<\left(\frac12\right)^{\frac{1}{\gamma}}$.

For $2 \leq j \leq r$, the equilibrium equation of each population is
\begin{gather*}
	\left(\frac{3P-2}{2}\right) \left(f_j^\p\right)^2 + \left[ (3P-2)\sum_{k=1}^{j-1} f_k^\p + (2P-1)\sum_{k=1}^{j-1} f_k^\q + \left(\frac{3P-2}{2}\right) f_j^\q + (1-2P)\rho^\p - P\rho^\q\right] f_j^\p \\
	+ (1-P)f_j^\q\left(\rho^\p-\sum_{k=1}^{j-1} f_k^\p\right) = 0, \quad \forall\;\p
\end{gather*}
and summing this for $\p=1$ and $\p=2$ we obtain
\[
	\left(\frac{3P-2}{2}\right) \left(\fu_j+\fd_j\right)^2 + \left[(3P-2)\sum_{k=1}^{j-1}\left(\fu_k+\fd_k\right) + (1-2P)\left(\ru+\rd\right)\right]\left(\fu_j+\fd_j\right)=0.
\]
Start from $j=2$. Clearly, for $P\geq\frac12$, substituting the related equilibrium given in equation~\eqref{eq:stable-sumf1}, the stable root is again $\fu_2+\fd_2=0$. For $P<\frac12$, the equation for $\fu_2+\fd_2$, with $\fu_1+\fd_1$ given by~\eqref{eq:stable-sumf1}, becomes
\[
	\left(\frac{3P-2}{2}\right) \left(\fu_2+\fd_2\right)^2 - (1-2P)\left(\ru+\rd\right)\left(\fu_2+\fd_2\right)=0.
\]
Comparing with the equation~\eqref{eq:sumf1}, we see that now the stable root is $\fu_2+\fd_2=0$. Thus, recalling that the solutions of~\eqref{eq:explicit-discretesys} are positive in time, the equilibrium values of each class of vehicles are $\left(\f_r\right)_2^\p\equiv 0$, $\forall\;P\in[0,1]$. Analogously, it is easy to prove that $\left(\f_r\right)_j^\p\equiv 0$, $\forall\;j=3,\dots,r$.

Consider now $r+1 \leq j \leq 2r$, the equilibrium equations $\frac{d}{dt}\fp_j=0$ is computed by using the self- and cross-collision terms given in~\eqref{eq:Jppallj}-\eqref{eq:Jpqmidsmallj1} if $\p=1$ and in~\eqref{eq:Jppallj}-\eqref{eq:Jpqallj2} if $\p=2$. Therefore, the equations will contain four extra terms:
\[
	\frac{P}{2} \left(\fp_{j-r}\right)^2 + P\fp_{j-r}\sum_{k=j-r+1}^{\np} f_k^\p + \frac{P}{2} \fp_{j-r}\fq_{j-r} + P\fp_{j-r}\sum_{k=j-r+1}^{\nq}\fq_k
\]
and they write as
\begin{equation}\label{eq:fmidj}
\begin{split}
	\left(\frac{3P-2}{2}\right)\left(\fp_j\right)^2 + \left[(3P-2)\sum_{k=1}^{j-1}\fp_k + (2P-1)\sum_{k=1}^{j-1}\fq_k+\left(\frac{3P-2}{2}\right)\fq_j+(1-2P)\rho^\p-P\rho^\q\right]\fp_j \\
	+(1-P)\fq_j\left(\rho^\p-\sum_{k=1}^{j-1} \fp_k\right) + P \fp_{j-r}\left[ -\frac12\left(\fp_{j-r}+\fq_{j-r}\right)+\rho^\p+\rho^\q-\sum_{k=1}^{j-r-1}\left(\fp_k+\fq_k\right)\right]=0,\;\forall\;\p.
\end{split}
\end{equation}
If $j=r+1$ and $P\geq\frac12$, then from the previous step $\fp_k=0$, $\forall\;k=1,\dots,r$ and $\forall\;\p$. Therefore, the new terms are certainly zero and the equation is identical to~\eqref{eq:f1}, so $\left(\f_r\right)^\p_{r+1}\equiv 0$ $\forall\;\p$. Instead, if $P<\frac12$, then $\fp_k=0$, $\forall\;k=2,\dots,r$ and $\forall\;\p$, thus the equation~\eqref{eq:fmidj} reduces to
\begin{equation}\label{eq:fr+1}
\begin{split}
	\left(\frac{3P-2}{2}\right)\left(\fp_{r+1}\right)^2 + \left[(3P-2)\fp_1 + (2P-1)\fq_1+\left(\frac{3P-2}{2}\right)\fq_{r+1}+(1-2P)\rho^\p-P\rho^\q\right]\fp_{r+1} \\
	+(1-P)\fq_{r+1}\left(\rho^\p-\fp_1\right) + P \fp_1\left[ -\frac12\left(\fp_1+\fq_1\right)+\rho^\p+\rho^\q\right]=0,\;\forall\;\p.
\end{split}
\end{equation}
Summing this equation for $\p=1$ and $\p=2$
%we find
%\begin{gather*}
%	\left(\frac{3P-2}{2}\right)\left(\fu_{r+1}+\fd_{r+1}\right)^2 + \left[ (3P-2)\left(\fu_1+\fd_1\right)+(1-2P)\left(\ru+\rd\right)\right]\left(\fu_{r+1}+\fd_{r+1}\right)\\
%	+P\left(\fu_1+\fd_1\right)\left[-\frac12\left(\fu_1+\fd_1\right)+\rho^\p+\rho^\q\right]=0
%\end{gather*}
and substituting the expression for $\fu_1+\fd_1$ we obtain
\begin{gather*}
	\left(\frac{3P-2}{2}\right)\left(\fu_{r+1}+\fd_{r+1}\right)^2 + (2P-1)\left(\ru+\rd\right)\left(\fu_{r+1}+\fd_{r+1}\right)\\
	+2P\frac{(P-1)(2P-1)}{(3P-2)^2}\left(\ru+\rd\right)^2=0.
\end{gather*}
The resulting stable equilibrium solution is $\fu_{r+1}+\fd_{r+1}=\left(\ru+\rd\right)\frac{(1-2P)-\sqrt{\Delta_{r+1}}}{3P-2}$ where $\Delta_{r+1}=(2P-1)\left[(2P-1)-\frac{4P(P-1)}{3P-2}\right]$ is positive provided $P<\frac12$. The equilibria of population $\p$ can be found recalling that they depend only on the quantity $\rho^\p$, so that
\[
	\fu_{r+1}=\ru\frac{(1-2P)-\sqrt{\Delta_{r+1}}}{3P-2}, \quad \fd_{r+1}=\rd\frac{(1-2P)-\sqrt{\Delta_{r+1}}}{3P-2}
\]
and it can be checked that these are the positive solutions of~\eqref{eq:fr+1}. Thus, if $j=r+1$ the stable equilibrium values are
\[
	\left(\f_r\right)^\p_{r+1}=\begin{cases} 0 & P\geq\frac12\\
							   \rho^\p\frac{(1-2P)-\sqrt{\Delta_{r+1}}}{3P-2} & P<\frac12
					 \end{cases}, \quad \forall\;\p.
\]
If, instead, $r+1 < j \leq 2r$, then $\fp_{j-r}\equiv 0$, $\forall\;P\in[0,1]$ and $\forall\;\p$, so the equilibrium equation resulting from $\frac{d}{dt}\left(\fu_j+\fd_j\right)=0$ is identical to~\eqref{eq:sumf1} if $P\geq\frac12$, while it is
\[
	\left(\frac{3P-2}{2}\right)\left(\fu_j+\fd_j\right)^2-\sqrt{\Delta_{r+1}}\left(\ru+\rd\right)\left(\fd_j+\fu_j\right)=0
\]
if $P<\frac12$. Then, we obtain $\left(\f_r\right)_j^\p \equiv 0$, for $j=r+2,\dots,2r$ and $\forall\;P\in[0,1]$.

Clearly, this procedure can be repeated, in fact, the cases $j=1,\dots,2r$ that we just computed are typical. Also for larger values of $j$ we find a quadratic equation for the unknown $f_j^\p$, which involves only previously equilibrium values, hence we can easily find successively all components of $\left(\f_r\right)^\p$, $\forall\;\p$. The values of the asymptotic solutions are listed in Appendix~\ref{app:equilibria}.
\end{proof}

\begin{figure}[t!]
\centering
\includegraphics[width=0.45\textwidth]{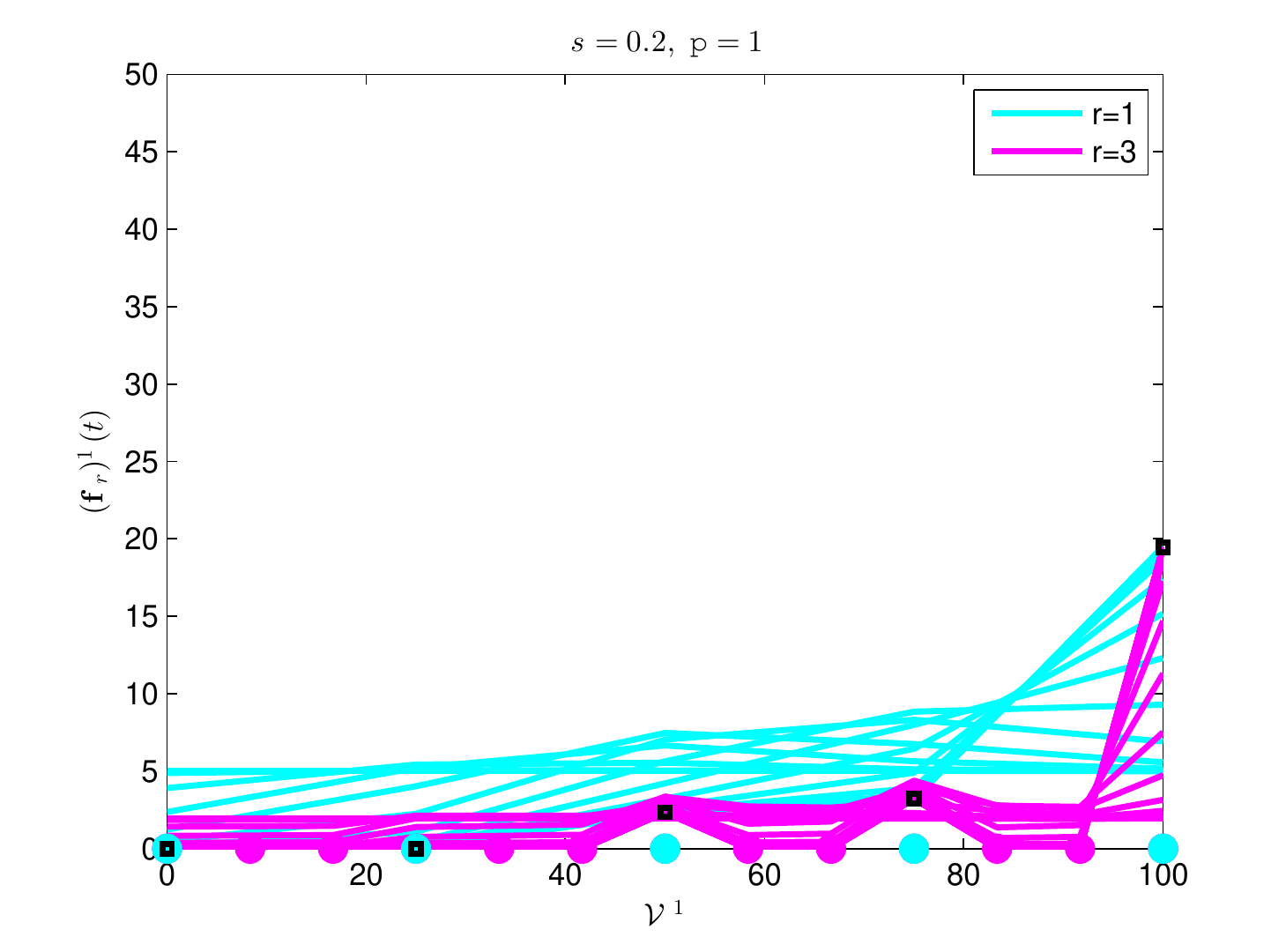}
\includegraphics[width=0.45\textwidth]{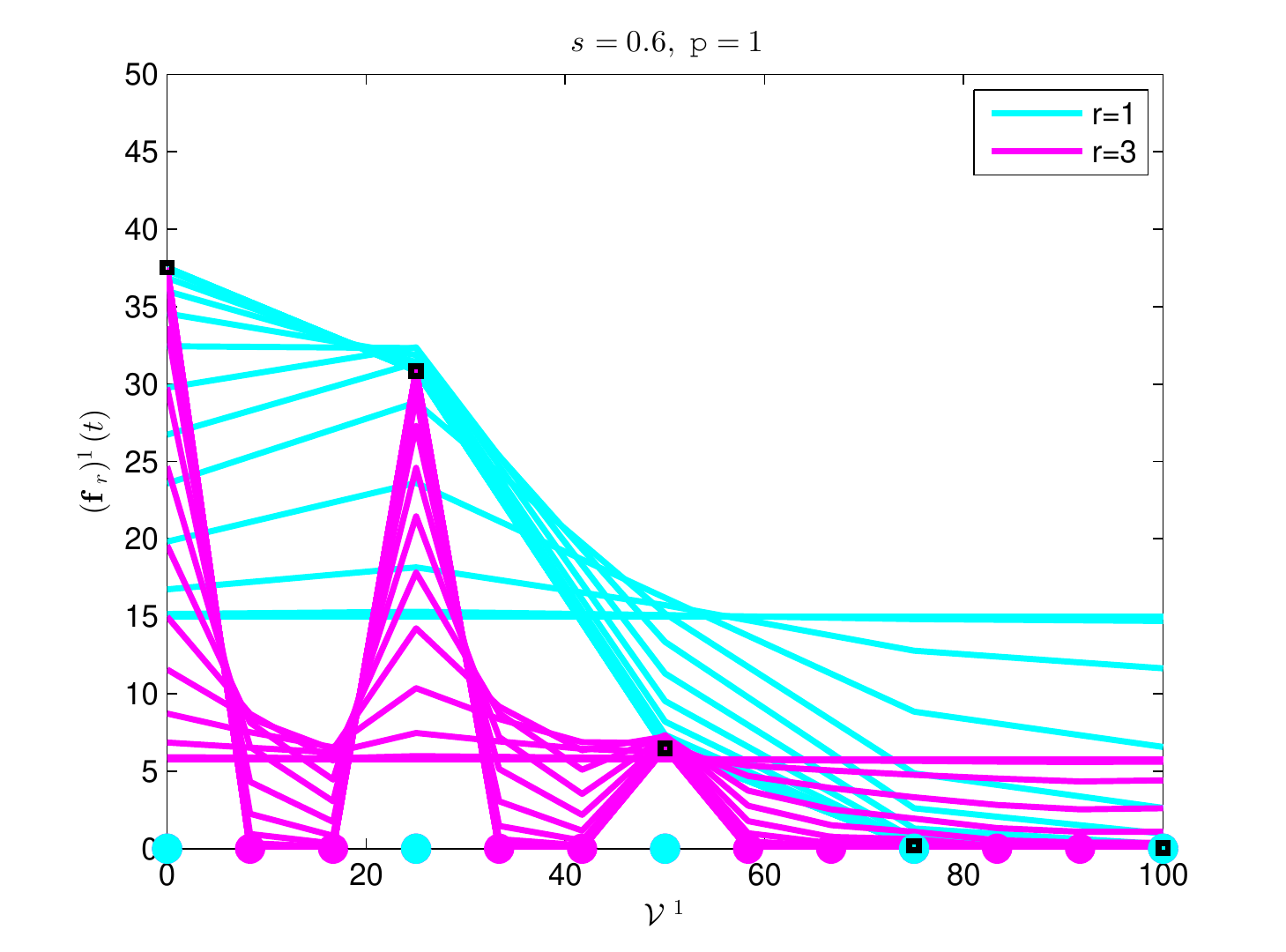}
\includegraphics[width=0.45\textwidth]{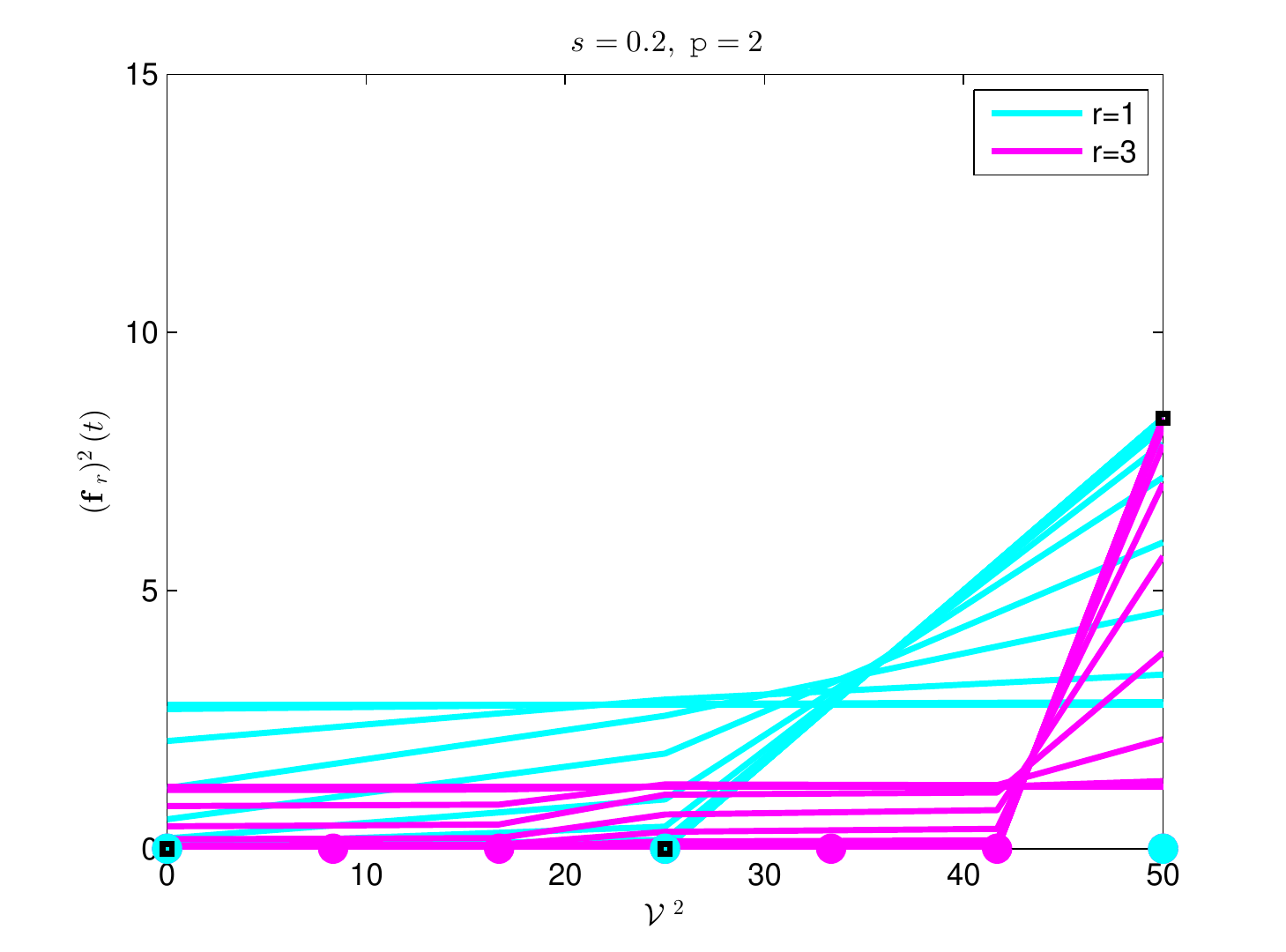}
\includegraphics[width=0.45\textwidth]{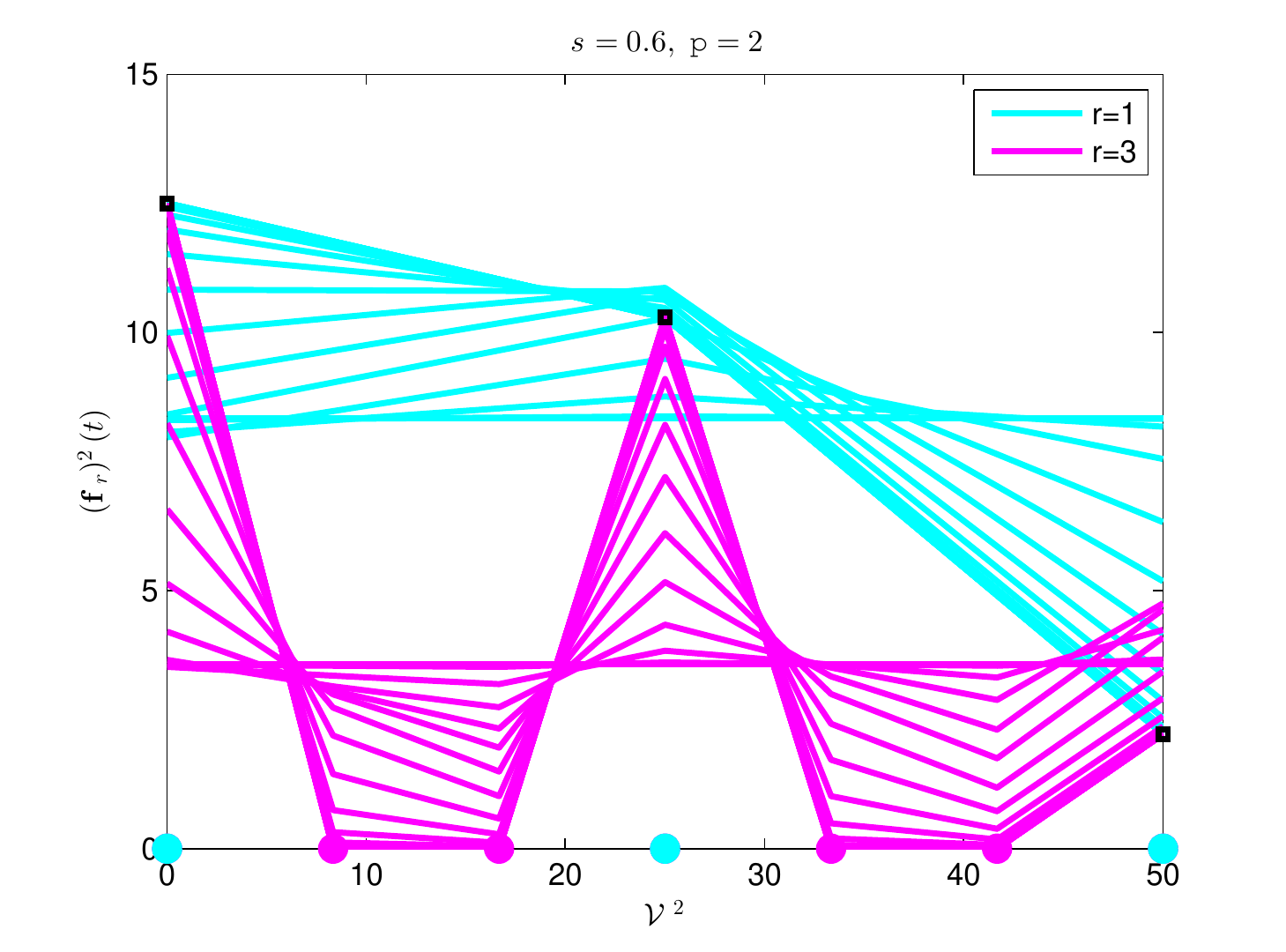}
\caption{Evolution towards equilibrium of the ODE system~\eqref{eq:explicit-discretesys} for the case of two populations, $\p=1$ (top panels) and $\p=2$ (bottom panels), with a fixed value of the velocity jump $\Dv=25~\unit{km/h}$. The maximum velocities are $\vm^{1}=100~\unit{km/h}$ and $\vm^{2}=50~\unit{km/h}$. The velocity grid is obtained with $r=1$ (cyan) and $r=3$ (magenta), which correspond to $n^{1}=5,13$ and $n^{2}=3,7$ grid points. Black squares indicate the equilibrium values.\label{fig:equilibrium}}
\end{figure}

Theorem~\ref{th:equilibria} ensures that the equilibrium values of the discretized model~\eqref{eq:explicit-discretesys} can be found on the coarser grid, namely taking $r=1$. The result can be better appreciated by looking at the evolution towards equilibrium shown in Figure~\ref{fig:equilibrium}. In this figure, we consider two populations such that $\vm^{1}=100~\unit{km/h}$ and $\vm^{2}=\vm^{1}-2\Dv=50~\unit{km/h}$. Then, $T^{1}=4$ and $T^{2}=2$. The different plots show the evolution towards equilibrium, starting from uniform initial distributions for $r=1$ (cyan) and $r=3$ (magenta), which correspond to $\np=5,13$ velocity grid points for the population $\p=1$ (top panels) and $\np=3,7$ velocity grid points for the population $\p=2$ (bottom panels). The left plots are obtained with the fraction of occupied space $s=0.2$, while the right plots with $s=0.6$. Note that different dynamics towards equilibrium are observed, for different values of the number $\np$, $\p=1,2$, of grid points, but as equilibrium is approached, the values of the equilibria go to zero, except for the cells $I_j^\p$, $\p=1,2$, corresponding to integer multiples of $\Dv$. In fact, for $r=3$, the nonzero equilibrium values, marked with black squares, are related to the velocities of the grid with $r=1$ and marked with cyan filled circles. While the additional velocities of the refined grid, marked with magenta filled circles, correspond to zero values of the equilibria.

\begin{remark}[$\Dv$ depending on $\p$]
	In Section \ref{sec:discrete-model} we have assumed that the speed jump $\Dv$ is independent of $\p$ so that it is a fixed parameter for all classes of vehicles. Taking the velocity jump dependent on $\p$ would mean to consider another microscopic difference characterizing the types of vehicles. More precisely, we can use it in order to model the subjective behavior of drivers, since we can think of $\Dv$ as a parameter describing different types of drivers, more or less aggressive depending on the value of the jump in velocity.\\Even with $\Dv$ dependent on $\p$, equilibrium solutions preserve the property of being quantized that is non-zero on a reduced number of discrete velocities. Let us consider the same example shown in the right panels of Figure \ref{fig:equilibrium} with two populations, $\p=1,2$, having maximum velocities $\vm^{1}=100~\unit{km/h}$ and $\vm^{2}=50~\unit{km/h}$. Now we take two different velocity jumps, $\Dv^1=20~\unit{km/h}$ and $\Dv^2=10~\unit{km/h}$ in order to model the case in which population $\p=1$ is more aggressive than population $\p=2$. We choose the discretization parameters $r^1$ and $r^2$ in such a way that the velocity grids of the two populations have the same spacing $\dv=\vm^1/r^1=\vm^2/r^2$. This means that we have to take $r^1=2r^2$.
	\begin{figure}[t!]
		\centering
		\includegraphics[width=0.45\textwidth]{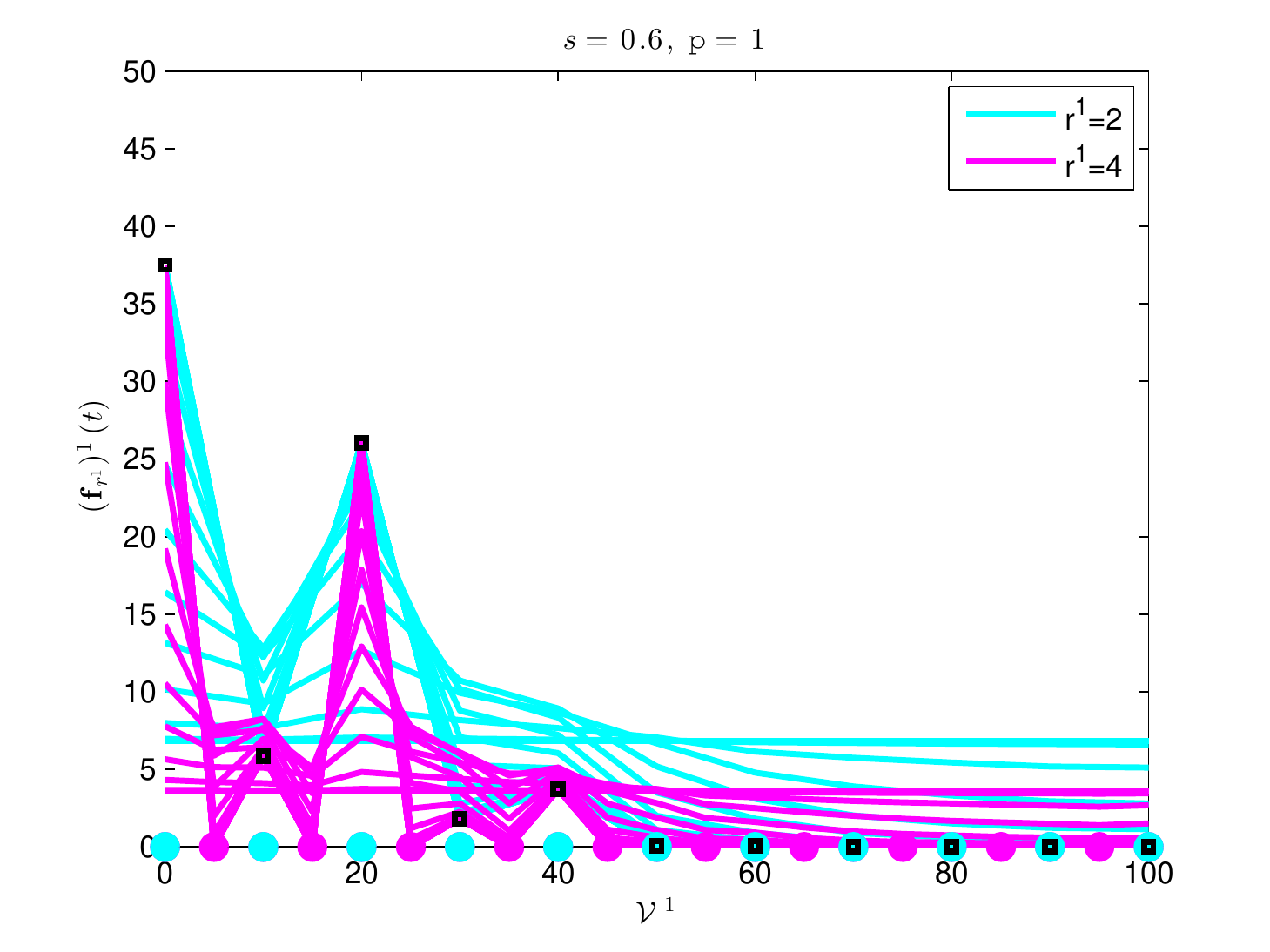}
		\includegraphics[width=0.45\textwidth]{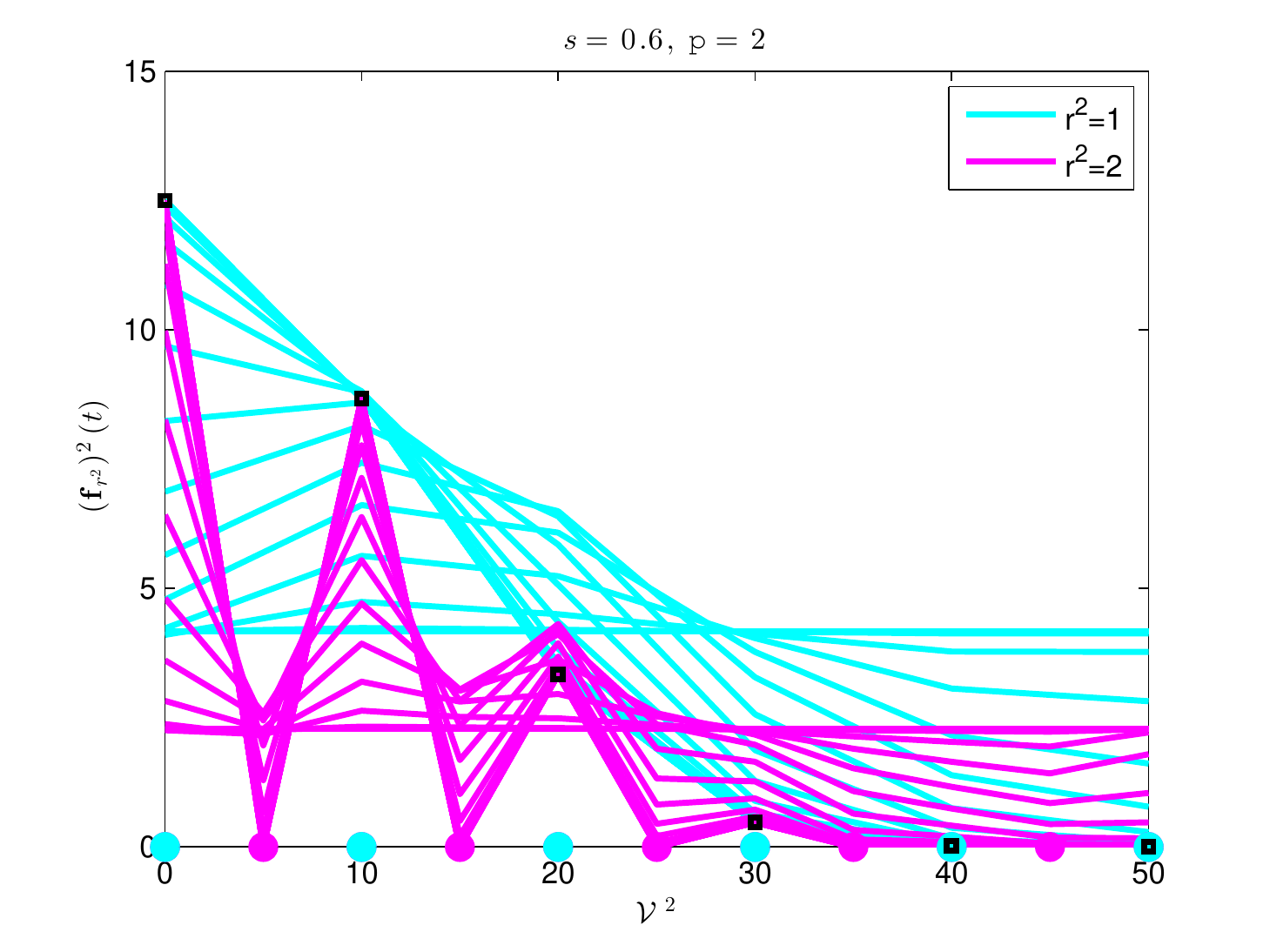}
		\caption{Evolution towards equilibrium of the ODE system~\eqref{eq:explicit-discretesys} for the case of two populations, $\p=1$ (left panels) and $\p=2$ (right panels), with two different values of the velocity jump $\Dv^1=20~\unit{km/h}$ and $\Dv^2=10~\unit{km/h}$. The maximum velocities are $\vm^{1}=100~\unit{km/h}$ and $\vm^{2}=50~\unit{km/h}$. The velocity grid is obtained with $r^1=2$, $r^2=1$ (cyan) and $r^1=4$, $r^2=2$ (magenta). Black squares indicate the equilibrium values.\label{fig:equilibrium:differentDV}}
	\end{figure}
	In Figure \ref{fig:equilibrium:differentDV} we show the evolution towards equilibrium for both populations $\p=1$ (left plot) and population $\p=2$ (right plot). We start from uniform initial distributions and we consider two cases: first we choose $r^2=1$ (and therefore $r^1=2$) (cyan), then we choose $r^2=2$ (and therefore $r^1=4$) (magenta). These values correspond to $n^1=13,25$ and $n^2=6,11$ velocity grid points. Again we observe that equilibria are quantized. Precisely only the values related to the cells $I_j^\p$, $\p=1,2$, corresponding to integer multiples of $\Dv^2$, i.e. the minimum of the two jumps, give a non-zero contribution irrespective of grid refinements. In fact, the equilibrium values of population $\p=1$ are quantized according to the velocity jump of population $\p=2$.\\Notice that in this case $\Dv^1$ is an integer multiple of $\Dv^2$ but the same phenomenon verifies in a more general situation.
\end{remark}

\begin{remark}[Unstable equilibria] Theorem~\ref{th:equilibria} gives the uniqueness of the stable equilibria of the
model. However, unstable ones may occur if the initial condition is such that $\fp_1(0) = 0$, $\forall\;\p$. In fact, the interaction rules in the case $\vb>\va$ do not allow for a post-interaction velocity $v\in\Vp$ which is lower than $\va$. Thus if $\fp_1(0) = 0$, i.e. there are no vehicles of the $\p$-class with velocity $v_1$ at the initial time, interactions will not lead to an increase of $\fp_1$. In this sense, the equilibrium solution of the multi-population model does not only depend on the initial densities, but also on the initial condition $\fp(0,v)$ because ``spurious'' equilibria on sub-manifolds of the state space may appear. However, as showed in~\cite{PgSmTaVg2}, these solutions are unstable: a small perturbation of $\fp_1(0)$ is enough to trigger the evolution towards the stable equilibrium, which depends only on the initial densities.
\end{remark}

%\begin{remark}[Indifferentiability principle] In Theorem~\ref{th:equilibria} the equilibria related to the sum of the distribution functions are identical to those computed in~\cite{PgSmTaVg2} for the single population model. This shows that the indifferentiability principle holds at equilibrium. However, Theorem~\ref{th:indiff-principle}, proved in Section~\ref{sec:delta-multimodel}, is more general because it states that the indifferentiability principle for the continuous multi-population model holds at all times, moreover without
%having to merge the self- and cross-collision terms into one. In fact, in~\cite{andries2001REPORT} the indifferentiability principle at all times is proved for a model featuring a single collision operator, which hinders the description of cross-interactions among particles of different species in the mixture. Instead, in more standard models for gas mixtures in which the collision terms are separate, the indifferentiability principle holds only at equilibrium.
%\end{remark}

\section{Fundamental diagrams of traffic}
\label{sec:funddiag}

In this Section we present the fundamental diagrams obtained with the multi-population kinetic model described in this work. We focus on the case of a mixture composed of three types of vehicles. Such diagrams provide the relation between flux and density at the macroscopic level.

Macroscopic variables for the $\p$-class of vehicles are recovered computing moments of the distribution functions $\fp$, see equations~\eqref{eq:macro-var}. However, fundamental diagrams are obtained by assuming that traffic is in equilibrium. Therefore, we compute the quantities~\eqref{eq:macro-var} using the asymptotic distributions $\left(\f_r\right)^\p$. Since for each population $\p$ the $\left(\f_r\right)^\p$'s depend uniquely on the densities of all vehicle classes, at equilibrium we have that the flux is a function of the initial densities and it is computed as
\begin{align}\label{eq:discrete-flux}
	q^\p(r)&=\int_{\Vp} v \left(\fp\right)^{\infty}(v) dv \approx \int_{\Vp} v \sum_{j=1}^{\np} \left(\f_r\right)^\p_j \frac{\chi_{I_j^\p}(v)}{\abs{I_j^\p}}\nonumber\\
	&=\left(\f_r\right)_1^\p v_1^\p + \sum_{j=2}^{\np-1} \left(\f_r\right)_j^\p v_j^\p + \left(\f_r\right)_{\np}^\p v_{\np}^\p
\end{align}
where the velocities $v_j^\p$ are the mid points of the cells $I_j^\p$ resulting with the grid size $\dv=\frac{\Dv}{r}$ and the $\left(\f_r\right)_j^\p$'s are the equilibrium solutions. We recall that the $v_j^\p$'s for $j=2,\dots,\np$ do not change when the grid is refined, i.e. when $r$ increases, while the first and last velocity grid point approach zero and the maximum velocity of the $\p$-class, respectively, when $r\to\infty$. Recall that Theorem~\ref{th:equilibria} states that the equilibrium values of the discretized model do not depend on the discretization, i.e. on the grid size $\dv$ and on the discrete speeds $v_j^\p$'s, but only on the velocity jump $\Dv$. Therefore, the flux can be computed exactly using the values $\left(\f_1\right)^\p_j$, $j=1,\dots,T^\p+1=\np$. In other words, for the case of two populations we can exploit the explicit formulas for the equilibrium solutions given in Theorem~\ref{th:equilibria} and in Appendix~\ref{app:equilibria}. Instead, for more than two populations we numerically integrate the system~\eqref{eq:explicit-discretesys} on the coarser velocity grid ($r=1$) to recover the equilibrium values and using the results of Theorem~\ref{th:equilibria} we obtain the exact flux as
\begin{equation}\label{eq:exact-flux}
	q^\p(\infty)=\sum_{j=1}^{T^\p+1} \left(\f_1\right)^\p_j (j-1)\Dv.
\end{equation}
Notice that this result is due to the quantized structure of the equilibria, which is in turn a consequence of the assumptions in Section~\ref{sec:discrete-model}. Finally, once the flux is given, the macroscopic speed of the $\p$-class can be obtained as
\begin{equation}\label{eq:speed}
	u^\p(\infty)=\frac{q^\p(\infty)}{\rho^\p}.
\end{equation}
Since in the space homogeneous case each density $\rho^\p$ is constant in time, the fraction of occupied space $s$ remains also constant. Hence we study the total flux $Q=\sum_{\p} q^\p$ and the mean speed $U=\frac{\sum_{\p} q^\p}{\sum_{\p} \rho^\p}$ at equilibrium as functions of the total number of vehicles per unit length $N_{\sf v}=\sum_{\p} \rho^\p$ and of the fraction of occupied space $s$.

In the case of the single population model~\cite{PgSmTaVg2}, the flux and the speed at equilibrium are single-valued functions of the initial density. However, this property does not reflect the structure of the fundamental diagrams provided by experimental data, because such diagrams are multivalued with a wide dispersion of the flux values in the congested phase of traffic, i.e. at high densities. Here, the scattered behavior of the real data will be recovered because traffic is treated as a mixture of more than one population characterized by different physical features. In fact, note that the equilibria related to the $\p$-population, and showed in Theorem~\ref{th:equilibria}, do not only depend on the initial density $\rho^\p$ but also on the values of the occupied space $s$. Thus, since the same value of $s$ can be obtained with different compositions of the traffic mixture, for a given $s\in[0,1]$ we may find different equilibria, hence different flux and speed values at equilibrium, depending on how the road is occupied.

\begin{table}[!t]
\begin{center}
\begin{tabular}{l|c|c|c|c}
& Fast Cars & Slow Cars & Vans & Trucks \\
& ($\p=\Cf$) & ($\p=\Cs$) & ($\p=V$) & ($\p=T$) \\
\hline
\hline
Typical length $l^\p$ & $4~\unit{m}$ & $4~\unit{m}$ & $6~\unit{m}$  & $12~\unit{m}$\\
Max.\,density $\rho^\p_{\max}$ & $250~\unit{veh./km}$ & $250~\unit{veh./km}$ & $166.6~\unit{veh./km}$ & $83.3~\unit{veh./km}$\\
Max.\,velocity $\vmp$ & $120~\unit{km/h}$ & $80~\unit{km/h}$ & $120~\unit{km/h}$ & $80~\unit{km/h}$\\
Velocity jump $\Dv$ & $40~\unit{km/h}$ & $40~\unit{km/h}$ & $40~\unit{km/h}$ & $40~\unit{km/h}$\\
\hline
\end{tabular}
\end{center}
\caption{Physical parameters of the four classes of vehicles chosen for the simulations.\label{tab:physical-parameters}}
\end{table}

In the following, we investigate the properties of the diagrams provided by the multi-population kinetic model. We show that they exhibits different regimes, or phases, of traffic and they reproduce the qualitative structure of experimental diagrams widely analyzed in~\cite{klarReview,FermoTosin14,kerner2004BOOK,PgSmTaVg}.

We introduce four typical classes of vehicles whose characterizing parameters are listed in Table~\ref{tab:physical-parameters} and all simulations are performed by choosing three of them. More precisely we consider {\em Fast Cars-Slow Cars-Trucks}~($\Cf$-$\Cs$-$T$) or {\em Fast Cars-Vans-Trucks}~($\Cf$-$V$-$T$). The diagrams are computed by sampling three random values of the initial densities for any initial fraction of occupied space $s\in[0,1]$. Moreover, recalling the computations already made in~\cite{PgSmTaVg2} to evaluate a physical velocity jump $\Dv$, here we can consider $\Dv=40~\unit{km/h}$. With this choice, the numbers of discrete velocities are $n^{\Cf}=n^{V}=4$ and $n^{\Cs}=n^{T}=3$.

\begin{description}
\item[Free phase of traffic] This traffic regime occurs at low densities, when there is a large distance between vehicles and the interactions are rare. Thus, we expect that the velocity of vehicles is ruled by the maximum allowed speed, which in this framework depends on the mechanical characteristics of the vehicles (e.g. when we assume $\vm^{\Cf}>\vm^{T}$) or on the type of drivers (e.g., when we assume that there are two types of cars such that $\vm^{\Cf}>\vm^{\Cs}$). Therefore, in the free phase of traffic the flux increases nearly linearly with respect to the total density, the data are not widely scattered and are contained in a cone whose upper and lower branch have a slope proportional to $\max_{\p}\{\vmp\}$ and $\min_{\p}\{\vmp\}$ respectively, $\forall\;\p\in\{\Cf,\Cs,V,T\}$.
\begin{figure}[t!]
\centering
\includegraphics[width=0.45\textwidth]{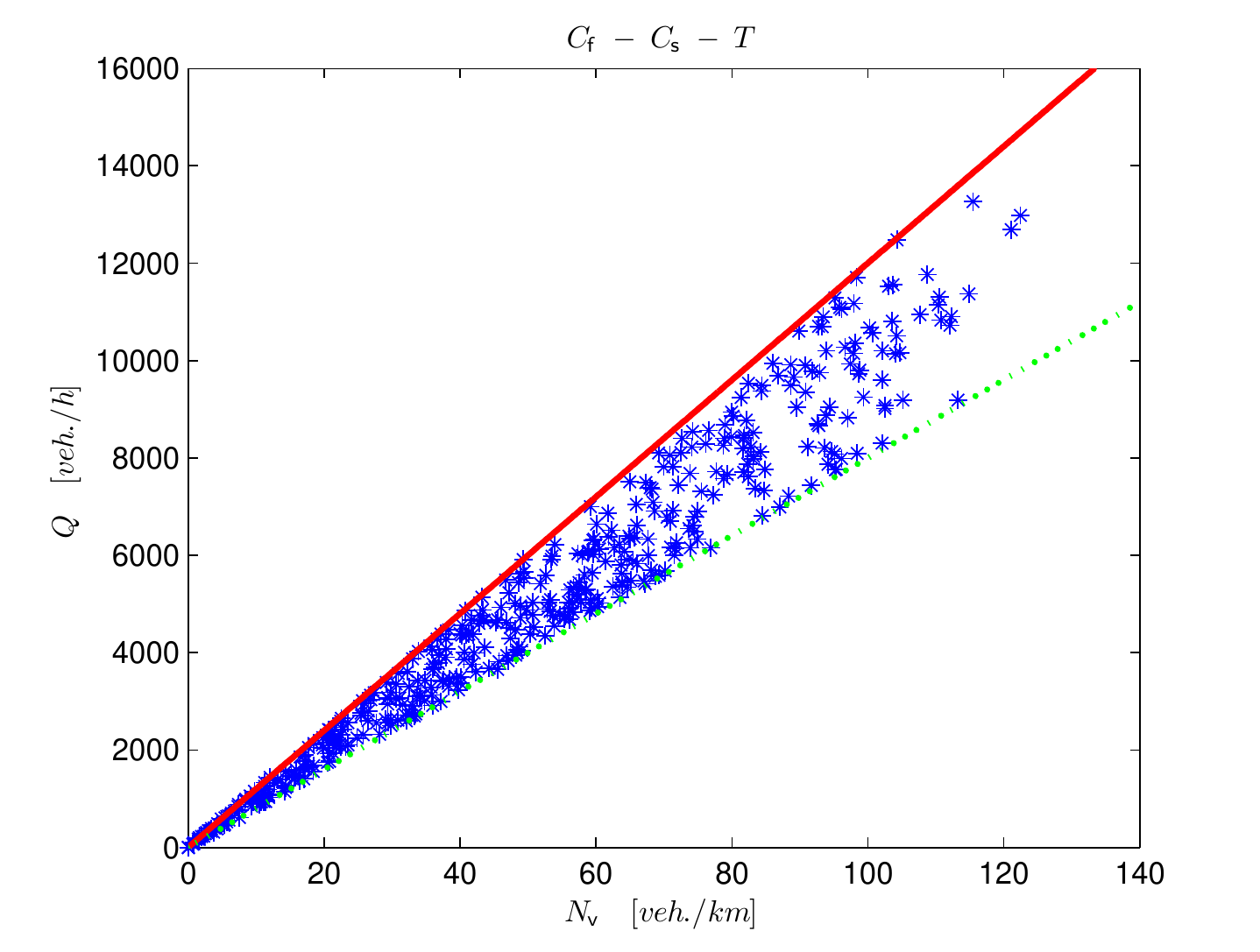}
\includegraphics[width=0.45\textwidth]{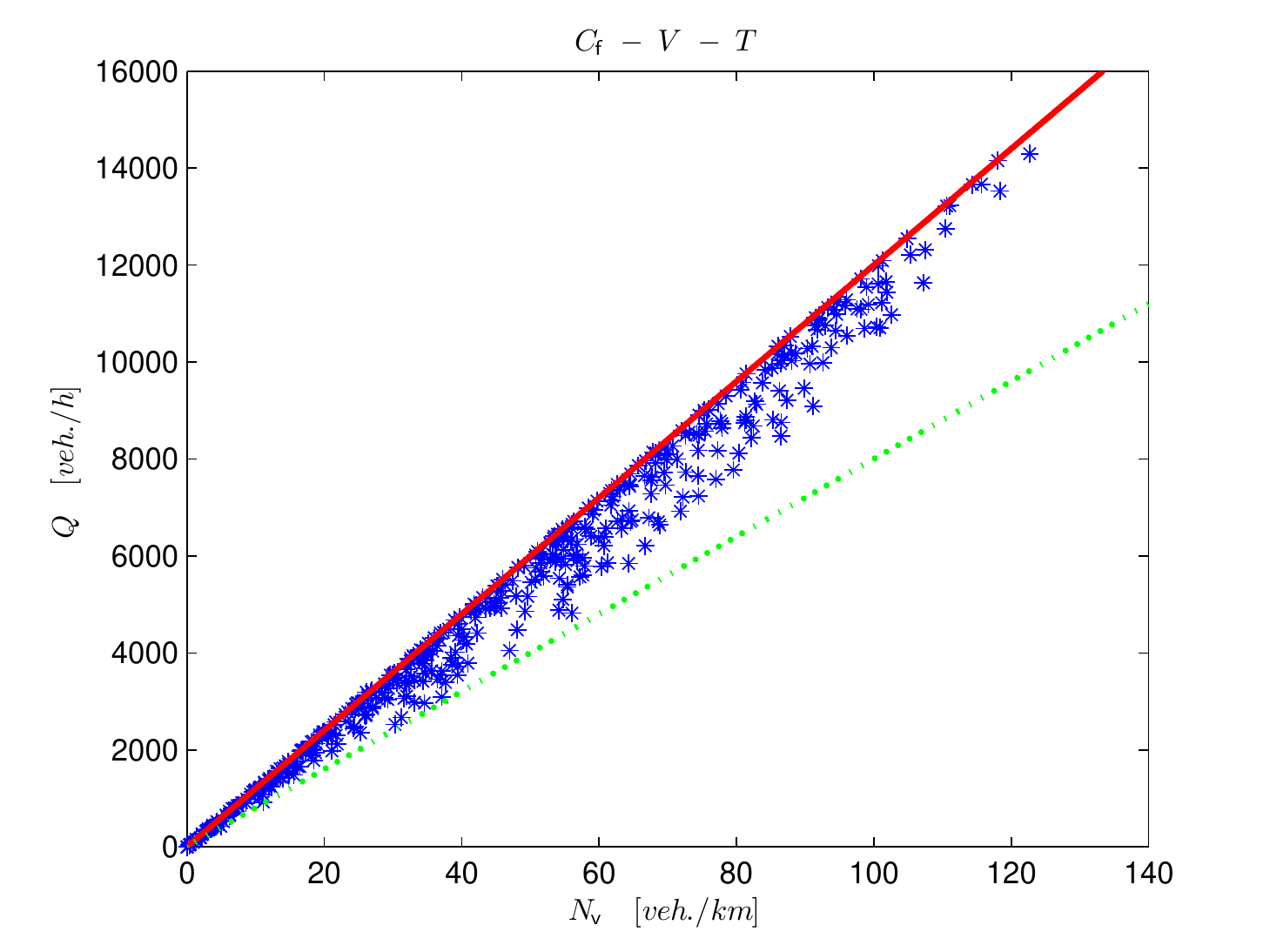}
\\
\includegraphics[width=0.45\textwidth]{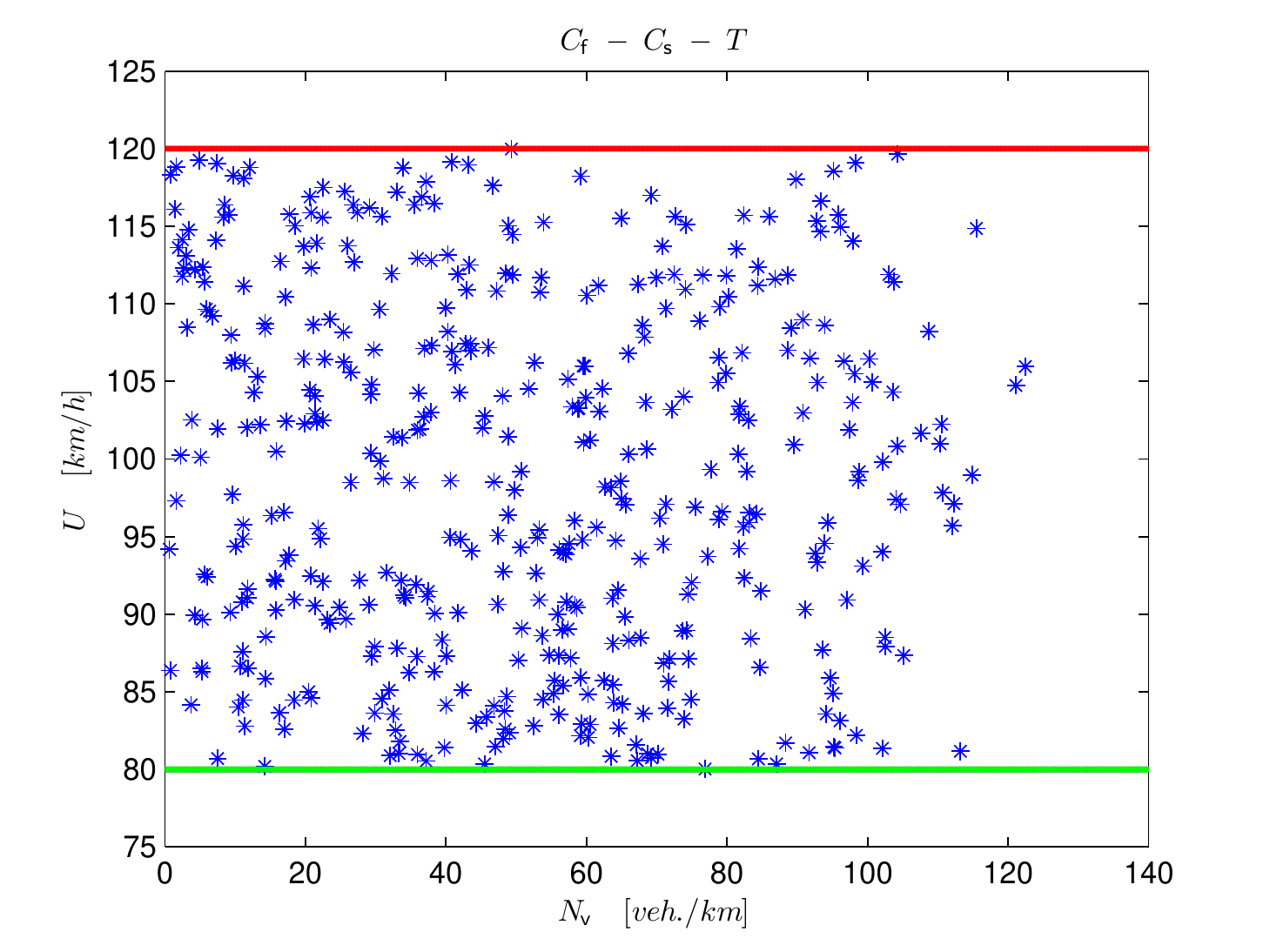}
\includegraphics[width=0.45\textwidth]{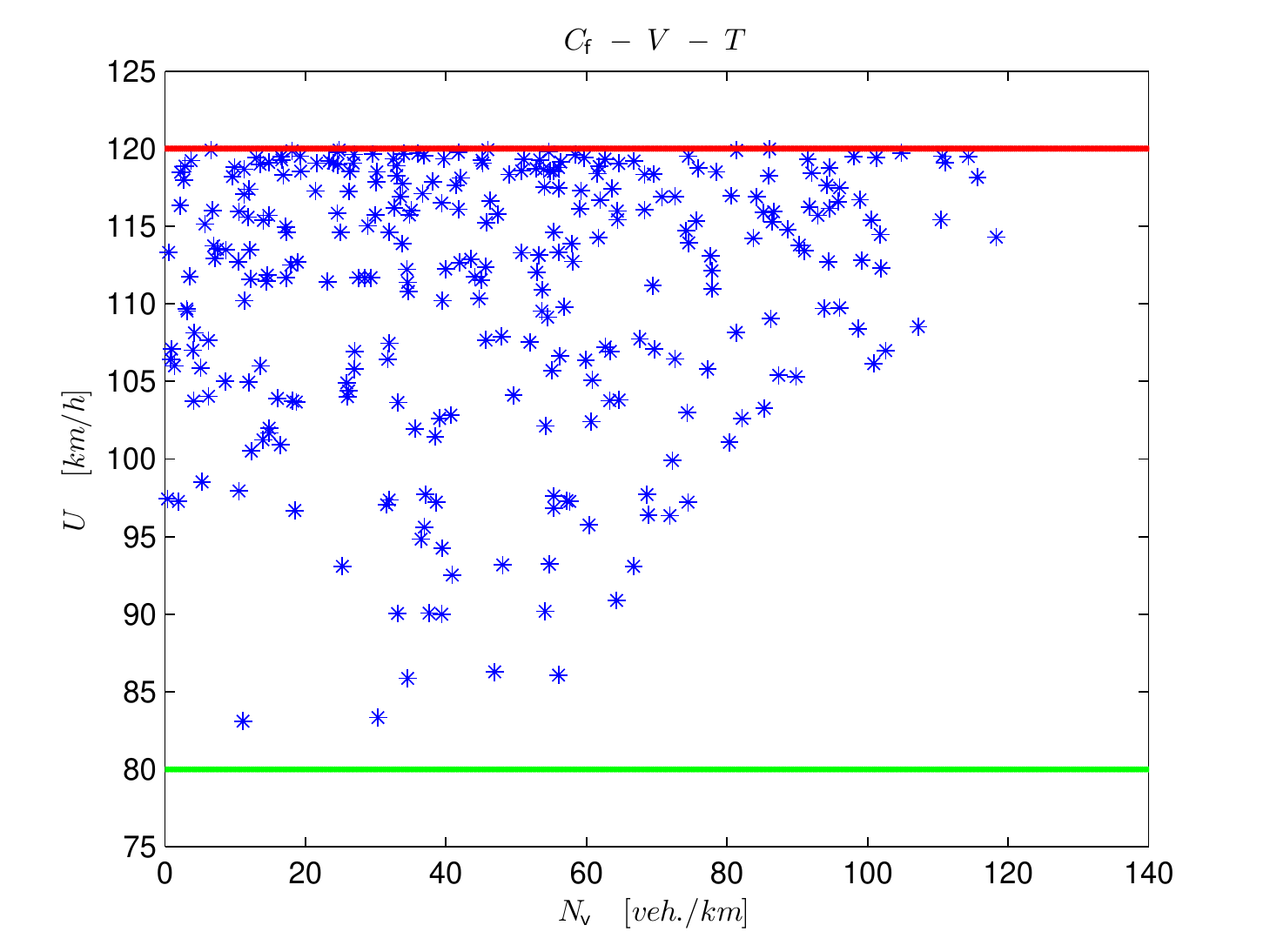}
\caption{Top: free phase of the flux-density diagrams. Bottom: free phase of the speed-density diagrams. On the left we consider the three populations $\Cf$-$\Cs$-$T$, on the right $\Cf$-$V$-$T$. The data are obtained for values of the fraction of road occupancy such that $P\geq\frac12$. In the top panels the solid red line and the dotted green line have slope respectively as the maximum velocity and the minimum velocity of the three populations. The probability $P$ is taken as in~\eqref{eq:gamma-law} with $\gamma=1$.\label{fig:free-phase}}
\end{figure}
For instance, in Figure~\ref{fig:free-phase} we show the free phase of the diagrams provided by the multi-population model with three classes of vehicles. This regime is obtained by taking only the values of the fraction of road occupancy for which $P\geq\frac12$. In fact, as proved in Theorem~\ref{th:equilibria} and as shown in the left panels of Figure~\ref{fig:equilibrium}, this choice produces equilibria of the form
\[
	\f^{\Cf}=\underbrace{[0,0,*,*]}_{n^{\Cf}=4}, \quad \f^{\Cs}=\underbrace{[0,0,\rho^{\Cs}]}_{n^{\Cs}=3}, \quad \f^{V}=\underbrace{[0,0,*,*]}_{n^{V}=4}, \quad \f^{T}=\underbrace{[0,0,\rho^T]}_{n^{T}=3}.
\]
Thus, all classes of vehicles travel at high velocities and in particular the flux of slow cars and trucks is always proportional to the their maximum velocity $\vm^{\Cs}$ and $\vm^{T}$ respectively. Instead, the flux values of fast cars and vans depend also on the maximum velocity of slow cars and trucks. In the left panel of Figure~\ref{fig:free-phase} we consider the test case $\Cf$-$\Cs$-$T$ and we observe that
% the flux values are scattered in a wider region than that resulting in the right panel for the case $C2$-$V$-$T$ and this is due to the fact that the difference between the maximum velocity of the class $C1$ is greater than the maximum velocity of the fastest classes in the test case $C2$-$V$-$T$, so that $\vm^{C1}-\vm^{T}>\vm^{C2}-\vm^{T}=\vm^{V}-\vm^{T}$. Notice also that
the flux values obtained are scattered in the whole cone,
%on the lower branch because there are two populations travelling at speed $80~\unit{km/h}$
in contrast with the case $\Cf$-$V$-$T$ shown in the right panels in which both fast cars and vans travel at speed $120~\unit{km/h}$ and therefore the flux values are mainly distributed on the upper branch. This consideration can be reinforced by looking at the macroscopic speed diagrams, bottom panels of Figure~\ref{fig:free-phase}.

\item[Phase transition] It represents the transition between the free and the congested phase of traffic. The value of the fraction of occupied space $s$ at which the phase transition occurs is called {\em critical value}. The flux is maximum when this value is reached. If $s$ increases then we observe a decrease of the flux and of the mean speed in the diagrams of traffic. From a mathematical point of view, the phase transition occurs when there is the bifurcation of the equilibrium values, that is when $P$ becomes smaller than $\frac12$. In fact, when $P\geq\frac12$ all vehicles are moving and only when $P<\frac12$ the lower speed classes begin to fill up, see equation~\eqref{eq:stable-f1} in Theorem~\ref{th:equilibria}. Since $P$ is a function of $s$, the choice of the probability $P$ influences the critical value. In order to investigate this phenomenon, in Figure~\ref{fig:s-diagrams} we consider the total diagram of the flux with respect to the fraction of occupied space $s$.
\begin{figure}[t!]
\centering
\includegraphics[width=0.45\textwidth]{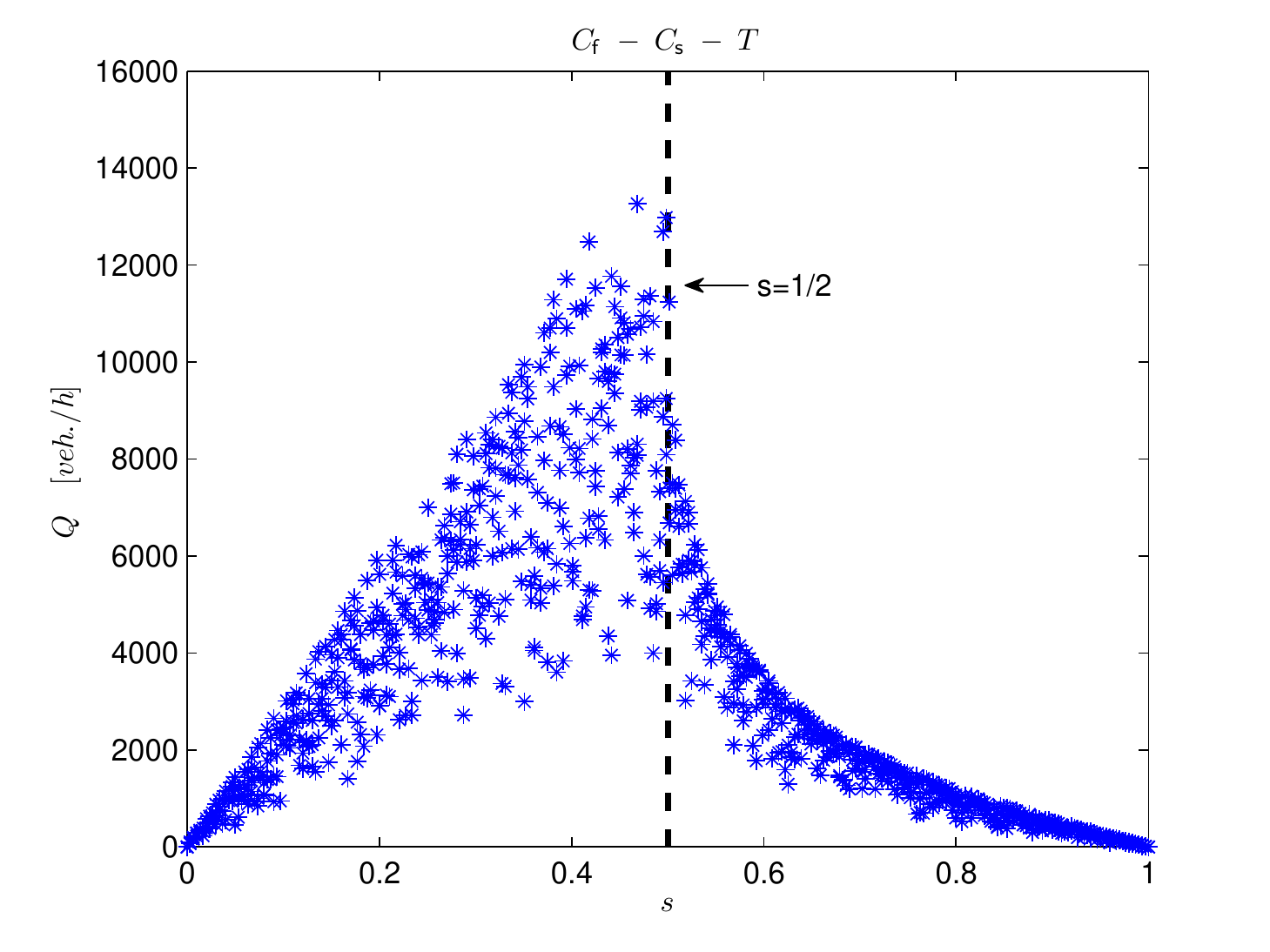}
\includegraphics[width=0.45\textwidth]{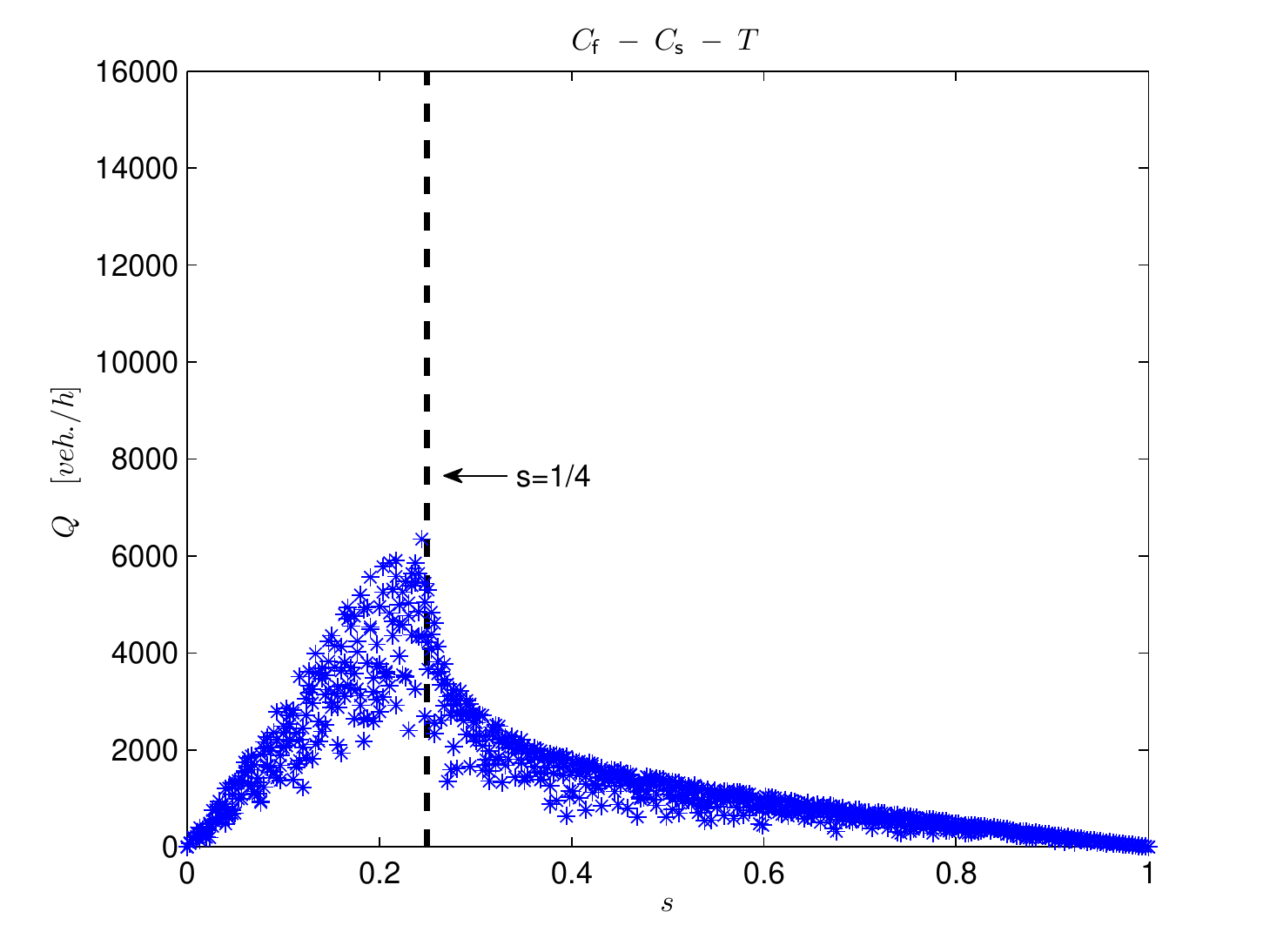}
\caption{Diagrams of the flux vs. the fraction of occupied space for the test case $\Cf$-$\Cs$-$T$. The probability $P$ is taken as in~\eqref{eq:gamma-law} with $\gamma=1$ (left) and $\gamma=\frac12$ (right).\label{fig:s-diagrams}}
\end{figure}
In both panels we consider the law given in~\eqref{eq:gamma-law}, but we choose two different values of $\gamma\in(0,1]$: more precisely, $\gamma=1$ and $\gamma=\frac12$ in the left and right plot, respectively. Note that the critical value of the fraction of occupied space decreases from $s=\frac12$ to $s=\frac14$. In fact, with~\eqref{eq:gamma-law}
\[
	P<\frac12 \quad \Longleftrightarrow \quad s>\left(\frac12\right)^{\frac{1}{\gamma}}
\]
and this means that high values of $\gamma$ increase the value of $s$ at which the transition between the two regimes of traffic occurs. Thus, $\gamma>1$ is not a good choice because it does not reflect the structure of the phase transition usually observed in the experimental diagrams.

\item[Congested phase of traffic] This traffic regime occurs at high densities, that is when the fraction of occupied space $s$ exceeds the critical value. The congested phase is characterized by frequent interactions among vehicles which are forced to slow down, i.e. they cannot travel at the same high speeds as in free road conditions, because traffic becomes more and more jammed. As a consequence, the flux decreases as the fraction of occupied space increases and the experimental diagrams exhibit a large scattering of the flux values. In the congested phase, therefore, the flux can hardly be approximated by a single-valued function of the density.
\begin{figure}[t!]
\centering
\includegraphics[width=0.45\textwidth]{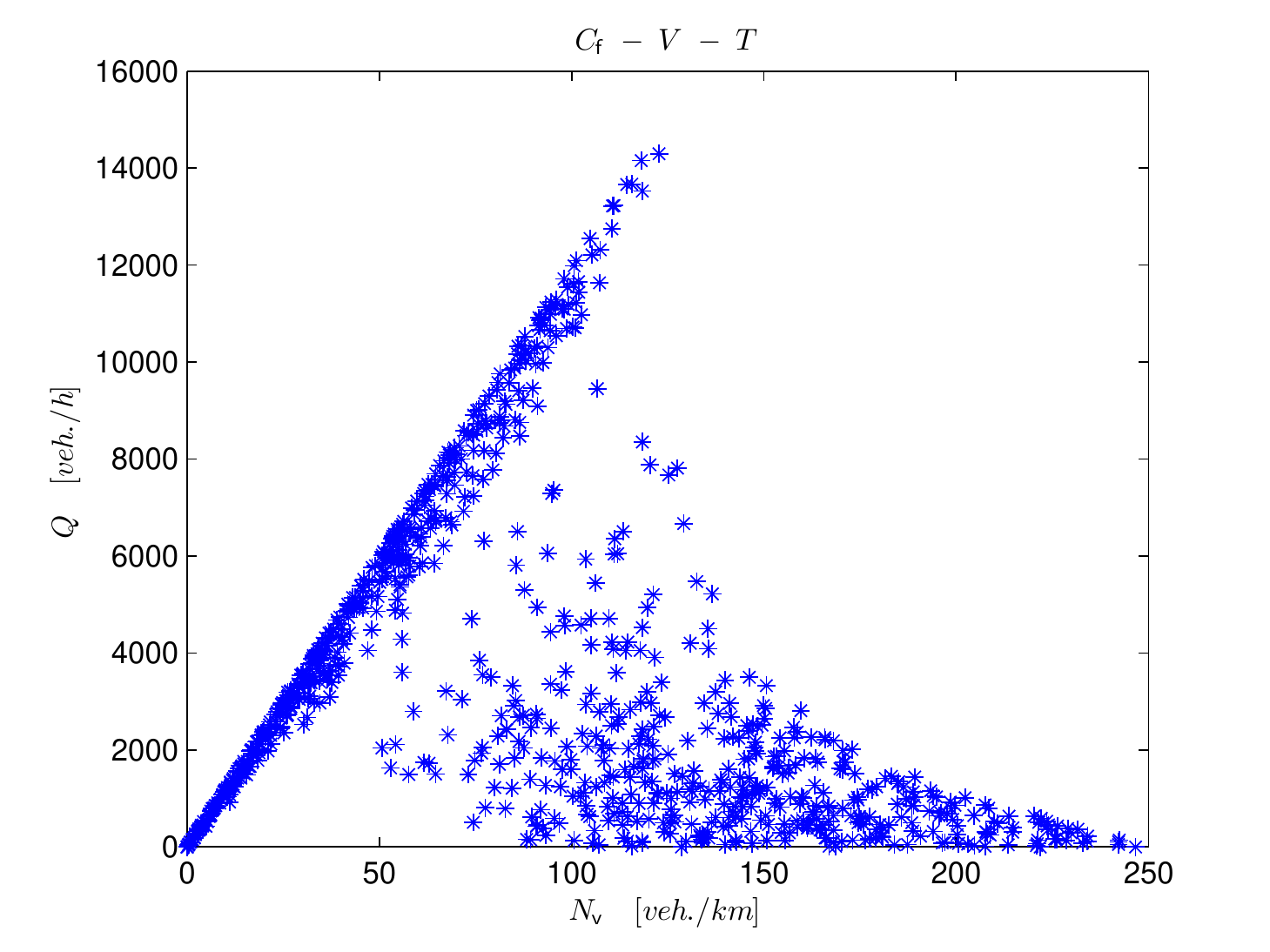}
\includegraphics[width=0.45\textwidth]{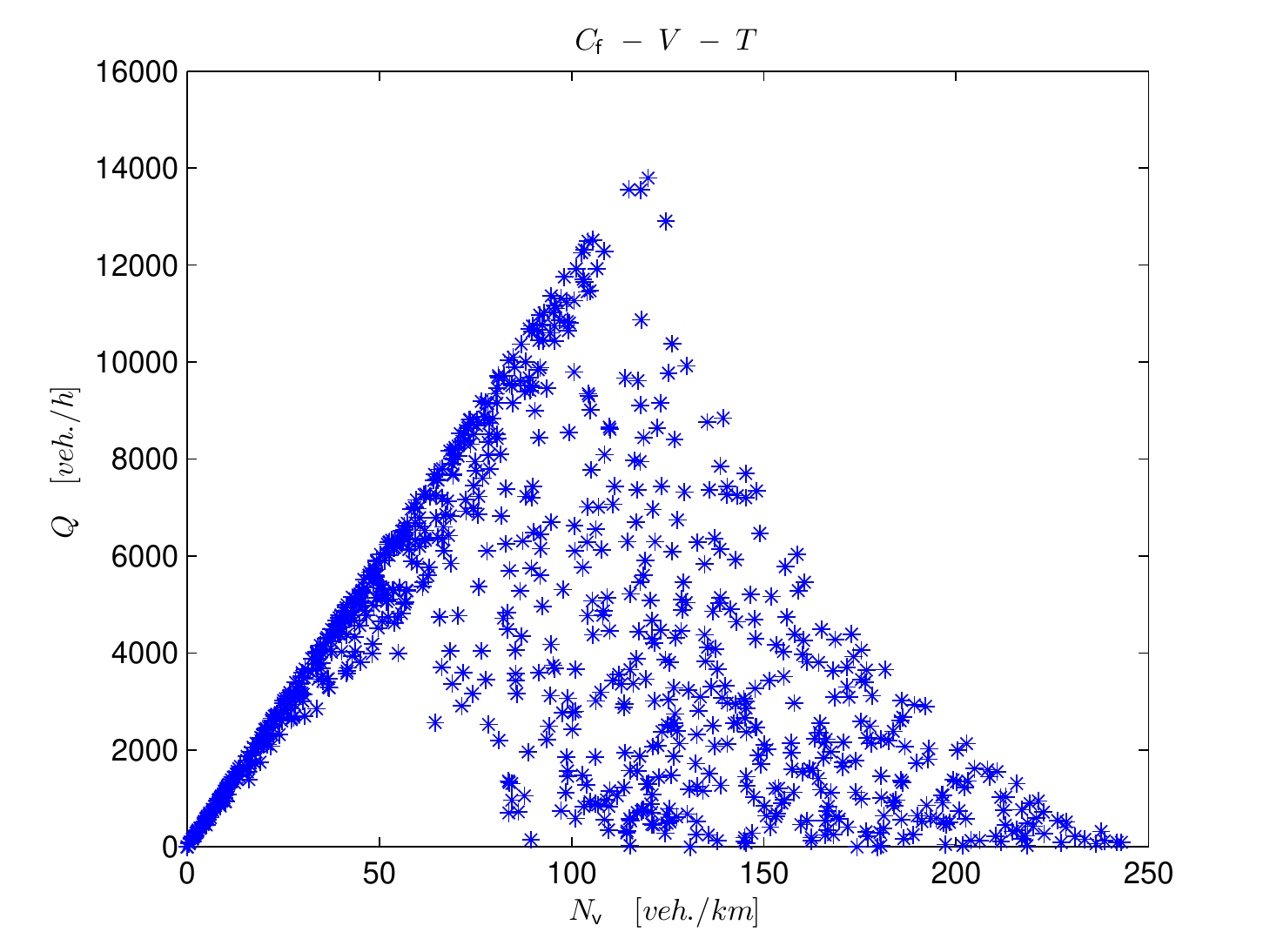}
\includegraphics[width=0.45\textwidth]{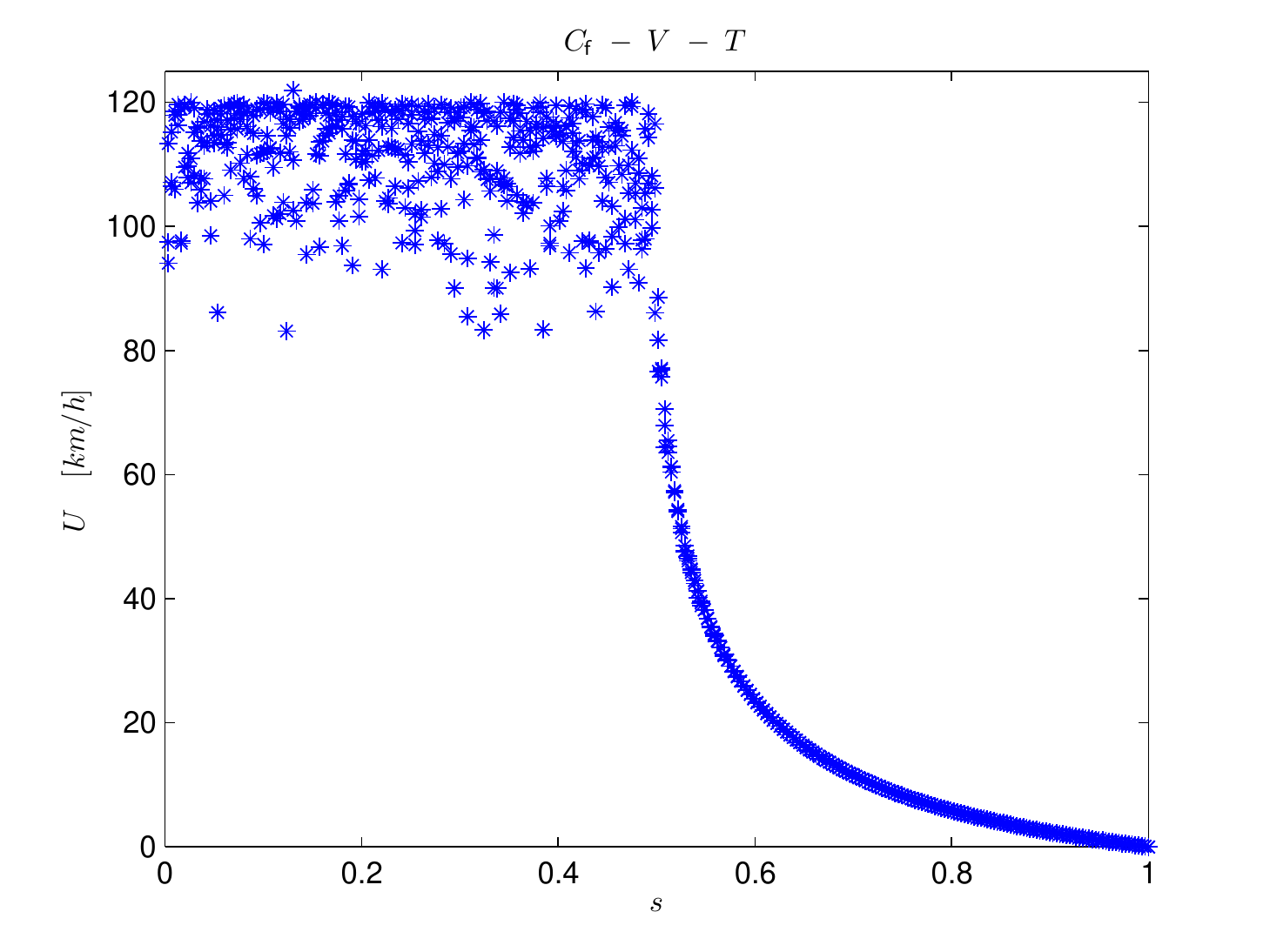}
\includegraphics[width=0.45\textwidth]{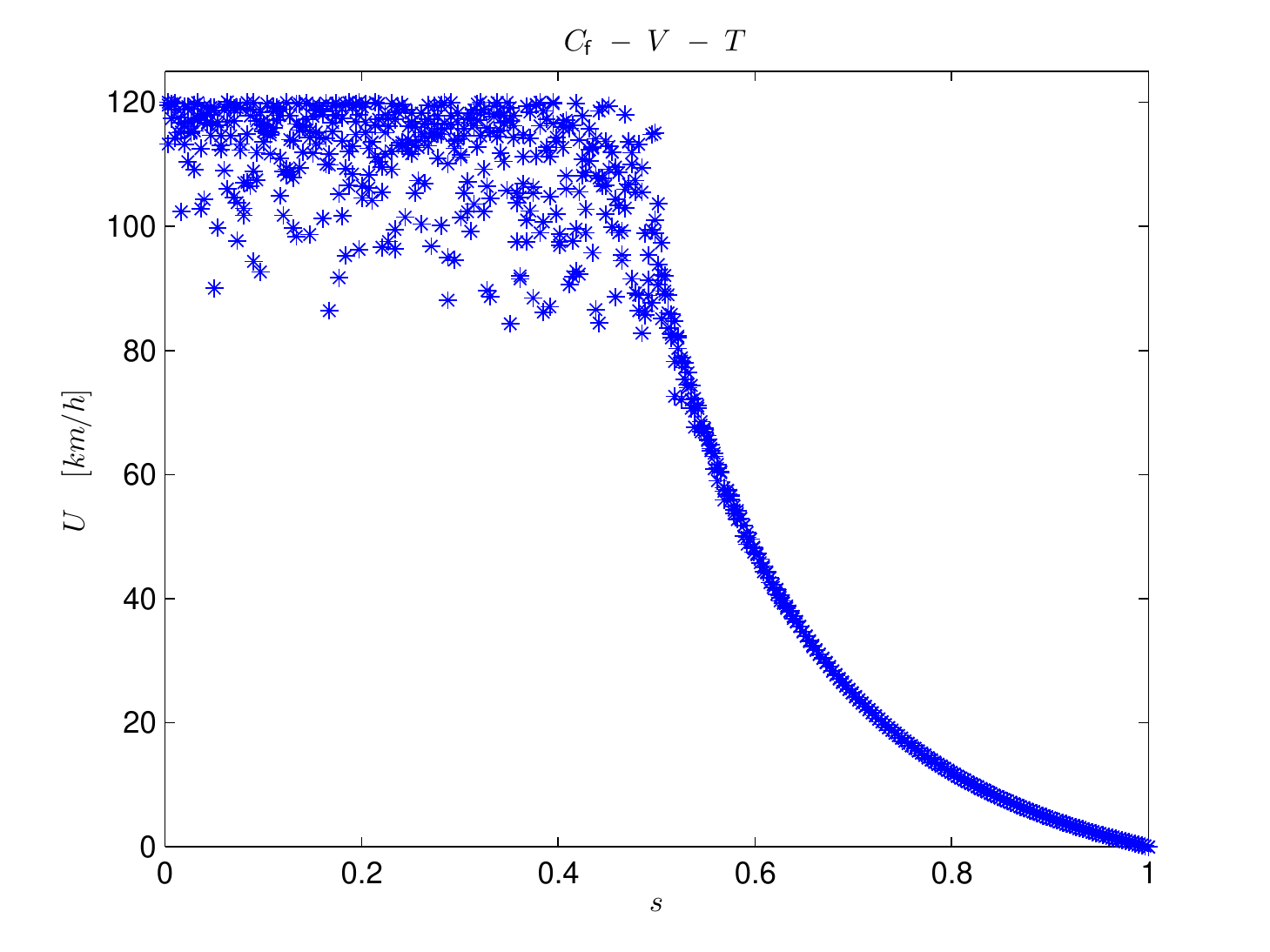}
\caption{Top: flux-density diagrams. Bottom: diagrams of the speed vs. the fraction of occupied space. We have considered the three populations $\Cf$-$V$-$T$. In the left panels the probability of changing velocity $P$ is taken as in~\eqref{eq:gamma-law} with $\gamma=1$, while in the right panels the probability $P$ is as in~\eqref{eq:piecewise-law} with $\scr=\frac12$ and $\mu=-\frac18$.\label{fig:diagrams}}
\end{figure}
The scattered behavior is automatically reproduced by our multi-population kinetic model, see the left panels of Figure~\ref{fig:diagrams} in which we plot the total flux- and speed-density diagrams for the test case $\Cf$-$V$-$T$ with the probability law~\eqref{eq:gamma-law}. Recalling that the diagrams of traffic are obtained by means of the equilibrium distributions, the multi-population model naturally accounts for the dispersion of the flux values in the congested phase because the equilibrium solutions do not only depend on the fraction of occupied space $s$ but also on the single densities of the vehicles. Therefore, the explanation for the multivalued behavior is based on the consideration that the flow along a road is strongly influenced by the composition of the traffic mixture. In particular, this aspect is evident at high densities because the different mechanical characteristics (the typical length) of the vehicles on the road become a key factor to adjust the speed in congested conditions. Conversely, if we focus on the diagrams of the speed vs. the fraction of occupied space at the bottom of Figure~\ref{fig:diagrams}, we deduce that in free flow conditions the macroscopic speed is influenced by the fact that fast vehicles slow down as a consequence of their interactions with slower vehicles. In contrast, at high values of $s$, all types of vehicles are forced to slow down, reaching the same macroscopic speed. These remarks reflect the daily experience of driving on highways, in particular the fact that in congested flow all vehicles tend to travel at the same speed, which steadily decreases as the traffic congestion increases.

\item[Capacity drop] The flux-density diagram in the top-left panel of Figure~\ref{fig:diagrams} is very similar to the experimental ones, whose main characteristics are well reproduced. However, as observed also in the single population model~\cite{PgSmTaVg2}, the diagrams seem to be strictly dependent on $T^\p$, which defines uniquely the number of discrete velocities once $\vmp$ and $\Dv$ are given. In fact, the {\em capacity drop}, that is the jump between the maximum flux values in free and congested phases, see e.g.~\cite{ZhangMultiphase}, becomes sharper and sharper when $T^\p$ increases. For instance, with the choice of the physical parameters listed in Table~\ref{tab:physical-parameters}, the number of velocities of fast cars is $n^{\Cf}=4$ and this provides a sharp decrease of the flux of the $\Cf$-class beyond the critical fraction of occupied space. Clearly, this phenomenon influences also the capacity drop of the total diagram obtained with the multi-population model. We try to overcome this drawback acting on the law which defines the dependence of $P$ on $s$. To this end, we introduce a new law relating $P$ to $s$ in order to better fit experimental data. As a matter of fact, the sharp transition is due to over-crowding of the low-speed equilibrium distributions, also for values of $s$ just greater than the critical value. Thus, if $P_\gamma(s)$ is the {\em $\gamma$-law} given in equation~\eqref{eq:gamma-law}, the purpose is now to introduce a new and less simplistic function $P(s)$ such that $P_\gamma(s)<P(s)<\frac12$, $\forall\;s>\scr$, where $\scr$ is the critical value of $s$ obtained with the $\gamma$-law, i.e. $\scr=\left(\frac12\right)^{\frac{1}{\gamma}}$. Thus, when $s$ just exceeds $\scr$, the new probability $P(s)$ provides an under-crowding of the first equilibrium distributions~\eqref{eq:stable-f1}. As a consequence, the maximum flux value of the congested phase increases and the capacity drop abates. Recalling that $P$ is the probability of achieving the maximum speed prescribed by the interaction rules, the desired function $P(s)\in[0,1]$ should satisfy the following properties:
\begin{enumerate}
\item $P(0)=1$: when the road is empty, the probability of accelerating is maximum;
\item $P(1)=0$: in contrast to the previous request, the probability is zero in jammed traffic situations;
\item $P(\scr)=\frac12$: we impose that the transition from free to congested phase corresponds to the bifurcation of the equilibrium solutions and it occurs at $s=\scr$. The value of the critical space can be chosen by means of experimental data;
\item $\frac{d}{ds} P_{\gamma \mid_{s=\scr^+}}<\frac{d}{ds} P_{\mid_{s=\scr^+}}:=\mu<0$: since any reasonable $P$ should be a decreasing function of $s$, this is a sufficient condition to verify $P_\gamma(s)<P(s)<\frac12$, $\forall\;s>\scr$. Thus, since $\gamma$ is uniquely determined once the value of the critical space is fixed, we require that $\mu>\frac{d}{ds} P_{\gamma \mid_{s=\scr^+}}=-\gamma\scr^{\gamma-1}=-\gamma\left(\frac12\right)^{1-\frac{1}{\gamma}}$.
\end{enumerate}
\begin{figure}[t!]
\centering
\includegraphics[width=0.45\textwidth]{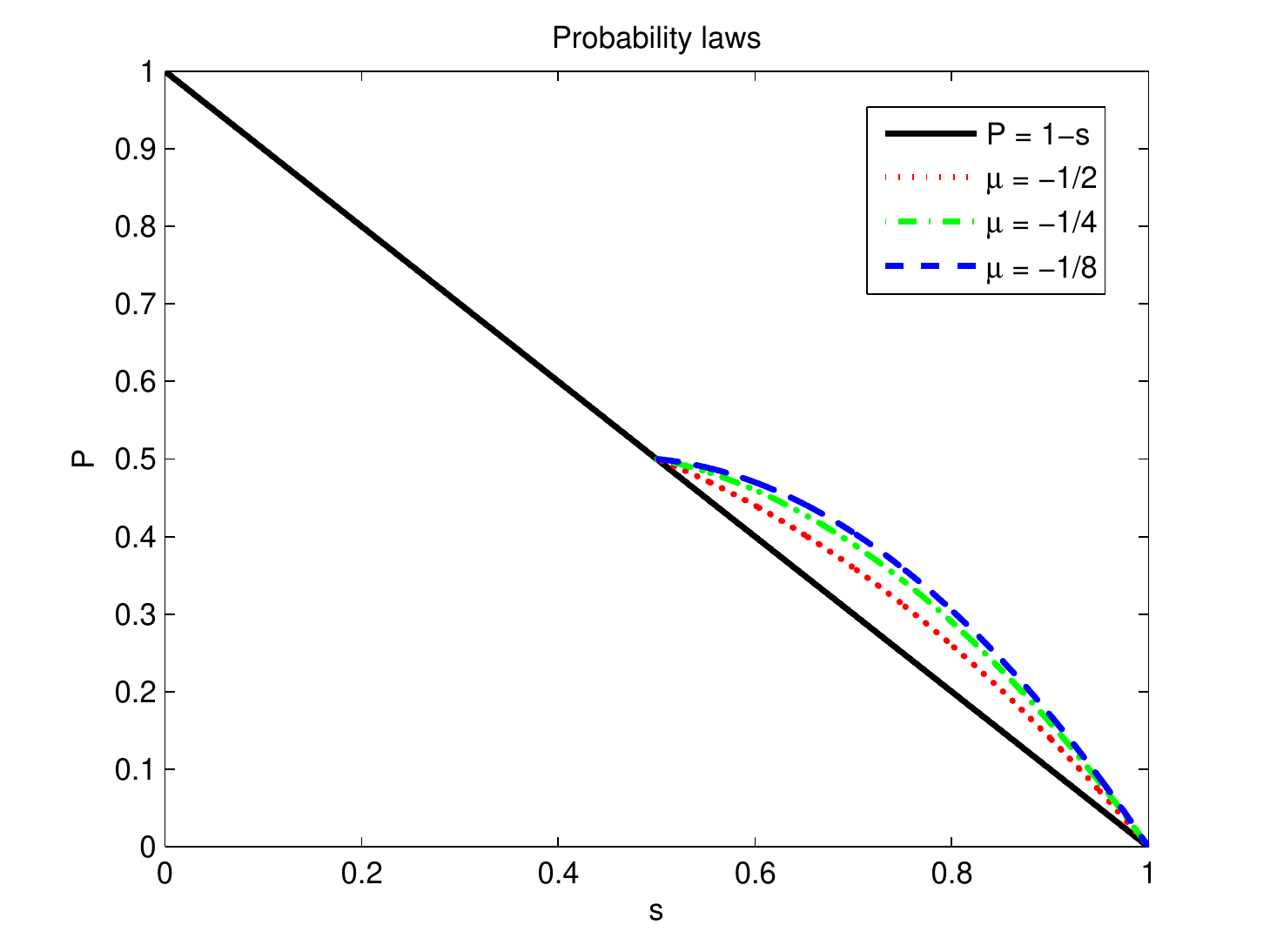}
\includegraphics[width=0.45\textwidth]{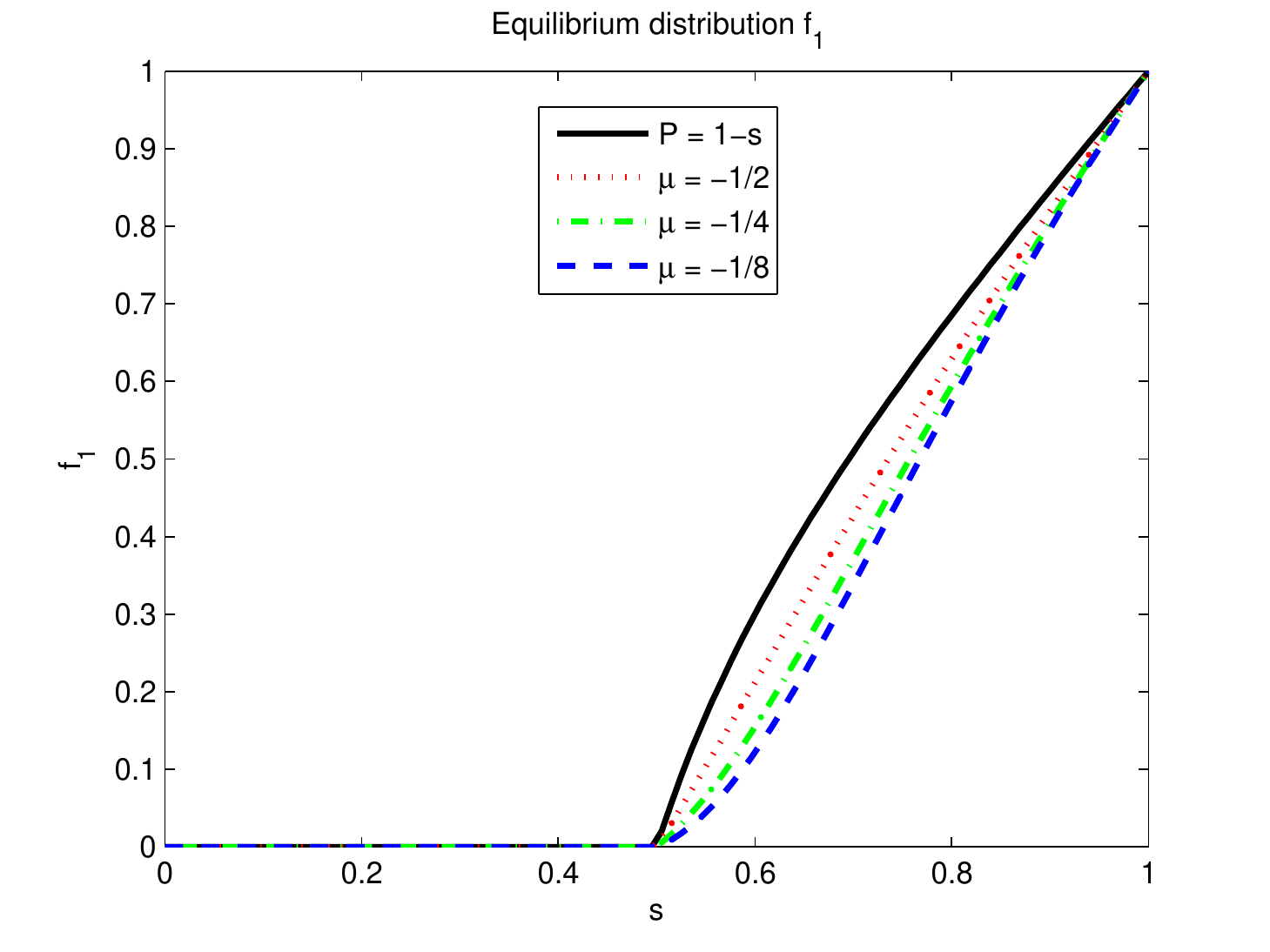}
\caption{Left: the probability law~\eqref{eq:gamma-law} (black solid line) with $\scr=\frac12$, $\gamma=1$, and the probability law~\eqref{eq:piecewise-law}, which differs from the $\gamma$-law only for $s>\scr=\frac12$, obtained with three different values of the slope $\mu$ for $s=\scr$. Right: the asymptotic distribution $f_1$ (lower speed) obtained with the probability laws considered in the left panel.\label{fig:Plaws}}
\end{figure}
In order to satisfy the four properties above and since the equilibrium solutions do not depend on the analytical expression of $P(s)$ (for $s\in[0,\scr]$) we consider $P$ as a piecewise function of $s$, so that
\begin{equation}\label{eq:piecewise-law}
P(s)=P_1(s) \chi_{[0,\scr]}(s) + P_2(s) \chi_{(\scr,1]}(s)
\end{equation}
where $P_1(s)$ is a linear polynomial satisfying the first and the third property, while $P_2(s)$ is a quadratic polynomial satisfying the second, the third and the fourth property. Therefore
\[
	P_1(s)=1-\frac{s}{2\scr}, \quad P_2(s)=as^2+bs+c
\]
the coefficients of $P_2$ being
\[
a=\frac{2\mu(\scr-1)-1}{2(\scr-1)^2}, \quad b=-\frac{\mu(\scr^2-1)-\scr}{(\scr-1)^2}, \quad c=\frac{2\scr\left[\mu(\scr-1)-1\right]+1}{2(\scr-1)^2}.
\]
For simplicity, in the left panel of Figure~\ref{fig:Plaws} we consider the probability laws with $\scr=\frac12$. Thus, $\gamma=1$ and the $\gamma$-law~\eqref{eq:gamma-law} writes as $P=1-s$, while the piecewise probability law~\eqref{eq:piecewise-law} is plotted for different values of the slope $\mu$ computed in $s=\scr=\frac12$ such that $\mu>-1=\frac{d}{ds} P_{\gamma \mid_{s=\scr=\frac12}}$. Notice that the probability values resulting from the piecewise law increase for $\forall\;s>\scr$ when $\mu$ is increased and this provides a little decrease of the first equilibrium distribution. This can be seen it in the right panel of Figure~\ref{fig:Plaws}, in which we plot the asymptotic distribution $f_1$ as a function of $s$ for the case of a single population. Thus, the piecewise law allows one to reduce the sharp capacity drop. For instance, this can be appreciated by comparing the diagrams in the left panels of Figure~\ref{fig:diagrams} obtained with the $\gamma$-law and the diagrams in the right panels of Figure~\ref{fig:diagrams} obtained with the piecewise probability law.
\end{description}

\section{Conclusions and perspectives}
\label{sec:conclusions}

In this paper we have introduced a kinetic model for traffic flow which accounts for the heterogeneous composition of the flow. Thus, we have considered more than one kinetic function describing the distribution of vehicles on the road. The model is based on the Boltzmann-like description of the microscopic interactions which take place among different types of vehicles.

The aim of this work was to refine the construction of the recent multi-population kinetic model~\cite{PgSmTaVg} based on a discrete space of microscopic speeds in order to make it more amenable to a sound physical interpretation and mathematical analysis. In fact, here we have generalized the single-population framework introduced in~\cite{PgSmTaVg2}, to the case of more than one class of vehicles. In particular, we have considered continuous and bounded velocity spaces and we have introduced a fixed parameter $\Dv$ to account for the physical velocity jump produced by a vehicle that increases its speed. The types of vehicles are characterized by few microscopic features, in this case the typical length of a vehicle and the maximum speed. After modeling the collision terms describing the acceleration and the slowing down interactions, we have proved that the model satisfies the indifferentiability principle at all times, which makes it consistent with the original model when the vehicles have the same physical characteristics.

Next, we have discretized the model in order to investigate numerically the asymptotic kinetic distributions. We have shown that they approach a series of delta functions centered on a finite number of velocities. More precisely, these velocities are integer multiples of $\Dv$. Therefore, the asymptotic distributions can be found by solving numerically the discretized model with only {\em few} discrete velocities in the grid. It is worth stressing that the knowledge of the equilibrium distributions is crucial for both the study of average characteristic of traffic, such as the flux- and the speed-density relations, and the derivation of macroscopic equations from the kinetic approach, because the richer closure law provided by the kinetic approach can be used.

We have also studied the analytical properties of the system of ordinary differential equations resulting from the discretization of the continuous-velocity model. We have proved that the solution exists, that it is unique, and that it remains positive in time. In addition to that, we have provided the explicit formulas for the equilibrium distribution functions and the corresponding macroscopic variables of traffic.

\medskip

\section*{Acknowledgment}
This work was partly supported by 
\textquotedblleft National Group for Scientific Computation (GNCS-INDAM)\textquotedblright. Andrea Tosin acknowledges that this work has been written within the activities of the \textquotedblleft National Group for Mathematical Physics (GNFM-INDAM)\textquotedblright.
The authors wish to acknowledge the anonymous Referee for contributing to simplify the proof of Theorem \ref{th:well.posedness}. The authors thank also Prof. Emanuele Casini (University of Insubria) for helpful discussions.

\medskip

\bibliographystyle{plain}
\bibliography{PgSmTaVg-analysis_multipop}

%\newpage
\appendix
\numberwithin{equation}{section}
\section{Matrix elements for the discretized model}
\label{app:discrete-terms}

We compute explicitly the right-hand side of the system~\eqref{eq:discretesys}.
%The bottom-right index is referred to the candidate vehicle distribution (such as e.g. $f_h^\p$), while the top-right one to %the field vehicles distributions (such as e.g. $f^{\p k}$ or $f^{\q k}$).
In the following formulas, the kinetic distribution functions of  candidate and field vehicles are distinguished by the position of the index of the components: bottom-right for candidate vehicles (such as e.g. $f_h^\p$), top-right for field vehicles (such as e.g. $f^{\p k}$ or $f^{\q k}$).

Starting from the self-collision term $\Qpp_j$ and recalling the computations in~\cite{PgSmTaVg2} for the single population model, we obtain:
\begin{subequations} \label{eq:Jpp:allj}
\begin{align}\label{eq:Jppsmallj}
	\frac{1}{\eta^\p}\Qpp_j[\fp,\fp](t)&=
	(1-\frac{P}{2})f^{\p j}f_j^\p+
	Pf_j^\p\sum_{k=1}^{j-1}f^{\p k}
	+(1-P)f_j^\p\sum_{k=j+1}^{\np}f^{\p k}\\
    &+(1-P)f^{\p j}	\sum_{h=j+1}^{\np}f_h^\p-f_j^\p\sum_{k=1}^{\np}f^{\p k},
    \quad%\qquad\qquad\qquad\qquad\qquad
    \text{for $j=1,\dots,r$} \nonumber
%\end{align}
%\begin{align}
\\ \label{eq:Jppallj}
	\frac{1}{\eta^\p}\Qpp_j[\fp,\fp](t)&=
	\frac{P}{2}f^{\p j-r}f_{j-r}^\p
	+(1-\frac{P}{2})f^{\p j} f_j^\p
	+Pf_{j-r}^\p\sum_{k=j-r+1}^{\np}f^{\p k}
	+Pf_j^\p\sum_{k=1}^{j-1}f^{\p k}\\
	&+(1-P)f_j^\p\sum_{k=j+1}^{\np}f^{\p k}
	+(1-P)f^{\p j}\sum_{h=j+1}^{\np}f_h^\p-f_j^\p\sum_{k=1}^{\np}f^{\p k},\quad \text{for $j=r+1,\dots,\np-1$}\nonumber
%\end{align}
%\begin{align}
\\ \label{eq:Jpplastj}
	\frac{1}{\eta^\p}\Qpp_{\np}[\fp,\fp](t)&=
	P \sum_{h=\np-r}^{\np-1} f_h^\p \left[ \frac12 f^{\p h} + \sum_{k=h+1}^{\np}f^{\p k}\right]
	+f^{\p \np}f_{\np}^\p
	+Pf_{\np}^\p\sum_{k=1}^{\np-1}f^{\p k}-f_{\np}^\p\sum_{k=1}^{\np}f^{\p k}.
\end{align}
\end{subequations}

For the cross-collision terms, we distinguish the two cases described in~\eqref{eq:cross-gain1} and~\eqref{eq:cross-gain2}. Therefore, let $\Vp\subset\Vq$, we obtain:
\begin{subequations} \label{eq:Jpq:allj2}
\begin{align}\label{eq:Jpqsmallj2}
	\frac{1}{\eta^{\p\q}}\Qpq_j[\fp,\fq](t)&=
	(1-\frac{P}{2})f^{\q j}f_j^\p+
	P f_j^\p \sum_{k=1}^{j-1}f^{\q k}
	+(1-P)f_j^\p\sum_{k=j+1}^{\nq}f^{\q k}\\
    &+(1-P)f^{\q j}	\sum_{h=j+1}^{\np}f_h^\p-f_j^\p\sum_{k=1}^{\nq}f^{\q k},
    \quad%\qquad\qquad\qquad\qquad\qquad
    \text{for $j=1,\dots,r$} \nonumber
%\end{align}
%\begin{align}
\\ \label{eq:Jpqallj2}
	\frac{1}{\eta^{\p\q}}\Qpq_j[\fp,\fq](t)&=
	\frac{P}{2}f^{\q j-r}f_{j-r}^\p
	+(1-\frac{P}{2})f^{\q j} f_j^\p
	+Pf_{j-r}^\p\sum_{k=j-r+1}^{\nq}f^{\q k}
	+Pf_j^\p\sum_{k=1}^{j-1}f^{\q k}\\
	&+(1-P)f_j^\p\sum_{k=j+1}^{\nq}f^{\q k}
	+(1-P)f^{\q j}\sum_{h=j+1}^{\np}f_h^\p-f_j^\p\sum_{k=1}^{\nq}f^{\q k}, \quad \text{for $j=r+1,\dots,\np-1$}\nonumber
\\ \label{eq:Jpqlastj2}
	\frac{1}{\eta^{\p\q}}\Qpq_{\np}[\fp,\fq](t)&=
	P \sum_{h=\np-r}^{\np-1} f_h^\p \left[ \frac12 f^{\q h} + \sum_{k=h+1}^{\nq}f^{\q k}\right]
	+f^{\p}_{\np}\left[ f^{\q \np} + \sum_{k=\np+1}^{\nq} f^{\q k} \right]\\
	&+Pf_{\np}^\p\sum_{k=1}^{\np-1}f^{\q k}-f_{\np}^\p\sum_{k=1}^{\nq}f^{\q k}.\nonumber
\end{align}
\end{subequations}

Finally, let $\Vp\supset\Vq$. If $\vmp-\vmq>\Dv$, we obtain:
\begin{subequations} \label{eq:Jpq:allj1}
\begin{align}\label{eq:Jpqsmallj1}
	\frac{1}{\eta^{\p\q}}\Qpq_j[\fp,\fq](t)&=
	(1-\frac{P}{2})f^{\q j}f_j^\p+
	P f_j^\p \sum_{k=1}^{j-1}f^{\q k}
	+(1-P)f_j^\p\sum_{k=j+1}^{\nq}f^{\q k}\\
    &+(1-P)f^{\q j}	\sum_{h=j+1}^{\np}f_h^\p-f_j^\p\sum_{k=1}^{\nq}f^{\q k},
    \quad%\qquad\qquad\qquad\qquad\qquad
    \text{for $j=1,\dots,r$} \nonumber
%\end{align}
%\begin{align}
\\ \label{eq:Jpqmidsmallj1}
	\frac{1}{\eta^{\p\q}}\Qpq_j[\fp,\fq](t)&=
	\frac{P}{2}f^{\q j-r}f_{j-r}^\p
	+(1-\frac{P}{2})f^{\q j} f_j^\p
	+Pf_{j-r}^\p\sum_{k=j-r+1}^{\nq}f^{\q k}
	+Pf_j^\p\sum_{k=1}^{j-1}f^{\q k}\\
	&+(1-P)f_j^\p\sum_{k=j+1}^{\nq}f^{\q k}
	+(1-P)f^{\q j}\sum_{h=j+1}^{\np}f_h^\p-f_j^\p\sum_{k=1}^{\nq}f^{\q k}, \quad \text{for $j=r+1,\dots,\nq-1$}\nonumber
\\ \label{eq:Jpqnq1}
	\frac{1}{\eta^{\p\q}}\Qpq_{\nq}[\fp,\fq](t)&=
	\frac{P}{2}f^{\q \nq-r}f_{\nq-r}^{\p}
	+(1-\frac{P}{4})f^{\q \nq} f_{\nq}^{\p}
	+Pf_{\nq-r}^{\p}\sum_{k=\nq-r+1}^{\nq}f^{\q k}
	+Pf_{\nq}^{\p}\sum_{k=1}^{\nq-1}f^{\q k}\\
	&+(1-P)f^{\q \nq}\sum_{h=\nq+1}^{\np}f_h^{\p}-f_{\nq}^{\p}\sum_{k=1}^{\nq}f^{\q k},\nonumber
\\ \label{eq:Jpqmidlastj1}
	\frac{1}{\eta^{\p\q}}\Qpq_j[\fp,\fq](t)&=
	\frac{P}{2}f^{\q j-r}f_{j-r}^{\p}
	+Pf_{j-r}^{\p}\sum_{k=j-r+1}^{\nq}f^{\q k}
	+Pf_j^{\p}\sum_{k=1}^{\nq}f^{\q k}-f_j^{\p}\sum_{k=1}^{\nq}f^{\q k}, \quad \text{for $j=\nq+1,\dots,\nq+r-1$}
\\ \label{eq:Jpqnqr1}
	\frac{1}{\eta^{\p\q}}\Qpq_{\nq+r}[\fp,\fq](t)&=
	\frac{P}{4}f^{\q \nq}f_{\nq}^{\p}
	+Pf_{\nq+r}^{\p}\sum_{k=1}^{\nq}f^{\q k}-f_{\nq+r}^{\p}\sum_{k=1}^{\nq}f^{\q k},
%\end{align}
%\begin{align}
\\ \label{eq:Jpqlastj1}
	\frac{1}{\eta^{\p\q}}\Qpq_j[\fp,\fq](t)&=
	Pf_j^\p\sum_{k=1}^{\nq}f^{\q k}-f_j^\p\sum_{k=1}^{\nq}f^{\q k},  \quad \text{for $j=\nq+r+1,\dots,\np$}.
\end{align}
\end{subequations}
On the other hand, if $\vmp-\vmq=\Dv$ then $\nq+r=\np$ and the cases~\eqref{eq:Jpqlastj1} do not appear.

\section{Equilibrium solutions of the discretized model}
\label{app:equilibria}

The equilibrium values of the ODE system~\eqref{eq:explicit-discretesys} are explicitly computed in Theorem~\ref{th:equilibria} for the $j=1,\dots,2r$ cases, which are typical. The other solutions can be easily obtained as follows.

For $j=lr+1$, $l=2,\dots,T^{2}-1$, the equation for $\fp_j$ is again computed by using~\eqref{eq:Jppallj}-\eqref{eq:Jpqmidsmallj1} if $\p=1$ and in~\eqref{eq:Jppallj}-\eqref{eq:Jpqallj2} if $\p=2$.
%Thus, we obtain
%\begin{gather*}
%	\left(\frac{3P-2}{2}\right)\left(\fp_{lr+1}\right)^2 + \left[ (3P-2)\sum_{k=0}^{l-1} \fp_{kr+1} + (2P-1) \sum_{k=1}^{l-1} \fq_{kr+1} + (1-2P)\rho^\p - P\rho^\q\right] \fp_{lr+1}\\
%	+(1-P)\fq_{lr+1}\left(\rho^\p-\sum_{k=1}^{l-1}\fp_{kr+1}\right)\\+P\fp_{(l-1)r+1}\left[-\frac12 \left(\fp_{(l-1)r+1}+\fq_{(l-1)r+1}\right)+\rho^\p+\rho^\q-\sum_{k=0}^{l-2}\left(\fp_{kr+1}+\fq_{kr+1}\right)\right]=0,\,\forall\;\p
%\end{gather*}
Let $F_j=\fu_j+\fd_j$, then summing the evolution equation of both populations we obtain
\begin{gather*}
	\left(\frac{3P-2}{2}\right) F_{lr+1}^2 + \left[ (3P-2)\sum_{k=0}^{l-1} F_{kr+1} + (1-2P) \left(\ru+\rd\right)\right] F_{lr+1} \\
	+ P F_{(l-1)r+1} \left[-\frac12 F_{(l-1)r+1}+\rho^\p+\rho^\q-\sum_{k=0}^{l-2} F_{kr+1}\right]=0.
\end{gather*}
If $P\geq\frac12$, then the above equation becomes identical to~\eqref{eq:f1} and the stable root is $F_{lr+1}=0$. If instead $P<\frac12$, then the stable solution is
\[
	F_{lr+1}=\frac{\ds{-(3P-2)\sum_{k=0}^{l-1} F_{kr+1} - (1-2P) \left(\ru+\rd\right) - \sqrt{\Delta_{lr+1}}}}{3P-2}
\]
where
\begin{gather*}
	\Delta_{lr+1}=\left[(3P-2)\sum_{k=0}^{l-1} F_{kr+1} + (1-2P) \left(\ru+\rd\right)\right]^2\\
	-2P(3P-2) F_{(l-1)r+1}\left[-\frac12 F_{(l-1)r+1} + \ru+\rd-\sum_{k=0}^{l-2} F_{kr+1}\right],
\end{gather*}
which is positive because it is a sum of two positive terms, provided $P<\frac12$. The equilibrium solutions of population $\p$ can be again deduced by assuming that they are functions only of the quantities depending on $\p$, so that
\[
	\left(\f_r\right)^\p_{lr+1}=\begin{cases} 0 & P\geq\frac12\\
					\ds{\frac{\ds{-(3P-2)\sum_{k=0}^{l-1} \fp_{kr+1} - (1-2P) \rho^\p - \sqrt{\Delta_{lr+1}^\p}}}{3P-2}} & P<\frac12 \end{cases}, \quad l=0,\dots,T^{2}-1, \forall\;\p
\]
where
\[
	\Delta_{lr+1}^\p=\left[(3P-2)\sum_{k=0}^{l-1} \fp_{kr+1} + (1-2P) \rho^\p\right]^2
	-2P(3P-2)\fp_{(l-1)r+1}\left[-\frac12 \fp_{(l-1)r+1}+\rho^\p-\sum_{k=0}^{l-2}\fp_{kr+1}\right].
\]
Finally, if $lr+2 \leq j \leq (l+1)r$, $l=2,\dots,T^{2}-1$, then the equilibria are $\left(\f_r\right)^\p_j \equiv 0$, $\forall\;P\in[0,1]$ and $\forall\;\p$.

Note that for $j=\nd$, i.e. for the maximum speed of population $\p=2$, the equilibrium value related to the population $\p=2$ can be found by using mass conservation, so that
\begin{equation}\label{eq:stable-f2n2}
	\left(\f_r\right)^{2}_{\nd} = \rd - \sum_{k=1}^{\nd-1} \fd_k.
\end{equation}
Instead, for the population $\p=1$, the equilibrium equation is obtained by means of~\eqref{eq:Jppallj}-\eqref{eq:Jpqnq1}. If $P<\frac12$,  then $\fp_j\neq 0$, if $j=kr+1$, $k=0,\dots,T^{2}-1$ and $\forall\;\p$. Thus, using the last equilibrium value~\eqref{eq:stable-f2n2} for $\p=2$, we find
\begin{gather*}
	\left(\frac{3P-2}{2}\right)\left(\fu_{\nd}\right)^2+\left[(3P-2)\sum_{k=0}^{T^{2}-1} \fu_{kr+1} + \frac{P}{4} \sum_{k=0}^{T^{2}-1} \fd_{kr+1} + (1-2P)\ru +\left(\frac34 P - 1\right)\rd\right] \fu_{\nd} \\
	+(1-P)\left(\ru-\sum_{k=0}^{T^{2}-1} \fu_{kr+1}\right)\left(\rd-\sum_{k=0}^{T^{2}-1} \fd_{kr+1}\right)\\
	+ P\fu_{\nd-r}\left[-\frac12\left(\fu_{\nd-r}+\fd_{\nd-r}\right)+\ru+\rd-\sum_{k=0}^{T^{2}-2} \left(\fu_{kr+1}+\fd_{kr+1}\right)\right] %-\sum_{k=1}^{\nd-r-1} \left(\fu_k+\fd_k\right)\right]
\end{gather*}
whose stable root is
\begin{equation}\label{eq:stable-f1n2}
	\left(\f_r\right)^{1}_{\nd}=\frac{\ds{-\left[(3P-2)\sum_{k=0}^{T^{2}-1} \fu_{kr+1} + \frac{P}{4} \sum_{k=0}^{T^{2}-1} \fd_{kr+1} + (1-2P)\ru +\left(\frac34 P - 1\right)\rd\right]-\sqrt{\Delta_{\nd}}}}{3P-2}
\end{equation}
where the discriminant
\begin{gather}
	\Delta_{\nd}=\left[(3P-2)\sum_{k=0}^{T^{2}-1} \fu_{kr+1} + \frac{P}{4} \sum_{k=0}^{T^{2}-1} \fd_{kr+1} + (1-2P)\ru +\left(\frac34 P - 1\right)\rd\right]^2 \nonumber \\
	-2(3P-2)(1-P)\left(\ru-\sum_{k=0}^{T^{2}-1} \fu_{kr+1}\right)\left(\rd-\sum_{k=0}^{T^{2}-1} \fd_{kr+1}\right) \label{eq:deltan2} \\
		-2(3P-2) P\fu_{\nd-r}\left[-\frac12\left(\fu_{\nd-r}+\fd_{\nd-r}\right)+\ru+\rd-\sum_{k=0}^{T^{2}-2} \left(\fu_{kr+1}+\fd_{kr+1}\right)\right] \nonumber
\end{gather}
is positive provided $P<\frac12$. If instead $P\geq\frac12$, then the equilibrium value is again of the form~\eqref{eq:stable-f1n2} with the discriminant~\eqref{eq:deltan2}, but it is obtained by taking $\fp_j=0$, $\forall\;j<\nd$, $\forall\;\p$. It can be proved that $\Delta_{\nd}$ is positive also provided $P\geq \frac12$.

Let $j=\nd+1$, now $v_j\in\mathcal{V}^{1}$ but $v_j\notin\mathcal{V}^{2}$. The equilibrium equation for $\p=1$ is computed by using~\eqref{eq:Jppallj}-\eqref{eq:Jpqmidlastj1}. We obtain
\begin{equation}\label{eq:fn2+1}
	\left(\frac{3P-2}{2}\right)\left(\fu_{\nd+1}\right)^2+\left[(3P-2)\sum_{l=0}^{T^{2}} \fu_{lr+1} + (1-2P)\ru + (P-1)\rd\right] \fu_{\nd+1} = 0
\end{equation}
and $\fu_{\nd+1}= 0$ results the stable solution $\forall\;P\in[0,1]$. This consideration holds for each $\fu_j$, $\nd+1 < j \leq \nd+r-1$ because the equation resulting from $\frac{d}{dt} \fu_j=0$ is identical to~\eqref{eq:fn2+1}. Thus, $\left(\f_r\right)^{1}_j\equiv 0$, $\nd+1 \leq j \leq \nd+r-1$.

Now, let $j=\nd+r$, then in order to compute $\frac{d}{dt}\fu_{\nd+r}=0$ we use the equations~\eqref{eq:Jpplastj}-\eqref{eq:Jpqnqr1} if $\vm^{1}-\vm^{2}=m^{1,2}\Dv$, with $m^{1,2}=1$. Notice that in this case $n^{2}+r=n^{1}$, thus we can use mass conservation to find $\left(\f_r\right)^{1}_{n^{2}+r}$. Instead, if $m^{1,2}>1$ we may use~\eqref{eq:Jppallj}-\eqref{eq:Jpqnqr1}, and since $\fu_j\equiv 0$, if $j\neq kr+1$, $k=0,\dots,T^{2}$, we find
\begin{gather}\label{eq:f1n2+r}
	\left(\frac{3P-2}{2}\right)\left(\fu_{\nd+r}\right)^2 + \left[ (3P-2)\sum_{k=0}^{T^{2}} \fu_{kr+1} + (1-2P)\ru + (P-1)\rd\right] \fu_{\nd+r} \\
	P \fu_{\nd} \left[-\frac12 \fu_{\nd} + \frac14 \fd_{\nd} + \ru - \sum_{k=0}^{T^{2}-1} \fu_{kr+1}\right]=0. \nonumber
\end{gather}
If $P<\frac12$ the stable solution of the above equation is
\begin{equation}\label{eq:stable-f1n2+r}
	\left(\f_r\right)^{1}_{\nd+r}=\frac{\ds{-\left[ (3P-2)\sum_{k=0}^{T^{2}} \fu_{kr+1} + (1-2P)\ru + (P-1)\rd\right]-\sqrt{\Delta_{\nd+r}}}}{3P-2}
\end{equation}
where the discriminant
\begin{gather}\label{eq:deltan2+r}
	\Delta_{\nd+r}=\left[ (3P-2)\sum_{k=0}^{T^{2}} \fu_{kr+1} + (1-2P)\ru + (P-1)\rd\right]^2\\
	-2P(3P-2) \fu_{\nd} \left[-\frac12 \fu_{\nd} + \frac14 \fd_{\nd} + \ru - \sum_{k=0}^{T^{2}-1} \fu_{kr+1}\right] \nonumber
\end{gather}
is positive provided $P<\frac12$. If instead $P\geq\frac12$ then the stable root of~\eqref{eq:f1n2+r} is again of the form~\eqref{eq:stable-f1n2+r} with the discriminant~\eqref{eq:deltan2+r}, but it is obtained by taking $\fu_j=0$, $\forall\;j\ne\nd$, and it can be proved that $\Delta_{\nd+r}$ is positive provided $P\geq \frac12$.

Let $\nd+r+1\leq j \leq n^{1}-1$, then the equilibrium equation resulting from $\frac{d}{dt}\fu_j=0$ is computed by using~\eqref{eq:Jppallj}-~\eqref{eq:Jpqlastj1}. For any $j=\nd+lr$, $l=2,\dots,m^{1,2}-1$, where $m^{1,2}$ is such that $\vm^{1}-\vm^{2}=m^{1,2}\Dv$, we obtain
\begin{gather*}
	\left(\frac{3P-2}{2}\right)\left(\fu_j\right)^2 + \left[(3P-2)\sum_{k=0}^{T^{2}+l-1} \fu_{kr+1} + (1-2P)\ru + (P-1)\rd\right] \fu_j\\
	+P\fu_{j-r}\left[-\frac12 \fu_{j-r} + \ru - \sum_{k=0}^{T^{2}+l-2} \fu_{kr+1} \right]=0
\end{gather*}
whose stable solution is
\[
	\left(\f_r\right)^{1}_j=\frac{\ds{-\left[(3P-2)\sum_{k=0}^{T^{2}+l-1} \fu_{kr+1} + (1-2P)\ru + (P-1)\rd\right]-\sqrt{\Delta_j}}}{3P-2}
\]
for all values of $P\in[0,1]$ and where the discriminant
\begin{gather*}
	\Delta_j=\left[(3P-2)\sum_{k=0}^{T^{2}+l-1} \fu_{kr+1} + (1-2P)\ru + (P-1)\rd\right]^2\\
	-2P(3P-2)\fu_{j-r}\left[-\frac12 \fu_{j-r} + \ru - \sum_{k=0}^{T^{2}+l-2} \fu_{kr+1} \right]
\end{gather*}
is positive $\forall\;P\in[0,1]$. While, if $\nd+lr+1 \leq j \leq \nd+(l+1)r-1$, $l=2,\dots,m^{1,2}-1$, the equilibrium equation for $\fu_j$ is
\[
	\left(\frac{3P-2}{2}\right)\left(\fu_j\right)^2 + \left[(3P-2)\sum_{k=0}^{T^{2}+l-1} \fu_{kr+1} + (1-2P)\ru + (P-1)\rd\right] \fu_j=0
\]
whose stable solution results $\left(\f_r\right)^{1}_j=0$, $\forall\;P\in[0,1]$.

Finally, the last equilibrium value for the population $\p=1$ can be found by mass conservation, so that
\[
	\fu_{n^{1}}=\ru-\sum_{k=1}^{n^{1}-1} \fu_k.
\]

\end{document}